\documentclass[12pt]{amsart}
\usepackage[dvips]{graphicx}% for dvi viewer% Izumi

\usepackage{amssymb}

\usepackage{enumerate}
\usepackage{graphicx}

\usepackage{epic,eepic,multirow}% Izumi

\makeatletter
\@namedef{subjclassname@2010}{%
  \textup{2010} Mathematics Subject Classification}
\makeatother

\newtheorem{thm}{Theorem}[section]
\newtheorem{cor}[thm]{Corollary}
\newtheorem{lem}[thm]{Lemma}
\newtheorem{prop}[thm]{Proposition}

\theoremstyle{definition}
\newtheorem{defin}[thm]{Definition}
\newtheorem{rem}[thm]{Remark}
\newtheorem{exa}[thm]{Example}

\numberwithin{equation}{section}

\numberwithin{equation}{section}

\newcommand{\BB}[1]{\mathbb{#1}}
\newcommand{\CAL}[1]{\mathcal{#1}}
\newcommand{\FRAK}[1]{\mathfrak{#1}}
\newcommand{\BS}[1]{\boldsymbol #1}

\newcommand{\BLF}[2]{\left\langle #1\,\|\,#2\right\rangle}

\newcommand{\BLFFF}[2]{\Bigl\langle #1\,\Big\|\,#2\Bigr\rangle}
\newcommand{\COMP}{\raisebox{0.3ex}{\hspace{0.3ex}%
   {$\scriptstyle{\circ}$}\hspace{0.5ex}}}  %%% composition
\newcommand{\DA}{\hbox{\hskip-.4em}\downarrow}

\newcommand{\FRAC}[2]{\frac{#1}{#2}}
\newcommand{\id}{\mbox{\upshape id}}
\newcommand{\Ker}{\mbox{\upshape Ker\,}}

\renewcommand{\mod}{\ \,\mbox{mod}\ }
\newcommand{\OL}[1]{\overline{#1}}
\newcommand{\ord}{\mbox{\upshape ord\,}}
\newcommand{\SLF}[2]{\left\langle #1\, |\, #2\right\rangle}
\newcommand{\SLFF}[2]{\bigl\langle #1\, \big|\, #2\bigr\rangle}
\newcommand{\SLFFF}[2]%
{\Bigl\langle #1\, \Big|\, #2\Bigr\rangle}
\newcommand{\SLFFFF}[2]%
{\biggl\langle #1\, \bigg|\, #2\biggr\rangle}
\newcommand{\Span}{\mbox{\upshape Span}}

\newcommand{\TAY}{\mbox{\sffamily\itshape T\,}}

\newcommand{\UA}{\hbox{\hskip-.4em}\uparrow}
\begin{document}

\baselineskip=17pt

\author[S. Izumi]{Shuzo Izumi}
\address{Research Center for Quantum Computing, 
Kindai University\\ 
Higashi-Osaka 577-8502, Japan}
\email{sizmsizm@gmail.com}

\date{\today}

\title[Spaces of polynomial functions]
{Spaces of polynomial functions of 
bounded degrees on an embedded 
manifold and their duals}

{\scriptsize
\footnote{This work is partly supported by 
JSPS KAKENHI (21540054) 
and partly done during the author's stay at Sydney University 
with the aid of JSPS Bilateral Joint Projects/Seminars (2010). 
}}

\begin{abstract}
Let $\CAL{O}(U)$ denote the algebra of holomorphic functions 
on an 
open subset $U\subset\BB{C}^n$ and 
$Z\subset\CAL{O}(U)$ its finite-dimensional vector subspace. By 
the theory of least space of de~Boor and Ron, there 
exists a projection $\TAY_{\BS{b}}$ from 
the local ring $\CAL{O}_{n,\BS{b}}$ onto 
the space $Z_{\BS{b}}$ of germs of elements of $Z$ at 
$\BS{b}$. At a general point $\BS{b}\in U$ its kernel is an 
ideal and $\TAY_{\BS{b}}$ induces a structure of 
an Artinian algebra 
on $Z_{\BS{b}}$. In particular, it holds at points where $k$-th 
jets of elements of $Z$ form a vector bundle for 
each $k\in\BB{N}$. For an embedded manifold $X\subset\BB{C}^m$, 
we introduce a 
space of higher order tangents following Bos and Calvi. In 
the case of curve, using $\TAY_{\BS{b}}$, we define the Taylor 
projector of order $d$ at a general point $\BS{a}\in X$, 
generalising results of Bos and Calvi. It is a retraction of 
$\CAL{O}_{X,\BS{a}}$ onto the set of the polynomial functions 
on $X_{\BS{a}}$ 
of degree up to $d$. Using the ideal property stated 
above, we show that the transcendency index, defined by the 
author, of the embedding of manifold $X\subset\BB{C}^m$ is not 
very high at a general point of $X$.\\

\end{abstract}
\maketitle

\section{Introduction}
\label{introduction}
The motivation of this paper is applications of least spaces of 
de~Boor and Ron and a generalisation of the theory of 
Bos and Calvi on Taylor projector on a plane algebraic curve. 
In particular, the main problem is to clarify the nature of 
``singularities of an affine embedding of a manifold" 
found by Bos and 
Calvi in the case of plane algebraic curves. 
\\

In early nineties, there was a big progress in the theory of 
multivariate polynomial interpolation 
with arbitrary interpolation nodes 
(see e.g. \cite{REFgs}, \cite{REFinterpolation}). 
Outstanding methods are application of Gr\"{o}bner basis 
(e.g. Marinari-M\"{o}ller-Mora, \cite{REFmmm}), 
use of exponential 
polynomials for evaluation at interpolation points 
(Dyn-Ron \cite{REFdr}, 
de~Boor-Ron \cite{REFbr1}) and the duality between the space of 
interpolating functions and its least space 
(de~Boor-Ron \cite{REFbr1}, \cite{REFbr2}). 
The first method is applied widely treating 
the interpolation theory 
as algebraic geometry of 0-dimensional sub-schemes. 
The second is related to systems of PDEs 
with constant coefficients. 
The author feel that the third method is also applicable 
to general problems beyond interpolation. 
This is a trial for such a direction. 

The essence of the third method above is that, 
for a finite-dimensional vector 
space 
$$
Z_{\BS{b}}\subset\BB{C}\{\BS{t}-\BS{b}\}\quad
(\BS{t}:=(t_1,\dots,t_n),\ \BS{b}:=(b_1,\dots,b_n))
$$ 
of holomorphic function germs at 
$\BS{b}$, the initial forms of its elements with respect to the 
total degree generate a dual space with respect to 
a sesquilinear form as follows. 

Let $U$ be an open subset of an affine space $\BB{C}^n$ and 
let $\CAL{O}_n(U)$ denote the ring of holomorphic functions 
on $U$. 
Take $f(\BS{t})\in\CAL{O}_n(U)$ and $\BS{b}\in U$. 
The \textit{least part} $f_{\BS{b}}\DA$ of $f$ at 
$\BS{b}$ means the non-zero homogeneous part of the lowest 
degree of the power series expansion of $f$ with respect 
to the affine 
coordinates $\BS{t}':=\BS{t}-\BS{b}$ centred at $\BS{b}$. 
Since we shall consider these least parts as elements of 
a dual space of $\CAL{O}_n(U)$, we replace 
the variable $\BS{t}'$ in the least parts 
by the corresponding Greek letter $\BS{\tau}$. 
In \S \ref{sectionLEAST}, 
this will be treated as a Schwarz distribution supported at 
$\BS{b}$. We have no need to attach the symbol $\BS{b}$ to 
$\BS{\tau}$. The polynomial $f_{\BS{b}}\DA$ may be 
considered as a homogeneous element of $\BB{C}[\BS{\tau}]$. 
For example 
$$
\bigl((t_1-b_1)^p(t_2-b_2)^q
+(t_1-b_1)^{p+1}(t_2-b_2)^q\bigr)_{\BS{b}}\DA
=\tau_1^p\tau_2^q\in\BB{C}[\BS{\tau}]
.$$
If $Z$ is a vector subspace of 
$\CAL{O}_{n,\BS{b}}:=\BB{C}\{\BS{t}-\BS{b}\}$ or $\CAL{O}(U)$, 
then $Z_{\BS{b}}\DA\subset\BB{C}[\BS{\tau}]$ denotes 
the linear span 
$\Span_{\BB{C}}\left(f_{\BS{b}}\DA:\,f\in Z\right)$ 
over $\BB{C}$ 
of the least parts of elements of $Z$. The mapping 
$$
\DA:\,Z\longrightarrow Z_{\BS{b}}\DA,\qquad
f\longmapsto f_{\BS{b}}\DA
$$
is called the \textit{least operator.} 
Let 
$$
S_{n,\BS{b}}:\,\BB{C}[\BS{\tau}]\times\CAL{O}_{n,\BS{b}}
\longrightarrow\BB{C}
$$
denote the ordinary sesquilinear form 
(see \S \ref{section-weak}). 
If $Z$ is finite-dimensional, the restriction 
$S_Z:\,Z_{\BS{b}}\DA\times Z\longrightarrow\BB{C}[\BS{t}]$ 
of $S_{n,\BS{b}}$ is found to be non-degenerate by 
de~Boor and Ron \cite{REFbr1}, \cite{REFbr2}. 
Let us define the projector 
$$
\TAY_{Z,\BS{b}}:\,\CAL{O}_{n,\BS{b}}\longrightarrow Z_{\BS{b}}
$$ 
as the adjoint linear mapping of the inclusion 
$Z_{\BS{b}}\DA\longrightarrow
\CAL{O}_{n,\BS{b}}\DA=\BB{C}[\BS{\tau}]$. 
This is a retract i.e. 
$\TAY_{Z,\BS{b}}\COMP\kappa=\rm{id}_{Z_{\BS{b}}}$, 
where $\kappa$ denotes the inclusion mapping. 
A vector subspace of $\BB{C}[\BS{\tau}]$ is defined to be  
$D$\textit{-invariant} if it is closed with respect 
to the partial 
differentiations with respect to $\tau_1,\dots,\tau_n$. 
Let $U_Z^{\,\rm{inv}}$ denote the set of points where 
$Z_{\BS{b}}\DA$ $(\BS{b}\in U)$ is $D$-invariant and 
$U_Z^{\,\rm{bdl}}$ the set of points of $U$ where, 
for each $k\in\BB{N}$, 
the $k$-jets of elements of $Z$ form a vector bundle 
in a respective small neighbourhood. 
The set $U_Z^{\,\rm{bdl}}$ is invariant under biholomorphic 
transformation of $U$. We prove the following.\\

%%%%%%%%%%%%%%%%%
\noindent{\bf Theorem \ref{THMd-closed}.} 
\textit{
The set $U_Z^{\,\rm{bdl}}$ is 
non-empty and analytically open and 
$U_Z^{\,\rm{bdl}}\subset U_Z^{\,\rm{inv}}$, 
where a set is called 
\emph{analytically open} if it is the complement of a 
closed analytic subset.  
}\vspace*{1ex} 

%%%%%%%%%%%%%%%%
This is the key theorem of this paper but it can be 
proved simply 
describing $\BB{C}[\BS{\Phi}]^d_{\BS{b}}\DA$ by the language of 
``formal theory of differential equations". 
This method is suggested to the author by Tohru Morimoto. \\

%%%%%%%%%%%%%%%%%%%%
\noindent{\bf Theorem \ref{THMvector}.} 
\textit{
Let $U$ be an open subset of $\BB{C}^n$ and $Z$ 
a finite-dimensional 
vector subspace of $\CAL{O}_n(U)$. 
Then $U_Z^{\,\rm{inv}}$ is invariant under biholomorphic 
transformation of $U$ and the vector space 
$Z_{\BS{b}}$ has a structure of an Artinian algebra 
as a factor algebra of 
$\CAL{O}_{n,\BS{b}}$ through the projector 
$\TAY_{Z,\BS{b}}:\,\CAL{O}_{n,\BS{b}}
\longrightarrow Z_{\BS{b}}$ 
at each $\BS{b}\in U_Z^{\,\rm{inv}}$. 
This structure is unique up to canonical isomorphism as 
a contravariant tensor  (see Remark \ref{covariant}).  
}\vspace*{1ex} 
%%%%%%%%%%%%%%%%%%%%

These results trace back to the very interesting theory of 
Bos and Calvi \cite{REFBC-1}. 
Let $X$ be a complex submanifold of an open subset 
$U\subset\BB{C}^m$ and $\CAL{O}_{X,\BS{a}}$ the local algebra 
of the germs of holomorphic functions on $X$ at $\BS{a}$. 
The vector space of polynomials of degrees at most $d$ 
in $\BS{x}$ 
is denoted by  $\BB{C}[\BS{x}]^d\subset\BB{C}[\BS{x}]$. We put 
$$
P^d(X_{\BS{a}})
:=\BB{C}[\BS{x}]^d|_{X_{\BS{a}}}\subset\CAL{O}_{X,\BS{a}}
,$$ 
the vector subspace of polynomial functions on $X$ of 
degree at most $d$. 
Let 
$$
\BS{\Phi}:=(\Phi_1,\dots,\Phi_m):\,
\BB{C}_{\BS{b}}^n\longrightarrow\BB{C}_{\BS{a}}^m
$$
be a local parametrisation of $X$, which means that, 
if its range 
is restricted to the image $X_{\BS{a}}$, it is the germ 
of a biholomorphic mapping. In this paper, 
we express an analytic mapping germ by the upper case 
of a bold Greek letter and the algebra homomorphism induced 
by it is 
denoted by the lower case of the corresponding letter as 
$$
\varphi:\,\CAL{O}_{m,\BS{a}}\longrightarrow
\CAL{O}_{n,\BS{b}},\qquad
f\longmapsto \varphi(f):=f\COMP\BS{\Phi}
.$$ 
Let 
$$
\BB{C}[\BS{\Phi}]^d:=\varphi(\BB{C}[\BS{x}]^d)
\subset\CAL{O}_{n,\BS{b}}
$$ 
denote the vector subspace of all the pull-backs of elements 
of $\BB{C}[\BS{x}]^d$ by $\BS{\Phi}$ namely the polynomials 
of degree not larger than $d$ in the component functions 
$\Phi_1,\dots,\Phi_m$. Following Bos and Calvi, 
we introduce a special set 
$D_{\BS{a}}^{\varphi,d}\subset\BB{C}[\BS{\xi}]$ of higher 
order tangents of $X$ at $\BS{a}$ as the set of 
the push-forwards of elements of 
$\BB{C}[\BS{\Phi}]^d_{\BS{b}}\DA\subset\BB{C}[\BS{\tau}]$ 
by $\BS{\Phi}$, where 
$\BB{C}[\BS{\Phi}]^d_{\BS{b}}\DA$ is considered as a set of 
higher order tangents of $\BB{C}^n$ at $\BS{b}$. 
It can be expressed 
as $D_{\BS{a}}^{\varphi,d}
:={}^s\varphi(\BB{C}[\BS{\Phi}]_{\BS{b}}^d\DA)
\subset\BB{C}[\BS{\xi}]$, using 
the adjoint homomorphism 
${}^s\varphi:\,\BB{C}[\BS{\tau}]\longrightarrow
\BB{C}[\BS{\xi}]$ 
of $\varphi$. It is a dual space of $P^d(X_{\BS{a}})$ 
with respect to the sesquilinear form induced 
by $S_{n,\BS{b}}$ and we call its elements 
\textit{Bos-Calvi tangents}. 
We define the $\varphi$-Taylor projector 
$$
\TAY_{\BS{a}}^{\varphi,d}:\,\CAL{O}_{X,\BS{a}}
\longrightarrow P^d(X_{\BS{a}})
$$
of degree $d$ on $X\subset\BB{C}^m$ at $\BS{a}$ as the adjoint 
homomorphism of the inclusion 
$D_{\BS{a}}^{\varphi,d}\longrightarrow
\varphi(\BB{C}[\BS{\tau}\DA)
\subset\BB{C}[\BS{\xi}]$. 
The mapping $\TAY_{\BS{a}}^{\varphi,d}$ and the space 
$D_{\BS{a}}^{\varphi,d}$ are defined 
using a local parametrisation $\varphi$ of $X$ around $\BS{a}$. 
If the reader wants to see this complicated relation among 
these spaces and mappings at a glance, see Diagram 4 in \S 
\ref{section-taylorian}. 

\

%%%%%%%%%%%%%%%%%%%
\noindent{\bf Proposition \ref{PROideal-closed}.}
\textit{
Let $X$ be a complex submanifold of an open subset 
of $\BB{C}^m$. 
For any point $\BS{a}\in X$, the following conditions 
are equivalent for each local parametrisation $\varphi$. 
\begin{enumerate}
\item
The set $D_{\BS{a}}^{\varphi,d}$ of Bos-Calvi tangents is 
$D$-invariant at $\BS{a}$.
\item
The space of annihilators 
$(D_{\BS{a}}^{\varphi,d})^{\bot_X}
=\Ker\TAY_{\BS{a}}^{\varphi,d}$
is an ideal of $\CAL{O}_{X,\BS{a}}$.
\end{enumerate}
}

\

%%%%%%%%%%%%%%%%%%%%%%
\noindent{\bf Remark \ref{REMgen}.} 
Let $X$ be a complex submanifold of an open subset
of $\mathbb{C}^m$. If $D_{\boldsymbol{a}}^{\psi,d}$
is $D$-invariant for some $\psi$, 
it is so also for all local parametrisations
$\varphi$ at $\boldsymbol{a}$, that is, 
$D_{\boldsymbol{a}}^{\varphi,d}$ is a covariant tensor.

\

Let us call a point $\BS{a}\in X$ \textit{$D$-invariant point 
of degree} $d$ 
if it satisfies the condition $(1)$ (or $(2)$) of Proposition 
\ref{PROideal-closed} for any (or some) $\varphi$ and 
\textit{$D$-invariant point of degree $\infty$} 
if it is $D$-invariant point of degree $d$ for all 
$d\in\BB{N}$. 
Theorem \ref{THMd-closed} implies that the set of 
$D$-invariant points of degree $d$ contains a non-empty 
analytically open subset $X$. Thus the set of points 
which are not $D$-invariant of degree $\infty$ are contained in 
a countable union of thin analytic subsets of $X$, 
a set of first 
category in Baire's sense with Lebesgue measure 0 in $X$. 

In the case of a plane algebraic curve, 
Bos and Calvi \cite{REFBC-2} prove that 
$\TAY_{\BS{a}}^{\varphi,d}$ 
is independent of the choice of $\varphi$ if and only 
if the powers of monomials appearing in 
$\BB{C}[\BS{\Phi}]^d_{\BS{b}}\DA\subset\BB{C}[{\tau}]$ 
form a \textit{gap-free} sequence and that this condition 
actually holds at all but finite number of points on $X$, 
using the Wronskian. In this case, 
the gap-free property is equivalent to 
the $D$-invariance property of 
$\BB{C}[\BS{\Phi}]^d_{\BS{b}}\DA$. 
Bos and Calvi \cite{REFBC-2} give the name 
``Taylorian property" to the independence of 
$\TAY_{\BS{a}}^{\varphi,d}$ from the local 
parametrisation $\varphi$. 
In this paper, we generalise a the theorem of Bos and Calvi 
to analytic curves with larger codimensions as follows.\\

%%%%%%%%%%%%%%%%%%%%%%%
\noindent{\bf Theorem \ref{THMequivalence}.}
\textit{
Let $X$ be a $1$-dimensional regular complex submanifold of an 
open subset of $\BB{C}^m$. 
Take a local parametrisation $\BS{\Phi}:\,
\BB{C}_{\BS{b}}^n\rightarrow\BB{C}_{\BS{a}}^m$. 
Then, for any $d\in\BB{N}$, the following three properties 
of $\BS{a}$ are equivalent.
\begin{enumerate}
\item
For all $k\in\BB{N}_0$, $\BS{a}$ is a bundle point of the 
$k$-jet spaces of $\BB{C}[\BS{\Phi}]_{\BS{b}}^d\DA$. 
\item
The powers of monomials appearing in 
$\BB{C}[\BS{\Phi}]^d_{\BS{b}}\DA\subset\BB{C}[{\tau}]$ 
form a \textit{gap-free} sequence. 
\item
Point $\BS{a}$ is Taylorian of degree $d$. 
\end{enumerate}
}
%%%%%%%%%%%%%%%%%%%%%%%%%%

For a higher-dimensional submanifold, however, 
$D$-invariance does not mean the Taylorian property and the set 
$D_{\BS{a}}^{\varphi,d}$ of higher order tangents 
depends upon the local parametrisation $\BS{\Phi}$ 
as we will see in 
Example \ref{EXAdifference}.

We can measure simplicity of embedding $X\subset\BB{C}^m$ 
by the set of Bos-Calvi tangents. Let us put 
$$
\theta_{\CAL{O}_{n,\BS{b}},\BS{\Phi}}(d)
:=
\max\left\{\deg~p:~p\in\BB{C}[\BS{\Phi}]^d_{\BS{b}}\DA
\setminus\{0\}\right\}
.$$
Let $\ord_{\BS{b}}(f)$ of $f\in\CAL{O}_{n,\BS{b}}$ 
denote the vanishing order of $f$ at $\BS{b}$. 
Since $\theta_{\CAL{O}_{n,\BS{b}},\BS{\Phi}}(d)$ is equal to 
$$
\sup\left\{\ord_{\BS{b}}F(\Phi_1,\dots,\Phi_m):\,
F\in\BB{C}[\BS{x}]^d,\ F(\Phi_1,\dots,\Phi_m)\neq 0\right\}
,$$ 
its estimate as a function of $d$ is called a 
\textit{zero-estimates} 
of $\BS{\Phi}$. 
It is related to transcendence of the embedding of 
$X\subset\BB{C}^m$ at $\BS{a}$. 
Zero-estimate is one of the most important methods in the 
transcendental number theory. It is given for exponential 
polynomials, for some number theoretic functions or 
solutions of 
some good system of differential equations so far. 
Here, using $D$-invariance of degree $\infty$, we show 
an effective zero-estimate for a certain set of quite general 
holomorphic functions, but only at a general point. Let 
\begin{gather*}
\chi(\overline{X}_{\BS{a}},\,d)
:=\dim_{\BB{C}}\BB{C}[\BS{\Phi}]^d
-\dim_{\BB{C}}\BB{C}[\BS{\Phi}]^{d-1}
\\
=\dim_{\BB{C}}P^d(X_{\BS{a}})-\dim_{\BB{C}}P^{d-1}(X_{\BS{a}})
\quad
(\dim_{\BB{C}}P^{-1}(X_{\BS{a}})=0)
\end{gather*} 
denote the Hilbert function of 
the smallest algebraic subset (the Zariski closure) 
$\overline{X}_{\BS{a}}$ 
of $\BB{C}^m$ that contains a representative of 
the germ $X_{\BS{a}}$. We have the following estimates. 
\vspace*{1ex}

%%%%%%%%%%%%%%%%%%%%%%%
\noindent{\bf Theorem \ref{THMinequality}.} 
\textit{Let 
$\BS{\Phi}:\,\BB{C}_{\BS{b}}^n
\longrightarrow X_{\BS{a}}\subset\BB{C}_{\BS{a}}^m$ 
be an embedding of a complex manifold. 
Then we have 
$$
{n+d\choose n}+
\theta_{\CAL{O}_{n,\BS{b}},\BS{\Phi}}(d)-d
\le 
\dim_{\BB{C}}\BB{C}[\BS{\Phi}]^d
=
\sum_{i=0}^d \chi(\overline{X}_{\BS{a}},\,i)
\le
{m+d\choose m}
$$
for any $D$-invariant point $\BS{a}$ of degree $d$. 
Hence the transcendency index 
$$
\alpha(X_{\BS{a}}):=
\limsup_{d\to\infty}
\log_d\theta_{\CAL{O}_{n,\BS{b}},\BS{\Phi}}(d)
,$$ 
defined in \cite{REFpitman}, 
is majorised by $\dim \overline{X}_{\BS{a}}\,(\le m)$ at 
a $D$-invariant point of degree $\infty$. 
} 
\vspace*{1ex}
%%%%%%%%%%%%%%%%%%%%%%%%%

That is, the transcendency index of an embedding 
$X\subset\BB{C}^m$ of complex manifold is bounded 
above effectively 
excepting points of a set expressed as a countable 
union of thin analytic subsets, even if $X$ is quite general. 
This estimate has both 
merits and demerits in comparison with the zero-estimate 
obtained as a corollary of Gabrielov's multiplicity-estimate 
\cite{REFgabrielov} for Noetherian functions on 
an integral manifold of 
a Noetherian vector field. 
%%%%%%%%%%%%%%%%%%%%%%%%%

All results remain valid also for the real analytic category. 
We do not treat the multi-point interpolation problem 
as \cite{REFBC-1} in this paper. 
The earlier half of this paper consists of detailed 
descriptions of 
the basic facts which may be well-known to specialists of the 
respective fields. They are included 
because the author could not guess the fields of readers. 
%%%%%%%%%%%%%%%%%%%%%%%%%%%%%%
\section{Least spaces}\label{sectionLEAST}
%%%%%%%%%%%%%%%%%%%%%%%%%%%%%%%%%
Here we recall the least space of a vector space of 
holomorphic functions 
at a point. It is the graded space associated 
to the maximal-ideal-adic filtration. 
We use the term ``least space" and the simple symbol $\DA$ 
of de~Boor and Ron \cite{REFbr1}, \cite{REFbr2} used 
in the interpolation theory. 
\\ 
 
First we define the least operator and the least space 
in an intrinsic way. 
Let 
$$
\CAL{O}_{n,\BS{b}}:=\BB{C}\{\BS{t}-\BS{b}\}
=\BB{C}\{t_1-b_1,\dots,t_n-b_n\}
$$ 
denote the local algebra of convergent power series centred at 
$\BS{b}:=(b_1,\dots,b_n)\in\BB{C}^n$ and 
$$
\FRAK{m}_{n,\BS{b}}:=(\BS{t}-\BS{b})\CAL{O}_{n,\BS{b}}
=(t_1-b_1,\dots,t_n-b_n)\CAL{O}_{n,\BS{b}}
$$ 
its maximal ideal. 
This algebra $\CAL{O}_{n,\BS{b}}$ has a filtration 
$$
\CAL{O}_{n,\BS{b}}
=\FRAK{m}_{n,\BS{b}}^0 \supset \FRAK{m}_{n,\BS{b}}^1
\supset \FRAK{m}_{n,\BS{b}}^2 \cdots
$$
and it satisfies the following conditions: 
$$
\bigcap_{i\in\BB{N}_0} \FRAK{m}_{n,\BS{b}}^i=\{0\},\quad 
\dim_{\BB{C}}
\FRAC{\FRAK{m}_{n,\BS{b}}^i}{\FRAK{m}_{n,\BS{b}}^{i+1}}
=\FRAC{(n+i-1)!}{(n-1)!\,i!}<\infty
.$$
Here the latter equality follows from the fact 
that the homogeneous 
polynomials of degree $i$ form a representative system of the 
residue classes of 
${\FRAK{m}_{n,\BS{b}}^i}/{\FRAK{m}_{n,\BS{b}}^{i+1}}$. 
We define the \textit{least space} of $\CAL{O}_{n,\BS{b}}$ by 
$$
\CAL{O}_{n,\BS{b}}\DA
:=\displaystyle\bigoplus_{i\in\mathbb{N}_0}
\FRAC{\FRAK{m}_{n,\BS{b}}^i}{\FRAK{m}_{n,\BS{b}}^{i+1}}
\quad(\BB{N}_0:=\{0,1,\dots\})
.$$
An element contained in a single component 
${\FRAK{m}_{n,\BS{b}}^i}/{\FRAK{m}_{n,\BS{b}}^{i+1}}$ is called 
\textit{homogeneous.} 
Let us define the \textit{order function} 
$$
\ord_{\BS{b}}:\,\CAL{O}_{n,\BS{b}}\longrightarrow\BB{N}_0,\quad
f\longmapsto\ord_{\BS{b}}f;
$$
by
$$
\ord_{\BS{b}}\, f
:=\max\left\{i:f \in \FRAK{m}_{n,\BS{b}}^i\right\}
\quad(\ord_{\BS{b}}\,0=+\infty)
.$$
If $\ord_{\BS{b}}\, f=i$, the \textit{least part} 
$f_{\BS{b}}\DA$ of $f$ is 
defined to be the residue class of $f$ in 
${\FRAK{m}_{n,\BS{b}}^i}/{\FRAK{m}_{n,\BS{b}}^{i+1}}$ i.e. 
$$
f_{\BS{b}}\DA:=f\ \rm{mod}\ 
\FRAK{m}_{n,\BS{b}}^{\alpha+1} 
\qquad (\alpha:=\ord_{\BS{b}}\, f,\ 0_{\BS{b}}\DA:=0)
.$$
The mapping 
$$
\DA\,:\,\CAL{O}_{n,\BS{b}}\longrightarrow
\CAL{O}_{n,\BS{b}}\DA,\quad 
f\longmapsto f_{\BS{b}}\DA
$$
is called the \textit{least operator}. 
This least operator is not a linear map. It is obvious that 
$\CAL{O}_{n,\BS{b}}\DA
=\Span_{\BB{C}}(f_{\BS{b}}\DA:\,f\in\CAL{O}_{n,\BS{b}})$, 
the linear span of the least parts of the elements of 
$\CAL{O}_{n,\BS{b}}$. Thus an element of 
$\CAL{O}_{n,\BS{b}}\DA$ may not be homogeneous. 
These are the intrinsic definitions of the 
least part and the least space. 

The original definition of the least part of $f$ by 
de~Boor and Ron is the non-zero homogeneous part  
$f_{\BS{b}}^{\BS{t}}\DA$ 
of $f$ of the smallest 
degree in the power series expansion of $f$ 
with respect to some 
affine coordinate system $\BS{t}$. 
This homogeneous part $f_{\BS{b}}^{\BS{t}}\DA$ is 
a coordinate expression of 
$f_{\BS{b}}\DA$. Before \S \ref{section-intrinsic}, 
we fix an affine coordinate system 
and hence we omit the superscript $\BS{t}$ even in 
the coordinate expression $f_{\BS{b}}^{\BS{t}}\DA$. 
Let us adopt the multi-exponent expression: 
$$
\BS{\nu}:=(\nu_1,\dots,\nu_n),\quad 
(\BS{t}-\BS{b})^{\BS{\nu}}
=(t_1-b_1)^{\nu_1}\cdots(t_n-b_n)^{\nu_n}
.$$ 
Since we consider elements of  
$\CAL{O}_{n,\BS{b}}\DA$ as belong to the dual space of 
$\CAL{O}_{n,\BS{b}}$ 
in the later sections, we express them by a polynomial in 
Greek variables corresponding to the original as:
$$
\tau_i:=(t_i-b_i)_{\BS{b}}\DA,\quad
\BS{\tau}^{\BS{\nu}}:=(\BS{t}-\BS{b})^{\BS{\nu}}_{\BS{b}}\DA 
\in \FRAK{m}_{n,\BS{b}}^{|\BS{\nu}|}/
\FRAK{m}_{n,\BS{b}}^{|\BS{\nu}|+1}
.$$
Hence the least space $\CAL{O}_{n,\BS{b}}\DA$ is denoted 
by the polynomial algebra $\BB{C}[\BS{\tau}]$ with 
$\BS{\tau}:=(\tau_1,\dots,\tau_n)$, 
$\tau_i:=(t_i-b_i)_{\BS{b}}\DA$.
The product in $\BB{C}[\BS{\tau}]$ is a natural operation 
as a consequence 
of the property 
$$
\FRAK{m}_{n,\BS{b}}^i\FRAK{m}_{n,\BS{b}}^j
=\FRAK{m}_{n,\BS{b}}^{i+j}
.$$
Let $Z_{\BS{b}}$ be a finite-dimensional vector subspace of 
$\CAL{O}_{n,\BS{b}}$. 
We put 
$$
Z_{\BS{b}}\DA:=\Span_{\BB{C}}(f_{\BS{b}}\DA:\,f
\in Z_{\BS{b}})\subset\BB{C}[\BS{\tau}]
,$$
the linear span of $\{f_{\BS{b}}\DA:\,f\in Z_{\BS{b}}\}$ 
and call this 
the \textit{least space} of $Z_{\BS{b}}$. We know the following 
(which will be strengthened to Theorem \ref{THMminimal}). 
%%%%%%%%%%%%%%%%%%%%%%%%%%%
\begin{thm}\label{THMgraded-space}{\rm (de~Boor-Ron 
\cite[Proposition 2.10]{REFbr2}; cf. 
\cite[Theorem 7.1]{REFinterpolation})}
Let $Z_{\BS{b}}$ be a finite-dimensional vector subspace of 
$\CAL{O}_{n,\BS{b}}$. Then we have
$$
\dim_{\mathbb{C}}Z_{\BS{b}}\DA = \dim_{\mathbb{C}}Z_{\BS{b}}.
$$
\end{thm}
%%%%%%%%%%%%%%%%%%%%%%%%%%%%%%%%%%%%%%%%%%%%%%%%%%%%%%%%%%%%%%%%
\section{Jet spaces and multivariate Wronskians}%
\label{section-wronskian}
%%%%%%%%%%%%%%%%%%%%%%%%%%%%%%%%%%%%%%%%%%%%%%%%%%%%%%%%%%%%%%%%%
Let $Z$ be a vector space of holomorphic functions on 
an open subset $U\subset\BB{C}^n$. 
The $k$-jets of the elements of $Z$ at a 
point of $U$ form a vector space. 
If we gather such vector spaces 
only at good points of $U$, we get a holomorphic vector 
bundle, $k$-jet bundle of $Z$. 
The theorem of Walker on Wronskians implies that the 
jet sections of elements of $Z$ with order up to 
$\dim_{\BB{C}}Z-1$ span those of any order on an analytically 
open subset. This order $\dim_{\BB{C}}Z-1$ is minimal 
for a general $Z$. 
\\

Let $\CAL{O}_n$ denote the sheaf of germs of holomorphic 
functions on $\BB{C}^n$. 
We call the sheaf of germs of holomorphic sections 
of a holomorphic vector bundle \emph{the associated sheaf}
or an \emph{$\mathcal{O}_n$-module} associated to the bundle. 
It is expressed by the script style of the letter 
which is used for the bundle. 
The correspondence of holomorphic vector bundles on $U$ 
to the associated $\CAL{O}_n$-module 
defines a bijective mapping of the set of isomorphism 
classes of holomorphic vector bundles of rank $r$ over 
$U$ onto the set of isomorphism classes of locally 
free $\CAL{O}_U$-modules ($\CAL{O}_U=\CAL{O}_n|_U$) 
of rank $r$ over $U$ (see e.g. \cite[Proposition 3.3]{REFpr}). 
We use the parenthesised $\BS{b}$ for the values 
or the sets of 
the values at $\BS{b}$ (bundle fibre, e.g. $R_{k}({\BS{b}})$) 
and $\BS{b}$ in subscript style for the germs or the sets 
of the germs at $\BS{b}\in U$ (stalk, sheaf fibre, e.g. 
$\CAL{L}_{\BS{b}}^k$) or the indication 
of the centre of the coordinates for which the least part 
is defined (as $Z_{\BS{b}}\DA$). 

Let 
$$
\pi_k:\,J^k(\CAL{O}_U)\longrightarrow U\quad
\left(J^k(\CAL{O}_U)\cong U\times\BB{C}^{N(n,k)},\ 
N(n,k):={n+k\choose k}\right)
$$ 
denote the $k$-jet space of holomorphic functions on $U$, 
the holomorphic vector bundle 
of $k$-jets of holomorphic functions defined on open subsets 
of $U$. Its coordinates are denoted by 
$$
\bigl(\BS{t},\,(u_{\BS{\nu}}:\,|\BS{\nu}|\le k)\bigr)\quad
(\BS{t}:=(t_1,\dots,t_n)\in U,\ \BS{\nu}:=(\nu_1,\dots,\nu_n))
.$$
Let $\CAL{O}_n(V)$ denote the algebra of sections of 
$\CAL{O}_n$ 
over $V\subset U$. The \textit{$k$-jet extension} 
$j{\,}^k{f}$ of $f\in\CAL{O}_n(V)$ is defined by 
\begin{gather*}
j{\,}^k{f}:\,V\longrightarrow J^k(\CAL{O}_U),\quad
\BS{t}\longmapsto\Bigl(\BS{t},\,
u_{\BS{\nu}}(j{\,}^kf):\,|\BS{\nu}|\le k\Bigr)
\\
\left(u_{\BS{\nu}}(j{\,}^kf)
%=\FRAC{1}{\BS{\nu}!}S_{n,\BS{t}}\SLF{\BS{\tau}^{\BS{\nu}}}{f}
:=\FRAC{1}{\BS{\nu}!}\FRAC{\partial^{|\BS{\nu}|}f(\BS{t})}%
{\partial\BS{t}^{\BS{\nu}}}\right)
.\end{gather*}
This is a section of the jet space $J^k(\CAL{O}_U)$ over $V$. 
The coefficient ${1}/{\BS{\nu}!}$ of the $\BS{\nu}$-th fibre 
coordinate is convenient in the calculation of prolongation 
below. Thus the coordinates $u_{\BS{\nu}}\ (|\BS{\nu}|\le k)$ 
are called the 
\textit{fibre coordinates} corresponding to the 
normalised $\BS{\nu}$-th derivative. 

If $Z$ is a finite-dimensional vector subspace of 
$\CAL{O}_n(U)$, 
the evaluation of the jet extension at 
$\BS{b}\in U$ defines the mapping 
$$
j{\,}^k|_Z({\BS{b}}):\,Z\longrightarrow J^k(\CAL{O}_U)(\BS{b}),
\quad
f\longmapsto j{\,}^kf(\BS{b})
.$$
Let 
$$
\bigl(\BS{b};\,R_{k}({\BS{b}})\bigr)
:=\{j{\,}^kf(\BS{b}):\,f\in Z\}
$$ 
denote its image. Then 
we have the natural commutative Diagram 1 of 
linear mappings of vector spaces. 
%%%%%%%%%%%%%%
\renewcommand{\tablename}{\textsf{diagram}}
\begin{table}[h]
\caption{\textsf{\scriptsize Jet spaces of 
$Z\subset\BB{C}[\BS{\tau}]$}}%
\label{tab:1}
\unitlength=1.2ex
% \unitlength=1mm
\begin{center}\begin{picture}(65,12)(8,0)
\thicklines
\put(36.5,14){$Z$}
\path(37,13)(37,7)\path(36.4,8)(37,7)(37.6,8)
   \put(37.5,8.7){\small$j^{k}|_Z({\BS{b}})$}% vertical
\path(15.5,5)(20.5,5)\path(16.3,5.4)(15.5,5)(16.3,4.6)
% horizon leftleft
\put(21.8,4){$R_{k-1}({\BS{b}})$}
\path(29,5)(34.5,5)\path(29.8,5.4)(29,5)(29.8,4.6)
% horizon left
\put(35.5,4){$R_{k}({\BS{b}})$}
   \put(30,2.7){\small$\Sigma^k({\BS{b}})$}
\path(35,13)(27,7)\path(28,7.2)(27,7)(27.5,8)
   \put(25.3,10.9){\small$j^{k-1}|_Z({\BS{b}})$}% left down
\path(41,5)(46,5)\path(41.8,5.4)(41,5)(41.8,4.6)% horizon right
   \put(43.3,11){\small$j^{k+1}|_Z({\BS{b}})$}
\path(54.5,5)(59.5,5)\path(55.3,5.4)(54.5,5)(55.3,4.6)
% horizon rightright
\put(47,4){$R_{k+1}({\BS{b}})$}
   \put(41,2.7){\small$\Sigma^{k+1}({\BS{b}})$}
\path(39,13)(47,7)\path(46,7.2)(47,7)(46.5,7.9)% right down
\dottedline{.5}(12.5,5)(15.5,5)
\dottedline{.5}(59.5,5)(62,5)
\end{picture}\end{center}
\end{table}
%%%%%%%%%%%%%%%%%%%%%%%%%%
Here, the horizontal mappings are projections defined 
by forgetting the coordinates corresponding to 
the highest order derivatives. The total image 
$\bigcup_{\BS{b}\in U}(\BS{b},\,R_{k}({\BS{b}}))$ 
is not an analytic subset of $J^k(\CAL{O}_U)$ nor even 
a closed subset in general. Let us put 
$$
r_k:=\max\{\dim_{\BB{C}}R_k({\BS{t}}):\,\BS{t}\in U\},\quad
U_Z^k:=\{\BS{t}\in U:\,\dim_{\BB{C}}R_k({\BS{t}})=r_k\}
.$$
Let us call the complement of a closed analytic subset in $U$ 
\textit{analytically open} in $U$. 
Suppose that $U$ is connected. 
Since the points of $U_Z^k$ are characterised by the full rank 
condition of certain matrices with holomorphic elements, 
$U_Z^k$ is a non-empty analytically open subset. Putting 
$$
R_k:=\left\{\left(\BS{b};\,R_k(\BS{b})\right):
\,\BS{b}\in U_Z^k\right\}
,$$
we have a holomorphic vector bundle 
$$
\pi_k|_{R_k}:\,R_{k}\longrightarrow U_Z^k
.$$
%%%%%%%%%%%%%%%%%%%%%%%%%%%
\begin{defin}\label{DEFbundlepoint}
We call a point of $U_Z^k$ \textit{bundle point of the 
$k$-jet space} of $Z$ and a point of 
$U_Z^{\,\rm{bdl}}:=U_Z^0\cap U_Z^1\cap\cdots$ 
\textit{bundle point of all 
the jet spaces} of $Z$. 
\end{defin}
%%%%%%%%%%%%%%%%%%%%%%%%%
\begin{exa}
Let us put $Z:=\Span(s^2,t^2,s^3))\in\CAL{O}_2(\BB{C}^2)$. 
Since the higher order derivatives of $s^2,t^2,s^3$ 
with respect to 
multiple order 
$$
\BS{\nu}=(0,0),\ (1,0),\ (0,1),\ (2,0),\ (1,1),\ (0,2)
,\ (3,0),\ (2,1),\ (1,2),\ (0,3)
$$ 
are listed as 
$$
\begin{pmatrix}
s^2&2s  &0 &2 &0 &0&0&0&0&0\\
t^2&0   &2t&0 &0 &2&0&0&0&0\\
s^3&3s^2&0 &6s&0 &0&6&0&0&0
\end{pmatrix}
,$$
we have 
\begin{gather*}
r_0=1,\ r_1=r_2=\cdots=3; 
\\
U_Z^0=\BB{C}^2\setminus\{(0,0)\}
\supsetneqq
U_Z^1=\BB{C}^2\setminus\{st=0\}
\\
\subsetneqq
U_Z^2=\BB{C}^2\setminus\{s=0\}
\subsetneqq
U_Z^3=U_Z^4=\cdots=\BB{C}^2;
\\
U_Z^{\,\rm{bdl}}=\BB{C}^2\setminus\{st=0\}
.\end{gather*} 
Hence there is no monotone inclusion relation among 
$U_Z^0,U_Z^1,U_Z^2,\dots$ (see \ref{LEMbdl-nbd}). 
\end{exa}
%%%%%%%%%%%%%%%%%%%%

Now we recall a known fact on multivariate Wronskians. 
For polynomials in multi-variable, Siegel \cite{REFsiegel}
and Roth \cite{REFroth} found that their 
linear independence is judged by non-vanishing 
of certain set of Wronskians and applied it to 
the theory of the rational approximation to algebraic numbers. 
Walker \cite{REFwalker} has obtained the minimal set of 
Wronskians needed to judge the linear independence. 
It allows us to write 
out a minimal finite system of PDEs explicitly 
whose solution space is a given finite-dimensional 
vector subspace $Z\subset\CAL{O}_n(U)$. 
It also enables us to state 
the subsequent arguments efficiently. 
%%%%%%%%%%%%%%%%%%%%%%%%%%%%%
\begin{defin}
Walker called 
$Y:=\{\BS{\nu}_1,\dots,\BS{\nu}_m\}\in(\BB{N}_0^n)^m$ 
\textit{Young-like} if it satisfies the following condition:
$$
\left(\BS{\nu}\in\BB{N}_0^n,\ \exists\ \BS{\nu}_i\in Y:\ 
\BS{\nu}\le\BS{\nu}_i\right)\ 
\Longrightarrow 
{\BS{\nu}}\in\BS{Y}
.$$
Here $\le$ implies the product order of the usual order 
of $\BB{N}_0$ defined as 
$$
\BS{\nu}:=(\nu_1,\dots,\nu_n)\le \BS{\mu}:=(\mu_1,\dots,\mu_n)
\Longleftrightarrow
\nu_1\le\mu_1,\dots,\nu_1\le\mu_m
.$$ 
The property of Young-likeness of $Y$ is equivalent to the 
condition that\\ 
$\Span_{\BB{C}}(\BS{\tau}^{\BS{\nu}_1},\dots,
\BS{\tau}^{\BS{\nu}_m})$ is 
$D$-invariant in the sense defined in 
\S \ref{section-prolongation}. 
Let us put 
$$
\CAL{Y}_m:=\{Y\in(\BB{N}_0^n)^m:\,Y\text{ is Young-like}\}
.$$
\end{defin}
%%%%%%%%%%%%%%%%%%%%%%%%%%
A set with this property is called in various ways: 
\emph{order-closed set} in \cite{REFbr2}, 
\emph{monotone} or \emph{lower set} in others. 
%%%%%%%%%%%%%%%%%%%%%%%%%%
\begin{thm}\label{THMwalker}%
{\rm(Siegel \cite{REFsiegel}; Roth \cite{REFroth}; 
Walker 
\cite[Theorem 3.1, Theorem 3.4, Remark in \S 3]{REFwalker})}
\hspace*{\fill}\\ 
\hspace*{3.5ex}(1)\hspace{1.2ex} 
\begin{minipage}[t]{64.5ex}
Let $f_1,\dots,f_m$ be meromorphic functions on a 
connected open subset $U\subset\BB{C}^n$. 
Then they are linearly independent if and only if 
there exists at least one 
$\{\BS{\nu}_1,\dots,\BS{\nu}_m\}\in\CAL{Y}_m$ 
such that 
$$
W(f_1,\dots,f_m;\,\BS{\nu}_1,\dots,\BS{\nu}_m):=
\begin{array}{|cccc|}
f_1^{(\BS{\nu}_1)} & f_1^{(\BS{\nu}_2)} & \ldots & 
f_1^{(\BS{\nu}_m)} \\
f_2^{(\BS{\nu}_1)} & f_2^{(\BS{\nu}_2)} & \ldots & 
f_2^{(\BS{\nu}_m)} \\
\vdots & \vdots & \vdots & \vdots \\
f_m^{(\BS{\nu}_1)} & f_m^{(\BS{\nu}_2)} & \cdots & 
f_m^{(\BS{\nu}_m)}
\end{array}
$$
does not vanish identically. 
\end{minipage}
\\
\hspace*{3.5ex}(2)\hspace{1.2ex} \begin{minipage}[t]{64.5ex}
The set $\CAL{Y}_m$ is the least set among 
the sets with this property in the following sense. 
If $\CAL{Y}'\subset(\BB{N}_0^n)^m$ and 
$\CAL{Y}'\subsetneq\CAL{Y}_m$, 
there exist linearly independent 
monomials $f_1,\dots,f_m$ such that 
$$
W(f_1,\dots,f_m;\,\BS{\nu}_1,\dots,\BS{\nu}_m)=0
$$
for all $(\BS{\nu}_1,\dots,\BS{\nu}_m)\in\CAL{Y}'$. 
\end{minipage}
\end{thm}
%%%%%%%%%%%%%%%%%%%%%%%%%%%%%%%%%%%%%

We have the following immediate consequence of this theorem. 
%%%%%%%%%%%%%%%%%%%%%%%%%%%%%%%%%%%%%
\begin{cor}\label{CORwalker}
Let $\{f_1,\dots,f_m\}$ be a basis of a vector subspace 
$Z\subset\CAL{O}_{n}(U)$. 
Then $Z$ is the space of holomorphic solutions of 
the system of PDEs: 
$$
W(f_1,\dots,f_{m},y\/;\,\BS{\nu}_1,\dots,\BS{\nu}_{m+1})=0
\qquad
(\{\BS{\nu}_1,\dots,\BS{\nu}_{m+1}\}\in\CAL{Y}_{m+1})
,$$
where $y=y(\BS{t})\in\CAL{O}_{n}(U)$ denotes 
the unknown function. 
\end{cor}
%%%%%%%%%%%%%%%%%%%%%%%%%%
\begin{lem}\label{LEMbdl-nbd}
Let $U$ be a connected open subset of $\BB{C}^n$ and 
let $\{f_1,\dots,f_m\}$ $(m\ge 1)$ be a basis of 
a vector subspace 
$Z\subset\CAL{O}_{n}(U)$. Then we have the following. 
\begin{enumerate}
\item
If $Y=\{\BS{\nu}_1,\dots,\BS{\nu}_m\}
\in\CAL{Y}_m$ and if $W(f_1,\dots,f_m;Y)(\BS{b})\neq 0$, 
we see that $\BS{b}\in U_Z^{m-1}$ and that the vectors 
$\left(f_1^{(\BS{\nu}_i)}(\BS{b}),\dots,
f_m^{(\BS{\nu}_i)}(\BS{b})\right)$ 
$(i=1,\dots,m)$ span all 
$\left(f_1^{(\BS{\nu})}(\BS{b}),\dots,f_m^{(\BS{\nu})}
(\BS{b})\right)$ 
$(\BS{\nu}\in\BB{N}_0^m)$. 
Hence the fibre coordinates 
$u_{\BS{\nu}_1},\dots, u_{\BS{\nu}_m}$ of $J^{m-1}(\CAL{O}_U)$ 
form a fibre coordinates system of $R_{m-1}$ over a 
neighbourhood of $\BS{b}$. 
\item
We have 
$$
1=r_0\le r_1\le \cdots\le r_{m-1}=r_m=\cdots=m
;$$$$
U_Z^{m-1}\subset U_Z^m\subset\cdots;\quad 
U_Z^{\,\rm{bdl}}=U_Z^1\cap\cdots\cap U_Z^{m-1}
.$$
\item
The vector bundles $R_k$ $(k\ge m-1)$ are all isomorphic 
if they are restricted to $U_Z^{m-1}$. 
\item
The sets $U_Z^k$ and $U_Z^{\,\rm{bdl}}$ 
are non-empty analytically open subset of $U$ and 
independent of 
the change of local coordinates of $U$. 
\end{enumerate}
\end{lem}
%%%%%%%%%%%%%%%%%%%%%%%%%%

\begin{proof}
\begin{enumerate}
\item
If $\boldsymbol{\nu}_i
\in Y\in\CAL{Y}_m$, there is a chain connecting it 
to $(0,\dots,0)$. It is not longer 
than $\# Y-1=m-1\quad(1\le i\le m)$. 
Then the coordinates $u_{\BS{\nu}_1},\dots,u_{\BS{\nu}_m}$ 
form a subset of the fibre coordinate system of 
$J^{m-1}(\CAL{O}_U)(\BS{b})$. 
The assumption $W(f_1,\dots,f_m;Y)(\BS{b})\neq 0$ implies that 
$R_{m-1}(\BS{b})$ is a $m$-dimensional subspace of 
$J^{m-1}(\CAL{O}_U)(\BS{b})$. 
Since $R_k(\BS{b})$ is spanned by $m$ vectors, this is maximal 
and we see that $\BS{b}\in U_Z^{m-1}$. 
The rest are now obvious. 
\item
The fact that $r_i$ are not decreasing is obvious. 
By Theorem \ref{THMwalker}, there exist $\BS{b}$ and 
$Y:=\{\BS{\nu}_1,\dots,\BS{\nu}_m\}\in\CAL{Y}_m$ with 
$W(f_1,\dots,f_m;Y)(\BS{b})\neq 0$. Then $r_{m-1}=m$ by (1). 
These prove the first expression. Since $r_{m-1}=r_m=\cdots=m$ 
and since $R_k$ are increasing, 
$U_Z^{m-1}\subset U_Z^m\subset\cdots$ follows. 
Therefore 
$U_Z^{\,\rm{bdl}}
=U_Z^1\cap U_Z^2\cap\cdots=U_Z^1\cap\cdots\cap U_Z^{m-1}$. 
\item
The property (1) implies that $u_{\nu}$ $(|\BS{\nu}|\le m-1)$ 
is a linear combination 
of $u_{\BS{\nu}_1},\dots,u_{\BS{\nu}_m}$ in $R_{m-1}$ with 
coefficients in $\CAL{O}_n$. 
Then $R_k$ $(k\ge m-1)$ are not proper extensions of $R_{m-1}$. 
\item
We have already stated that $U_Z^k$ are open in their 
definition above. Independence from coordinate change is 
a consequence of the fact 
that the jet spaces are contravariant geometric object 
 (see Remark \ref{covariant}), which will be detailed in 
\S \ref{section-intrinsic}. 
\end{enumerate}
\end{proof}
%%%%%%%%%%%%%%%%%%%%%%%%%%
\section{Generic $D$-invariance of least spaces}
\label{section-prolongation}
%%%%%%%%%%%%%%%%%%%%%%%%%%%%%%%%
Let $Z$ be a vector space of holomorphic functions 
on an open subset 
$U\subset\BB{C}^n$. 
The vector space of the germs of the elements of $Z$ 
at $\BS{b}\in U$ is denoted by $Z_{\BS{b}}$. 
The least space $Z_{\BS{b}}\DA$ of $Z_{\BS{b}}$ is identified 
as a vector subspace of $\BB{C}[\BS{\tau}]$ 
(as stated in Introduction). 
Our main purpose here is to prove that 
$Z_{\BS{b}}\DA$ is closed under partial differentiation by 
$\tau_i$ $(1\le i\le n)$ at a general point of $U$. 
The point of the proof is the prolongation of PDEs annihilating 
the jets of elements of $Z$, which is suggested to 
the author by Tohru Morimoto. 

%%%%%%%%%%%%%%%%%%%%%%%%
\begin{defin}
A vector subspace 
$Q\subset\BB{C}[\BS{\tau}]$ 
$(\BS{\tau}:=(\tau_1,\dots,\tau_n))$ is 
$D$\textit{-invariant} if it is closed with respect 
to the partial 
differentiations with respect to $\tau_1,\dots,\tau_n$. 
Let $Z$ be a vector space of holomorphic functions 
on an open subset 
$U\subset\BB{C}^n$. We put 
$$
U_Z^{\,\rm{inv}}
:=\{\BS{b}:Z_{\BS{b}}\DA\text{ is $D$-invariant}\}
.$$
\end{defin}
%%%%%%%%%%%%%%%%%%%%%%%
\begin{prop}\label{PROnonvanishing}
If $0<\dim_{\BB{C}}Z_{\BS{b}}<\infty$ and 
$\BS{b}\in U_Z^{\,\rm{inv}}$, then there exists 
$f\in Z_{\BS{b}}$ such that $f(\BS{b})\neq 0$. 
\end{prop}
%\vspace*{1ex}
%%%%%%%%%%%%%%%%%%%%%%%%%%

\begin{proof}
If $f(\BS{b})=0$ for all $f\in Z_{\BS{b}}$, we have 
$1\not\in Z_{\BS{b}}\DA$. This contradicts the assumption of 
$D$-invariance of $Z_{\BS{b}}\DA$. 
\end{proof}
%%%%%%%%%%%%%%%%%%%%%

Therefore, at a point $\BS{b}$ of the simultaneous 
vanishing locus of 
the elements of $Z$, the least space $Z_{\BS{b}}\DA$ is 
not $D$-invariant. 
The converse does not holds as we will see in 
Example \ref{EXA1}. 
%%%%%%%%%%%%%%%%%%%%%%%
\begin{prop}\label{PROtranslation}
If $Q\subset\BB{C}[\BS{\tau}]$ is $D$-invariant, 
it is translation invariant i.e. $p(\BS{\tau})\in Q$ implies 
$p(\BS{\tau}+\BS{b})\in Q$ for any constant vector $\BS{b}$. 
\end{prop}
%%%%%%%%%%%%%%%%%%%%%%

\begin{proof}
This is obvious from the ordinary Taylor formula:
$$% \\$\hfill\displaystyle
p(\BS{\tau}+\BS{b})
=\sum_{|\BS{\nu}|\le d}\FRAC{1}{\BS{\nu}!}
\FRAC{\partial^{|\BS{\nu}|} p(\BS{\tau})}%
{\partial\BS{\tau}^{\BS{\nu}}}\BS{b}^{\BS{\nu}}
\quad (d:=\deg p)
.$$
\end{proof}

Let $\CAL{L}^k$ denote the sheaf of germs of 
holomorphic functions 
on the manifold $J^k(\CAL{O}_{U_Z^{\,\rm{bdl}}})$ 
vanishing on the 
submanifold $R_k$ which are linear in the fibre coordinates $u_{\BS{\nu}}$. 
This is an $\CAL{O}_{U_Z^{\,\rm{bdl}}}$-module 
on $U_Z^{\,\rm{bdl}}$. 
The local sections of $\CAL{L}^k$ are functions 
$$
A(\BS{t},\BS{u})
:=\sum_{|\BS{\nu}|\le k}\alpha_{\BS{\nu}}(\BS{t})
\cdot u_{\BS{\nu}}
$$
which are homogeneous linear in fibre coordinates 
$u_{\BS{\nu}}$ of 
$J^k(\CAL{O}_{U_Z^{\,\rm{bdl}}})$ with coefficients 
$\alpha_{\BS{\nu}}(\BS{t})\in\CAL{O}_n(V)$ 
$(V\subset U_Z^{\,\rm{bdl}})$. 
If $u_{\BS{\nu}}$ are replaced by 
the corresponding differential operators 
$(1/\BS{\nu}!)\cdot{\partial^{|\BS{\nu}|}}
/{\partial\BS{t}^{\BS{\nu}}}$, 
then elements of ${\CAL{L}}^k(V)$ become linear partial 
differential 
operators which annihilates the functions of $Z$ on $V$. 
%%%%%%%%%%%%%%%%%%%%%%%%%%%%%%%%%%%%%
\begin{defin}
Take a local section 
$$
A(\BS{t},\BS{u})
:=\sum_{|\BS{\nu}|\le k-1}
\alpha_{\BS{\nu}}(\BS{t})\cdot u_{\BS{\nu}}
\in\CAL{L}^{k-1}(V)\quad 
(\alpha_{\BS{\nu}}(\BS{t})\in\CAL{O}_n(V))
$$
over $V\subset U$. 
Differentiating the relation 
$$
\sum_{|\BS{\nu}|\le k-1}\dfrac{1}{\BS{\nu}!}
\alpha_{\BS{\nu}}(\BS{t})
\cdot \FRAC{\partial^{|\BS{\nu}|}f}{\partial \BS{t}^{{\nu}}}=0
\qquad(f\in Z)
$$ 
by $t_i$, we have 
$$
\sum_{|\BS{\nu}|\le k-1}
\dfrac{1}{\BS{\nu}!} \left(
\alpha_{\BS{\nu}}^{(\BS{e}_i)}(\BS{t})
\cdot \FRAC{\partial^{|\BS{\nu}|}}{\partial\BS{t}^{\BS{\nu}}}
+\alpha_{\BS{\nu}}(\BS{t})
\cdot \FRAC{\partial^{|\BS{\nu}+\BS{e}_i|}}%
{\partial\BS{t}^{\BS{\nu}+\BS{e}_i}}
\right)f=0
\qquad(f\in Z)
,$$
where $\BS{e}_i:=(\delta_{1i},\delta_{2i},\dots,\delta_{ni})$ 
denotes the $i$-th unit vector. 
Hence $\alpha_{\BS{\nu}}^{(\BS{e}_i)}$ expresses 
the derivative of $\alpha_{\BS{\nu}}$ by $t_i$. 
This equation implies that
$$
\sum_{|\BS{\nu}|\le k-1}
\left(\alpha_{\BS{\nu}}^{(\BS{e}_i)}(\BS{t})\cdot u_{\BS{\nu}}
+(\nu_i+1)\alpha_{\BS{\nu}}(\BS{t})
\cdot u_{\BS{\nu}+\BS{e}_i}\right)\in\CAL{L}^k(V)
\quad(i=1,\dots,n)
.$$
These are called the \textit{first prolongations} 
of $A(\BS{t},\BS{u})$. 
\end{defin}
%%%%%%%%%%%%%%%%%%%%%%%%%%%%%%%%%%%%
\begin{thm}\label{THMd-closed}
Let $\CAL{O}_n(U)$ be the algebra of holomorphic functions on 
a connected open subset $U\subset\BB{C}^n$, 
$Z$ its finite-dimensional 
vector subspace and $U_Z^{\,\rm{bdl}}\subset U$ 
the set of bundle points of all 
the jet spaces of $Z$ (Definition \ref{DEFbundlepoint}). 
Then $U_Z^{\,\rm{bdl}}$ is a non-empty analytically 
open subset and $U_Z^{\,\rm{bdl}}\subset U_Z^{\,\rm{inv}}$.
\end{thm}
%%%%%%%%%%%%%%%%%%%%%%%%

\begin{proof}
We have seen that $U_Z^{\,\rm{bdl}}$ is a non-empty 
analytically open set in Lemma \ref{LEMbdl-nbd}.
The canonical projections 
$$
\Pi^k:\,{J}_k(\CAL{O}_U)\longrightarrow 
{J}_{k-1}(\CAL{O}_U),\qquad
\Sigma^k:\,R_k\longrightarrow R_{k-1}
$$
are constant rank homomorphisms over $U_Z^{\,\rm{bdl}}$ 
and they induce the inclusion 
$i^{\,k}:\,\Ker\Sigma^k\longrightarrow\Ker\Pi^k$ 
of locally free analytic sheaves by the following diagram: \\
%%%%%%%%%%%%%%%%%%%%
{\unitlength=1ex
\hspace*{3ex}
\begin{picture}(65,16)
%\thicklines
%%
\put(11,4){$\CAL{J}_{k-1}(\CAL{O}_U)$}
\path(20.5,5)(25.5,5)\path(21.3,5.4)(20.5,5)(21.3,4.6)
% horizon left
\put(26.6,4){$\CAL{J}_k(\CAL{O}_U)$}
\path(34,5)(39,5)\path(34.8,5.4)(34,5)(34.8,4.6)
% horizon middle
\path(48,5)(52.5,5)\path(48.3,5.4)(47.5,5)(48.3,4.6)
% horizon right
\put(53,4){0.}
\put(39.5,4){$\Ker\Pi^{k}$}
\put(14,11){$\CAL{R}_{k-1}$}
\path(20.5,12)(25.5,12)\path(21.3,12.4)(20.5,12)(21.3,11.6)
% horizon left
\put(27.6,11){$\CAL{R}_{k}$}
\path(33,12)(38,12)\path(33.8,12.4)(33,12)(33.8,11.6)
% horizon middle
\path(48,12)(52,12)\path(48.3,12.4)(47.5,12)(48.3,11.6)
% horizon right
\put(53,11){0}
\put(39,11){$\Ker\Sigma^{k}$}
\dashline[30]{.3}(41.5,10.5)(41.5,6)
\path(40.9,7)(41.5,6)(42.1,7)
   \put(42,8){\small$i^{\,k}$}% vertical
\path(28.5,10.5)(28.5,6)\path(27.9,7)(28.5,6)(29.1,7)
%   \put(29,8){\small$i^{k}$}% vertical
\path(15.5,10.5)(15.5,6)\path(14.9,7)(15.5,6)(16.1,7)
%   \put(16,8){\small$i^{k-1}$}% vertical
\put(22,12.5){\small$\Sigma^{k}$}
\put(22.5,3){\small$\Pi^{k}$}
%\put(26,0){$(l<k,\ k\ge k-1)$. \quad$l=k-1?$}
\end{picture}
}\\
%%%%%%%%%%%%%%%%%%%%%%%%%%%%%
A local section of $\Ker\Pi^k$ is expressed as 
$$
f(\BS{t},\BS{\tau})
=\sum_{|\BS{\nu}|=k}
\beta_{\BS{\nu}}(\BS{t})\BS{\tau}^{\BS{\nu}}
\quad
\left(\beta_{\BS{\nu}}(\BS{t}):=\dfrac{1}{\BS{\nu}!}
\dfrac{\partial^{|\BS{\nu}|}f}%
{\partial\BS{\tau}^{\BS{\nu}}}(\BS{t},\BS{0})\right)
,$$
where the monomials $\BS{\tau}^{\BS{\nu}}$ 
stand for the base of the fibre 
corresponding to coordinate $u_{\BS{\nu}}$. 
(We may write $\BS{\tau}^{\BS{\nu}}$ 
as $(d\BS{t})^{\odot\BS{\nu}}
=(dt_1)^{\odot\nu_1}\odot\cdots\odot (dt_n)^{\odot\nu_n}$, 
using the symmetric tensor product 
$\odot$, see Theorem \ref{THMvector}, (1).) 

The least space $Z_{\BS{b}}\DA$ is got by evaluating 
$\Ker\Sigma^k$ at $\BS{b}$. 
Hence $D$-invariance of $Z_{\BS{b}}\DA$ 
at degree $k$ reduces to the implication 
$$
f\in\Ker\Sigma^k\ \Longrightarrow\ 
\frac{\partial f}{\partial \tau_i}\in\Ker\Sigma^{k-1}
.$$
Take $f(\BS{t},\BS{\tau})\in\Ker\Sigma^k$ and 
any defining equation 
$$
A(\BS{t},\BS{u}):=\sum_{|\BS{\nu}|\le k-1} 
\alpha_{\BS{\nu}}(\BS{t})\cdot u_{\BS{\nu}}\in\CAL{L}^{k-1}(V)
$$
of $\CAL{R}_{k-1}$ over a neighbourhood $V$ of ${\BS{b}}$, 
we have 
\begin{gather*}
A\left(j{\,}^k\left(\FRAC{\partial f}{\partial\tau_i}
\right)\right)
=
\sum_{|\BS{\nu}|\le k-1} 
\alpha_{\BS{\nu}}(\BS{t})\cdot  u_{\BS{\nu}}
\left(j{\,}^k\left(\FRAC{\partial f}{\partial\tau_i}
\right)\right)
\\
=
\sum_{|\BS{\nu}|\le k-1}
(\nu_i+1)
\alpha_{\BS{\nu}}(\BS{t})\cdot u_{\BS{\nu}+\BS{e}_i}(j{\,}^kf)
\\
=
\sum_{|\BS{\nu}|\le k-1}
\left(
\alpha_{\BS{\nu}}^{(\BS{e}_i)}(\BS{t})\cdot u_{\BS{\nu}}
+(\nu_i+1)
\alpha_{\BS{\nu}}(\BS{t})\cdot u_{\BS{\nu}+\BS{e}_i}\right)
(j{\,}^kf)=0
.
\end{gather*}
Here, since $|\BS{\nu}|\le k-1$ implies 
$u_{\BS{\nu}}(j{\,}^kf)=0$, 
the third equality follows. 
The last equality follows from the first prolongations 
stated above. 
This proves that 
${\partial f}/{\partial \tau_i}\in \CAL{R}_{k-1,\BS{b}}$. 
The inclusion 
${\partial f}/{\partial \tau_i}\in\Ker\Sigma_{\BS{b}}^{k-1}$ 
follows from $f(\BS{t},\BS{\tau})\in\Ker\Sigma^k$ 
(homogeneity of $f$). 
Evaluating at $\BS{b}\in U_Z^k\cap U_Z^{k-1}$, 
we have shown that $Z_{\BS{b}}\DA$ is $D$-invariant at 
degree $k$. Then total $Z_{\BS{b}}\DA$ is $D$-invariant 
on the open subset 
$$
U_Z^{\,\rm{bdl}}
:=U_Z^0 \cap U_Z^1 \cap \cdots=U_Z^0 \cap\cdots\cap U_Z^{m-1}
$$$$
(m:=\dim_{\BB{C}}Z)
$$ 
described in Proposition \ref{LEMbdl-nbd}. 
\end{proof}
%%%%%%%%%%%%%%%%%%%%%%%%%%%
By Corollary \ref{CORwalker}, there exists $k\le m-1$ 
such that a 
system of linear PDEs of order $k+1$ is sufficient 
to select the 
sections of $Z$, namely $Z$ is defined by $\CAL{L}^k$. 
We may call the sheaf $\CAL{L}^k$ for such $k$ 
the \textit{defining system of PDEs} for $Z$. 
In particular, $\CAL{L}^{m-1}$ is a defining system of $Z$. 
%%%%%%%%%%%%%%%%%%%%%%%%%%%%
\section{Sesquilinear forms and weak topologies}
\label{section-weak}
%%%%%%%%%%%%%%%%%%%%%%%%%%%%%%%
Here we recall the sesquilinear form on the product 
$\BB{C}[\BS{\tau}]\times\BB{C}\{\BS{t}\}$ of 
the space of polynomials and the convergent 
power series algebra. 
The restriction of this sesquilinear form to the product of a 
finite-dimensional subspace $Z\subset\BB{C}\{\BS{t}\}$ and 
its least space $Z_{\BS{b}}\DA\subset\BB{C}[\BS{\tau}]$ 
is proved to be 
non-degenerate by de~Boor-Ron {\cite{REFbr1}}, {\cite{REFbr2}}. 
We decide the notation and certify their properties 
through the definite statements on bilinear forms in Bourbaki 
\cite{REFEVT}.\\

Let us define a complex bilinear form 
$$
B_{n,\BS{b}}:\BB{C}[\BS{\tau}]\times\CAL{O}_{n,\BS{b}}
\longrightarrow \BB{C},\quad
(p,\, f)\longmapsto B_{n,\BS{b}}\BLF{p}{f}
,$$
by 
$$
B_{n,\BS{b}}\BLFFF{\sum_{\mbox{\scriptsize finite}}
a_{\BS{\nu}}\BS{\tau}^{\BS{\nu}}}
{\sum b_{\BS{\mu}}(\BS{t}-\BS{b})^{\BS{\mu}}}
:=
\sum_{\mbox{\scriptsize finite}} \BS{\nu}! 
a_{\BS{\nu}}b_{\BS{\nu}}
,$$
where $\BS{\nu}!=\nu_1!\cdots\nu_m!$. In particular, 
\begin{gather*}
B_{n,\BS{b}}
\BLF{\BS{\tau}^{\BS{\nu}}}{(\BS{t}-\BS{b})^{\BS{\mu}}}
=
\frac
{\partial^{|\BS{\nu}|}(\BS{t}-\BS{b})^{\BS{\mu}}}
{\partial\BS{t}^{\BS{\nu}}}(\BS{b})
\\
=
\frac
{\partial^{\nu_1+\dots+\nu_n}(\BS{t}-\BS{b})^{\BS{\mu}}}
{\partial t_1^{\nu_1}\dots\partial t_n^{\nu_n}}(\BS{b})
=
\left\{
\begin{array}{ll}
\BS{\nu}! & (\BS{\nu}=\BS{\mu})\\ 
0      & (\BS{\nu}\neq\BS{\mu})
\end{array}
\right.
.\end{gather*}
Thus the monomial 
$\BS{\tau}^{\BS{\nu}}$ can be identified with 
the signed higher order 
derivative $(-1)^{|\BS{\nu}|}\delta_{\BS{b}}^{(\BS{\nu})}$ of 
the Dirac delta function supported at $\BS{b}\in\BB{C}^m$. 
The sign $(-1)^{|\BS{\nu}|}$ here originates 
from partial integral in the Schwarz distribution theory 
\cite[(II, 1;7)]{REFschwarz}. 

Now let $u$ denote the complex conjugations: 
\begin{gather*}
u:\BB{C}[\BS{\tau}]\longrightarrow \BB{C}[\BS{\tau}],\quad
\sum_{\mbox{\scriptsize finite}} 
a_{\BS{\nu}}\BS{\tau}^{\BS{\nu}}\longmapsto
\sum_{\mbox{\scriptsize finite}} 
\bar{a}_{\BS{\nu}}\BS{\tau}^{\BS{\nu}}
,\\
u:\CAL{O}_{n,\BS{b}}\longrightarrow\CAL{O}_{n,\BS{b}},\quad
\sum b_{\BS{\mu}}(\BS{t}-\BS{b})^{\BS{\mu}}\longmapsto
\sum \bar{b}_{\BS{\mu}}(\BS{t}-\BS{b})^{\BS{\mu}}
.\end{gather*}
The sesquilinear form 
$$
S_{n,\BS{b}}:\BB{C}[\BS{\tau}]\times\CAL{O}_{n,\BS{b}}
\longrightarrow \BB{C},\quad
(p,\,f)\longmapsto S_{n,\BS{b}}\SLF{p}{f}
$$
is defined by
$$
S_{n,\BS{b}}\SLF{p}{f}:=B_{n,\BS{b}}\BLF{p}{u(f)}
.$$
This can be expressed also as
$$
S_{n,\BS{b}}:=B_{n,\BS{b}}\COMP(\id_{\BB{C}[\BS{\tau}]},\, u)
.$$
Explicitly, we have 
$$
S_{n,\BS{b}}\SLFFFF{\sum_{\mbox{\scriptsize finite}} 
a_{\BS{\nu}} \BS{\tau}^{\BS{\nu}}}%
{\sum b_{\BS{\mu}}(\BS{t}-\BS{b})^{\BS{\mu}}}
=\sum \BS{\nu}!{a}_{\BS{\nu}}\bar{b}_{\BS{\nu}}
.$$

The \textit{weak topology} of $\BB{C}[\BS{\tau}]$ with respect 
to $B_{n,\BS{b}}$ 
is the coarsest topology such that the linear functionals 
$$
b_{\| f}:\BB{C}[\BS{\tau}]\longrightarrow\BB{C},\quad
p\longmapsto B_{n,\BS{b}}\BLF{p}{f}
$$
are continuous for all $f\in\CAL{O}_{n,\BS{b}}$. 
Similarly the \textit{weak topology} of $\CAL{O}_{n,\BS{b}}$ 
with respect to $B_{n,\BS{b}}$ 
is the coarsest topology such that the linear functionals 
$$
b_{p\|}:\CAL{O}_{n,\BS{b}}\longrightarrow\BB{C},\quad
f\longmapsto B_{n,\BS{b}}\BLF{p}{f}
$$ 
are continuous for all $p\in \BB{C}[\BS{\tau}]$. 
By these topologies, $\BB{C}[\BS{\tau}]$ and 
$\CAL{O}_{n,\BS{b}}$ become topological vector spaces. 
%%%%%%%%%%%%%
\begin{rem}\label{REMhermitian}
We have chosen the sesquilinear form expressed 
by the diagonal matrix with diagonal elements 
$\BS{\nu}!$ but note that this has not an intrinsic legitimacy. 
This form may be transformed to a different positive 
definite Hermitian matrix in another affine coordinates 
but the weak topologies remain unchanged. 
\end{rem}
%%%%%%%%%%%%%%%%%%

If $L$ is a vector subspace of $\BB{C}[\BS{\tau}]$, 
its annihilator space 
(orthogonal space) with respect to $B_{n,\BS{b}}$ 
is denoted by $L^{\top}$: 
$$
L^{\top}:=\{f:B_{n,\BS{b}}\BLF{p}{f}=0\mbox{ for all }p\in L\}
.$$ 
Similarly, if $K$ is a vector subspace of $\CAL{O}_{n,\BS{b}}$, 
its annihilator space 
with respect to $B_{n,\BS{b}}$ is denoted by $K^{\top}$: 
$$
K^{\top}:=\{p:B_{n,\BS{b}}\BLF{p}{f}=0\mbox{ for all }f\in K\}
.$$ 
These are vector subspaces. 
We adopt the abbreviation $L^{\top\top}=(L^{\top})^{\top}$ and 
$K^{\top\top}=(K^{\top})^{\top}$, etc. Recall that 
$L^{\top\top}$ and $K^{\top\top}$ are the weak closures of 
$L$ and $K$ 
respectively 
(Bourbaki \cite[4, \S 1, n${}^{\circ}$3, 4)]{REFEVT}). 

The \textit{weak topologies} of $\BB{C}[\BS{\tau}]$ and 
$\CAL{O}_{n,\BS{b}}$ 
with respect 
to $S_{n,\BS{b}}$ are defined in a similar way to those 
with respect to 
$B_{n,\BS{b}}$ and they become topological vector spaces. 
The weak topology of $\CAL{O}_{n,\BS{b}}$ is nothing 
but that of 
the coefficient-wise convergence. 
If $L\subset\BB{C}[\BS{\tau}]$ and 
$K\subset\CAL{O}_{n,\BS{b}}$ are vector 
subspaces, their annihilators spaces with respect to 
$S_{n,\BS{b}}$ 
are denoted by $L^{\bot}$, $K^{\bot}$ respectively: 
\begin{gather*}
L^{\bot}:=\{f:S_{n,\BS{b}}\SLF{p}{f}=0\mbox{ for all }p\in L\},
\\
K^{\bot}:=\{p:S_{n,\BS{b}}\SLF{p}{f}=0\mbox{ for all }f\in K\}
.\end{gather*}
These are also vector subspaces. 
We adopt the abbreviation $L^{\bot\bot}=(L^{\bot})^{\bot}$ and 
$K^{\bot\bot}=(K^{\bot})^{\bot}$, etc. 
It is obvious that 
\begin{gather*}
S_{n,\BS{b}}\SLF{p}{f}
=B_{n,\BS{b}}\BLF{p}{u(f)}
=\overline{B_{n,\BS{b}}\BLF{u(p)}{f}}
,\\
B_{n,\BS{b}}\BLF{p}{f}
=S_{n,\BS{b}}\SLF{p}{u(f)}
=\overline{S_{n,\BS{b}}\SLF{u(p)}{f}}
.\end{gather*}
Defining $s_{|f}$, $s_{p|}$ in a similar way as 
$b_{||f}$, $b_{p||}$, 
we have 
$$
s_{|u(f)}=b_{\| f},\quad s_{|f}=b_{\| u(f)},
\quad s_{p|}=b_{p\| }\COMP u,\quad b_{p\| }=s_{p|}\COMP u
.$$ 
Since $u$ is an involution 
(i.e. $u\COMP u$ is the identity) and 
commute with the operations ${}^{\bot}$ and ${}^{\top}$, 
we have the equalities  
$$
\{s_{|f}: f\in\CAL{O}_{n,\BS{b}}\}
=\{b_{|f}: f\in\CAL{O}_{n,\BS{b}}\}, 
\quad 
\{s_{p\|}: p\in\BB{C}[\BS{\tau}]\}=\{b_{p\| }: 
p\in\BB{C}[\BS{\tau}]\}
.$$ 
Hence we have the following.
%%%%%%%%%%%%%%%%%
\begin{prop}\label{PROorthogonal}
The weak topologies on $\BB{C}[\BS{\tau}]$ 
(resp. $\CAL{O}_{n,\BS{b}}$) 
with respect to $B_{n,\BS{b}}$ and $S_{n,\BS{b}}$ coincide.
If $L$ (resp. $K$) is a vector subspace of 
$\BB{C}[\BS{\tau}]$ (resp. 
$\CAL{O}_{n,\BS{b}}$), we have the equalities 
\begin{gather*}
L^{\bot}=u(L^{\top})=(u(L))^{\top},\quad 
K^{\top}=u(K^{\bot})=(u(K))^{\bot},
\\
K^{\bot}=u(K^{\top})=(u(K))^{\top},\quad 
L^{\top}=u(L^{\bot})=(u(L))^{\bot}
\end{gather*}
and hence 
$$
L^{\bot\bot}=L^{\top\top},\quad K^{\top\top}=K^{\bot\bot}
.$$
\end{prop}
%%%%%%%%%%%%%%%%%

By the first assertion, we have no need to refer to forms 
$B_{n,\BS{b}}$ and $S_{n,\BS{b}}$ for weak topologies. 
Now we can deduce a few properties of the space of 
annihilators with 
respect to the sesquilinear forms from those with respect to 
bilinear forms (Bourbaki \cite[4, \S 1, n${}^{\circ} 4$]{REFEVT}). 
%%%%%%%%%%%%%
\begin{cor}\label{CORorthogonal}
If $L$ (resp. $K$) is a vector subspace of 
$\BB{C}[\BS{\tau}]$ (resp. 
$\CAL{O}_{n,\BS{b}}$), we have the following.
\begin{enumerate}
\item
$K^{\bot}$, $L^{\bot}$ are all weakly closed. 
\item
$L^{\bot\bot}$ is the weak closure of $L$ in 
$\BB{C}[\BS{\tau}]$ and 
$K^{\bot\bot}$ is the weak closure of $K$ in 
$\CAL{O}_{n,\BS{b}}$. 
\item
$L^{\bot\bot\bot}=L^{\bot},\quad K^{\bot\bot\bot}=K^{\bot}$.
\end{enumerate}
\end{cor}
%%%%%%%%%%%%%%%%%%%%%%%%%%%%%%%

We can easily prove the following. 
%%%%%%%%%%%%%%%%%
\begin{prop}\label{PROnondegenerate}
The form
$$
S_{n,\BS{b}}:\BB{C}[\BS{\tau}]\times\CAL{O}_{n,\BS{b}}
\longrightarrow \BB{C}
$$
is a non-degenerate sesquilinear form, i.e. 
$\BB{C}[\BS{\tau}]^{\bot}=\{0\}$ and 
$\CAL{O}_{n,\BS{b}}^{\bot}=\{0\}$. 
\end{prop}
%%%%%%%%%%%%%%%%%%%

Let us follow the notation in \S\ref{sectionLEAST}: 
$$
\BS{\tau}^{\BS{\nu}}:=(\BS{t}-\BS{b})^{\BS{\nu}}_{\BS{b}}\DA 
\in \FRAK{m}_{n,\BS{b}}^{|\BS{\nu}|}/
\FRAK{m}_{n,\BS{b}}^{|\BS{\nu}|+1}
\subset\BB{C}[\BS{\tau}]=\CAL{O}_{n,\BS{b}}\DA
$$
for the elements of the least spaces. 
The following is a most natural and efficient for 
construction of 
a dual space of a finite-dimensional subspace of 
$\CAL{O}_{n,\BS{b}}$. 
%%%%%%%%%%%%%%%%%%%%%%%%%%%%%%%
\begin{thm}({\rm de~Boor-Ron \cite[Theorem 5.8]{REFbr2}})%
\label{THMminimal}
Let $Z$ be a vector subspace of 
$\CAL{O}_{n,\BS{b}}$. 
Then the sesquilinear form 
$$
S_Z:\,Z_{\BS{b}}\DA\times Z\longrightarrow \mathbb{C}
$$
obtained as the restriction of 
$S_{n,\BS{b}}:\,\BB{C}[\BS{\tau}]
\times\CAL{O}_{n,\BS{b}}\longrightarrow \BB{C}$ 
is non-degenerate, i.e. 
$\BB{C}[\BS{\tau}]^{\bot_Z}
=\{0\}$, $\CAL{O}_{n,\BS{b}}^{\bot_Z}=\{0\}$, 
where $\bot_Z$ denotes the space of annihilators 
with respect to $S_Z$. 
\end{thm}
%%%%%%%%%%%%%%%%%%%%%%%

\begin{proof}
Suppose that $f\neq 0$ belongs to $Z$ and annihilates 
$Z_{\BS{b}}\DA$. We can express $f$ as 
$$
f=\sum_{|\BS{\nu}|\ge d}a_{\BS{\nu}}\BS{\tau}^{\BS{\nu}}\quad 
(d:=\ord_{\BS{b}} f<\infty)
.$$ 
Then we have 
$$
0=S_Z({f_{\BS{b}}\DA,f})
=\sum_{|\BS{\nu}|=d}\BS{\nu}! |a_{\BS{\nu}}|^2\neq 0
,$$
a contradiction. 
This proves that $(Z_{\BS{b}}\DA)^{\bot_Z}=\{0\}$. 

Suppose that $p\neq 0$ belongs to $Z_{\BS{b}}\DA$ 
and annihilates $Z$. 
Let $p_{\BS{b}}\UA\neq 0$ denote the highest degree 
homogeneous part of $p$ 
at $\BS{b}$. 
Since $Z_{\BS{b}}\DA$ is generated by homogeneous elements, 
there exists $g\in Z$ such that $p_{\BS{b}}\UA=g_{\BS{b}}\DA$. 
Then 
$$
0<S_n\SLF{p_{\BS{b}}\UA}{p_{\BS{b}}\UA}
=S_n\SLF{p_{\BS{b}}\UA}{g_{\BS{b}}\DA}=S_Z\SLF{p}{g}=0
,$$
a contradiction. This proves that $Z^{\bot}=\{0\}$. 
\end{proof}

Let us recall an elementary fact: 
existence of the \textit{adjoint mapping}. 
%%%%%%%%%%%%%%%%%%
\begin{lem}\label{LEMtrans-adj}
Let 
$S:Q\times Z\longrightarrow\BB{C}$ and 
$S':Q'\times Z'\longrightarrow\BB{C}$ 
be non-degenerate sesquilinear forms defined on products of 
topological vector spaces equipped with the weak topologies 
with respect to $S$ and $S'$ respectively and let 
$\kappa:Z\longrightarrow Z'$ 
be a linear mapping. Then the flowing are equivalent. 
\begin{enumerate}
\item
$\kappa$ is weakly continuous.
\item
There exists a linear mapping 
${}^s\kappa:Q'\longrightarrow Q$ 
(which we call the adjoint of $\kappa$) such that 
$$
\forall p\in Q,\ \forall f\in Z':\ 
S\SLF{\kappa(p)}{f}=S'\SLF{p}{{}^s\kappa(f)}
.$$
\end{enumerate}
The linear mapping 
${}^s\kappa$ which satisfies (2) is unique for $\kappa$. 
Each of these conditions implies that ${}^s\kappa$ is weakly 
continuous and ${}^{ss}\kappa:={}^s(^s\kappa)=\kappa$. 
\end{lem}
%%%%%%%%%%%%%%%%%%%%%%%

\begin{proof}
Let $u$ denote the complex conjugation 
$\sum a_{\BS{\mu}}(\BS{t}-\BS{b})^{\BS{\mu}}\longmapsto
\sum \bar{a}_{\BS{\mu}}(\BS{t}-\BS{b})^{\BS{\mu}}$ and 
define the bilinear forms 
$B:Q\times Z\longrightarrow\BB{C}$ and 
$B':Q'\times Z'\longrightarrow\BB{C}$ 
by 
$$
B\BLF{p}{f}:=S\SLF{p}{u(f)},\quad B'\BLF{p}{f}:=S'\SLF{p}{u(f)}
$$
(cf. the third paragraph of \S \ref{section-weak}). 
If $\kappa$ is weakly continuous, 
it has the weakly continuous transposed 
mapping i.e. there exists a linear mapping 
${}^t\kappa:Q'\longrightarrow Q$ such that 
$$
\forall p\in Q,\ \forall f\in Z':\ 
B\BLF{\kappa(p)}{f}=B'\BLF{p}{{}^t\kappa(f)}
,$$
by Bourbaki \cite[4, \S 4, n${}^{\circ} 1$]{REFEVT}. 
We have only to put ${}^s\kappa:=u\COMP{}^t\kappa\COMP u$. 
The rest follow from the corresponding properties of 
bilinear forms. 
\end{proof}
%%%%%%%%%%%%%%%%%%
\begin{lem}\label{LEMadjoint}
Multiplication by $p(\BS{\tau})$ in $\BB{C}[\BS{\tau}]$ is 
the adjoint 
of the differential operator $u(p)(\BS{\partial}_{\BS{t}})$ 
$(\BS{\partial}_{\BS{t}}
:=({\partial}/{\partial t_1},\dots,{\partial}/{\partial t_n}))$ 
on $\CAL{O}_{n,\BS{b}}$, with respect to $S_{n,\BS{b}}$. 
In particular, multiplication 
by $\tau_i$ is the adjoint of differentiation by $t_i$. 
Similarly, multiplication by $f(\BS{t})$ in 
$\CAL{O}_{n,\BS{b}}$ is 
the adjoint of the infinite differential operator 
$u(f)(\BS{\partial}_{\BS{\tau}})$ $(\BS{\partial}_{\BS{\tau}}
:=({\partial}/{\partial\tau_1},\dots,
{\partial}/{\partial\tau_n}))$ 
on $\BB{C}[\BS{\tau}]$, with respect to $S_{n,\BS{b}}$. 
(Such an operator is valid for polynomials.) 
In particular, multiplication 
by $t_i$ is the adjoint of differentiation by $\tau_i$. 
\end{lem}
%%%%%%%%%%%%%%%%%%%%%%%

\begin{proof} 
First note that these operations are weakly continuous and 
hence they have weakly continuous adjoint by 
Lemma \ref{LEMtrans-adj}. 
The second assertion is obvious from the direct calculations: 
\begin{gather*}
S_{n,\BS{b}}\SLFFFF%
{\sum_{\BS{\lambda}}c_{\BS{\lambda}}\BS{\partial}_{\BS{\tau}}^{\BS{\lambda}}
\sum_{\BS{\nu}}a_{\BS{\nu}}\BS{\tau}^{\BS{\nu}}
}{
\sum_{\BS{\mu}} b_{\BS{\mu}}\BS{t}^{\BS{\mu}}}
=
\sum_{\BS{\nu},\BS{\lambda}} 
(\BS{\lambda}+\BS{\nu})!c_{\BS{\lambda}}
{a}_{\BS{\lambda}+\BS{\nu}}\bar{b}_{\BS{\nu}}
\\
=
\sum_{\BS{\nu}}
\sum_{\BS{\mu} \geq \BS{\nu}} 
\BS{\mu}! {a}_{\BS{\mu}}\bar{b}_{\BS{\nu}}c_{\BS{\mu}-\BS{\nu}}
=
S_{n,\BS{b}}\SLFFFF{\sum_{\BS{\nu}}
a_{\BS{\nu}}\BS{\tau}^{\BS{\nu}}
}{
\sum_{\BS{\lambda}}\OL{c}_{\BS{\lambda}}\BS{t}^{\BS{\lambda}}
\sum_{\BS{\mu}} b_{\BS{\mu}}\BS{t}^{\BS{\mu}}}
,\end{gather*}
where $\geq$ denotes the product order of the order 
of the set of integers. 
The proof of the first inequality is quite similar. 
\end{proof}
%%%%%%%%%%%%%%%%%%%%%%%%%%%%%%%%%%%%%%%%
\section{Weak topologies of analytic algebras}%
\label{section-continuity}
%%%%%%%%%%%%%%%%%%%%%%%%%%%%%%%
Here we recall some topological properties of analytic 
algebras over $\BB{C}$ and their homomorphisms. 
The main reference is Grauert-Remmert \cite{REFgr}.\\

An algebra isomorphic to a factor algebra of 
a convergent power series algebra 
by a proper ideal 
is called \textit{(local) analytic algebra.} 
Take an analytic algebra 
$A:=\CAL{O}_{n,\BS{b}}/I$. This is a local $\BB{C}$-algebra 
in the sense it has a unique maximal ideal $\mathfrak{m}_A$, 
which consists of the residue classes of elements of 
$\FRAK{m}_{n,\BS{b}}$, 
and $A$ is a vector space over the subalgebra 
$\BB{C}\subset A$ such that the canonical homomorphism 
$$
\BB{C}\longrightarrow A/\mathfrak{m}_A,\qquad
\lambda\longmapsto\lambda\cdot 1\mod \mathfrak{m}_A
$$ 
of the fields is an isomorphism. A homomorphism of algebras 
is always assumed to be unitary: $1\mapsto 1$. 
Then any homomorphism 
$\varphi:B\rightarrow A$ of local $\mathbb{C}$-algebras is local: $\varphi(\mathfrak{m}_B)\subset\mathfrak{m}_A$. 

Following Grauert-Remmert \cite{REFgr}, let us see 
that the weak topology on $A$ is independent of the expression 
$A=\CAL{O}_{n,\BS{b}}/I$. 
Let $\pi_i:A\longrightarrow A/\mathfrak{m}_A^i$ 
$(i\in\BB{N})$ be the factor epimorphism. 
The analytic algebra $A/\FRAK{m}_A^i$ is 
a finite-dimensional $\BB{C}$-vector space and 
hence it has a unique 
structure of a topological vector space by 
Bourbaki \cite[1, \S 2, n${}^{\circ}3$ Theorem 2]{REFEVT}. 
We give $A$ the coarsest topology such that all the $\pi_i$ 
$(i\in\BB{N})$ are continuous 
and call it the \textit{projective topology}. 
Of course this is independent of 
the expression $\CAL{O}_{n,\BS{b}}/I$. Note that the projective 
topology is expressed as ``schwache Topologie" in 
\cite{REFgr}, which will be proved to be the same as our 
weak topology by the following lemma. 

%%%%%%%%%%%%%%%%%%%%%%%
\begin{lem}\label{LEMstrict}
Let $I\subset\CAL{O}_{n,\BS{b}}$ be an ideal and 
$A:=\CAL{O}_{n,\BS{b}}/I$ 
the analytic algebra. Then we have the following.
\begin{enumerate}
\item
On a regular analytic algebra $\CAL{O}_{n,\BS{b}}$, 
the projective topology and the weak topology coincide. 
\item
The ideal $I$ is closed with respect to both 
the projective topology and the weak one with respect to 
$S_{n,\BS{b}}:\BB{C}[\BS{\tau}]
\times\CAL{O}_{n,\BS{b}}\longrightarrow \BB{C}$. 
\item
The sesquilinear form
$S_{n,\BS{b}}$ induces a non-degenerate one 
$$
S_A:\,I^{\bot_n}\times A\longrightarrow \BB{C}
.$$
\item
The weak topology of $I^{\bot_n}$ (resp. of $A$) with respect 
to this $S_A$ coincides with the relative one from 
$\CAL{O}_{n,\BS{b}}$ 
(resp. the factor one of $\CAL{O}_{n,\BS{b}}$). 
\item
The projective topology of $A:=\CAL{O}_{n,\BS{b}}/I$ 
coincides with the factor topology of $\CAL{O}_{n,\BS{b}}$ 
with the projective topology. 
\item
The weak topology and the projective topology on $A$ coincide 
and the factor epimorphism 
$\CAL{O}_{n,\BS{b}}\longrightarrow A$ is 
always a weakly continuous and open mapping.
\end{enumerate}
\end{lem}
% \vspace*{2ex}
%%%%%%%%%%%%%%%%%%%%%%%%%

\begin{proof}
The first assertion is easy to see.  Closeness is proved in 
Grauert-Remmert \cite[II, \S 1, Satz 2]{REFgr} 
and hence we have 
$I=I^{\bot_n\bot_n}$. 
Then $S_{n,\BS{b}}$ induces a non-degenerate sesquilinear 
form $S_A$ on 
$I^{\bot_n}\times A\cong I^{\bot_n}\times
(\CAL{O}_{n,\BS{b}}/I^{\bot_n\bot_n})$ 
by Bourbaki \cite[4, \S 1, n${}^{\circ} 5$, Proposition 5]{REFEVT}. 
This pairing defines a weak topology on $A$. 
It coincides with the factor topology 
of $\CAL{O}_{n,\BS{b}}$ by 
\cite[4, \S 1, n${}^{\circ} 5$, Prop 7]{REFEVT}. 
The weak topology on $I^{\bot_n}$ 
coincides with the relative one by 
\cite[4, \S 1, n${}^{\circ} 5$, Prop 6]{REFEVT}. 
The projective topology on $A$ is proved to coincide with the 
factor topology of $\CAL{O}_{n,\BS{b}}$ by 
\cite[II, \S 1, Satz 3]{REFgr}. The assertions (1), (4) and (5) imply 
that two topologies 
on $A$ coincides. Continuity and openness are 
obvious properties 
of a factor morphism of topological groups. 
\end{proof}
%%%%%%%%%%%%%%%%%%%%%%%
\begin{cor}\label{CORhomomorphism}
Let $\varphi:B\longrightarrow A$ be a homomorphism 
of analytic algebras. Then $\varphi$ is weakly continuous. 
\end{cor}
%%%%%%%%%%%%%%%%%%%%%%%%%

\begin{proof}
Suppose that 
$A=\CAL{O}_{m,\BS{b}}/I,\ B=\CAL{O}_{n,\BS{b}}/J$ and let 
$\pi_A:\CAL{O}_{m,\BS{a}}\longrightarrow A,\ 
\pi_B:\CAL{O}_{n,\BS{b}}\longrightarrow B$ 
denote the factor epimorphisms. Then $\varphi$ lifts to 
$\tilde{\varphi}:\CAL{O}_{n,\BS{b}}\longrightarrow
\CAL{O}_{m,\BS{a}}$ 
(see e.g. Grauert-Remmert \cite[II, \S 0, Satz 3)]{REFgr}. 
Since $\tilde{\varphi}$ 
is obtained by substituting for $y_i$ 
elements of the maximal ideal 
of $\CAL{O}_{n,\BS{b}}$, it is easy to see that 
$\tilde{\varphi}$ is weakly continuous. 
Then the inverse image $(\pi_A\COMP\tilde{\varphi})^{-1}(U)$ 
of an open 
subset $U\subset A$ is open by Lemma \ref{LEMstrict}. 
This implies that 
$\varphi^{-1}(U)$ is open again by Lemma \ref{LEMstrict} 
and that $\varphi$ is weakly continuous. 
\end{proof}
%%%%%%%%%%%%%%%%%%%%%%%%%%%%%%%%%%%%%%%%
\section{Projector to a vector subspace}%
\label{section-Bruno}
%%%%%%%%%%%%%%%%%%%%%%%%%%%%%%%
Bos and Calvi 
used the least space of de~Boor and Ron to define 
Taylor projector. 
Their construction can be abstracted as follows. 
Let $Z_{\BS{b}}$ be a finite-dimensional vector subspace of 
the local analytic algebra $\CAL{O}_{n,\BS{b}}$. 
Then there exists 
a retract (projector) 
$\TAY_{Z,\BS{b}}:\,\CAL{O}_{n,\BS{b}}
\longrightarrow Z_{\BS{b}}$ 
of vector spaces. By the results of former sections, 
at a general point $\BS{b}$, the kernel of 
$\TAY_{Z,\BS{b}}$ is an ideal and the space $Z_{\BS{b}}$ has 
a structure of Artinian algebra as a quotient of 
$\CAL{O}_{n,\BS{b}}$. 
\\

%%%%%%%%%%%%%%%%%%%%%%%%%%%%%%%%%
Let $\varphi:\,\CAL{O}_{m,\BS{a}}\longrightarrow
\CAL{O}_{n,\BS{b}}$ be a 
homomorphism of analytic algebras. There exists its adjoint 
$$
{}^s\varphi:\,\BB{C}[\BS{\tau}]=\CAL{O}_{n,\BS{b}}\DA
\longrightarrow \BB{C}[\BS{\xi}]=\CAL{O}_{m,\BS{a}}\DA
$$
by Lemma \ref{LEMtrans-adj}. 
This is nothing but the push-forward 
of derivatives of Dirac delta (a distribution with one point 
support $\{\BS{a}\}$, see Schwartz 
\cite[(III, 10;1), (IX, 4;1)]{REFschwarz}). 
Its concrete forms 
are given by multivariate versions of Fa\`{a} di Bruno formula 
(the formula for the higher order derivatives of composite 
multivariate functions, see e.g. \cite{REFlp}, \cite{{REFma}}). 
Note that the image of a homogeneous element by 
${}^s\varphi$ is 
not always homogeneous. 
It is a very troublesome task to calculate 
this formula by hand. 
% \end{rem}
%%%%%%%%%%%%%%%%%%%%
\begin{defin}\label{DEFprojector}
For a finite-dimensional vector subspace 
$Z_{\BS{b}}\subset\CAL{O}_{n,\BS{b}}$, we have defined 
a non-degenerate sesquilinear form 
$$
S_{Z,\BS{b}}:\,Z_{\BS{b}}\DA\times 
Z_{\BS{b}}\longrightarrow\BB{C}
$$
induced from $S_{n,\BS{b}}$ using some fixed coordinates 
$\BS{t}=(t_1,\dots,t_n)$ (Theorem \ref{THMminimal}). 
The topology of $Z_{\BS{b}}$ (resp. $Z_{\BS{b}}\DA$) 
as a topological vector space is unique 
by Bourbaki \cite[1, \S 2, n${}^{\circ}3$ Theorem 2]{REFEVT}. 
Let $\iota:\,Z_{\BS{b}}\DA\longrightarrow 
\CAL{O}_{n,\BS{b}}\DA$ and 
$\kappa:\, Z_{\BS{b}}\longrightarrow \CAL{O}_{n,\BS{b}}$ 
denote the inclusion mappings. 
Since all linear mappings defined on a 
finite-dimensional space are continuous by 
\cite[1, \S 2, n${}^{\circ} 3$, Corollary 2]{REFEVT}, 
the mappings $\iota$ and $\kappa$ are continuous. 
Then we have the weakly continuous adjoint linear mappings 
$\TAY_{Z,\BS{b}}:={}^s\iota:\,\CAL{O}_{n,\BS{b}}
\longrightarrow Z_{\BS{b}}$ of $\iota$ and 
${}^s\kappa:\,\CAL{O}_{n,\BS{b}}\longrightarrow Z_{\BS{b}}\DA$ 
of $\kappa$ by Lemma \ref{LEMtrans-adj}. 
Thus we have Diagram \ref{tab:2}, 
where the bold vertical lines imply the sesquilinear pairings. 
\end{defin}
%%%%%%%%%%%%%%%%%%%%%%%
\renewcommand{\tablename}{\textsf{diagram}}
\begin{table}[h]
\caption{\scriptsize\textsf{Projector}}\label{tab:2}
{\unitlength=.3em
\hspace*{8ex}
\begin{picture}(40,21)
\thicklines
%%ab^2
\put(0,16){$Z_{\BS{b}}\ \ \hookrightarrow\ \ 
\CAL{O}_{n,\BS{b}}=\BB{C}\{\BS{t}-\BS{b}\}$}
\path(2,5)(2,14)
  \path(2.1,5)(2.1,14)
\path(14,5)(14,14)
  \path(14.1,5)(14.1,14)
\put(0,2){$Z_{\BS{b}}\DA\ \ \hookrightarrow\ \ 
\CAL{O}_{n,\BS{b}}\DA\ =\BB{C}[\BS{\tau}]$}
\put(7.5,3.8){\small$\iota$}
\put(6.5,.5){$\dashleftarrow$}
\put(7,-1.8){\small${}^s\kappa$}
\put(7.2,18.3){\small$\kappa$}
\put(6.4,11.2){\small$\TAY_{Z,\BS{b}}$}
\put(6,14){$\dashleftarrow$}
\put(-4.1,9){\small$S_{Z,\BS{b}}$}
\put(15.2,9){\small$S_{n,\BS{b}}$}
\end{picture}
}\end{table}
%%%%%%%%%%%%%%%%%%%%%%%%%%%
\begin{prop}\label{PROprojector}
Let $Z_{\BS{b}}\subset\CAL{O}_{n,\BS{b}}$ be a 
finite-dimensional 
vector subspace. Then we have the following. 
\begin{enumerate}
\item
The mappings 
$\TAY_{Z,\BS{b}}$ and ${}^s\kappa$ are weakly continuous.
\item
We have the equalities 
\begin{gather*}
\Ker \TAY_{Z,\BS{b}}=(Z_{\BS{b}}\DA)^{\bot_n},\quad 
(\Ker \TAY_{Z,\BS{b}})^{\bot_n}=Z_{\BS{b}}\DA,
\\ 
\Ker {}^s\kappa=(Z_{\BS{b}})^{\bot_n},\quad 
(\Ker {}^s\kappa)^{\bot_n}=Z_{\BS{b}},\quad 
\end{gather*}
where $\bot_n$ 
denotes the subspace of annihilators with respect to 
$S_{n,\BS{b}}$.
\item 
The mappings $\TAY_{\BS{a},d}$ and ${}^s\kappa$ are 
retractions of 
vector spaces, 
i.e. $\TAY_{Z,\BS{b}}\COMP\kappa$ and 
${}^s\kappa\COMP\iota$ 
are the identities. 
\end{enumerate}
\end{prop}
%\vspace*{1ex}
%%%%%%%%%%%%%%%%%%%%%%%%%%

\begin{proof}
The property (1) is already stated above. 
The first equality of 
(2) follows from the fact that the sesquilinear form 
$S_{n,\BS{b}}$ is 
non-degenerate. Since $Z_{\BS{b}}\DA$ is finite-dimensional, 
it is weakly closed and 
$Z_{\BS{b}}\DA=(Z_{\BS{b}}\DA)^{\bot_n\bot_n}$. 
Then the first equality implies the second. 
The rest of (2) are quite similar. 
If $f\in Z_{\BS{b}}$ and $p\in Z_{\BS{b}}\DA$, we have
$$
S_{Z,\BS{b}}\SLF{p}{f}
=
S_{n,\BS{b}}\SLF{\iota(p)}{\kappa(f)}
=
S_{Z,\BS{b}}\SLF{p}{\TAY_{Z,\BS{b}}\COMP\kappa(f)}
.$$ 
Since $S_{Z,\BS{b}}$ is non-degenerate, this implies that 
$\TAY_{Z,\BS{b}}\COMP\kappa$ is the identity. 
Similarly ${}^s\kappa\COMP\iota$ is also the identity. 
\end{proof}

%%%%%%%%%%%%%%%%%%
\begin{rem}\label{RMKinterpolation}
A retraction of a space $F$ of smooth functions has 
a close relation with interpolation. 
Let $Z$ be a finite-dimensional vector subspace of $F$. 
An interpolation is the task to find a function $f\in Z$ 
which, together with its higher order derivatives, 
takes the prescribed values at a finite number of 
points $\BS{b}_i$. 
These regarded quantities, the values of the function itself 
or its higher order derivatives at $\BS{b}_i$, 
are expressed by Schwartz 
distributions $p_i$ $(1\le i\le q)$ supported at $\BS{b}_i$ 
as we have seen in \S \ref{section-weak} 
and can be expressed as elements of $\BB{C}[\BS{\tau}]$. 
For sake of simplicity, 
consider the case the support consists of a single 
point $\BS{b}$. 
If we apply the distributions $p_i$ to 
$f\in \CAL{O}_{n,\BS{b}}$, 
we get ``interpolation data" 
$S_{n,\BS{b}}\SLF{p_i}{f}\in\BB{C}$ 
$(1\le i\le q)$. Take a finite-dimensional space 
$Z_{\BS{b}}\subset\CAL{O}_{n,\BS{b}}$. Since 
$$
S_{n,\BS{b}}\SLF{p_i}{f}
=S_{Z,\BS{b}}\SLF{p_i}{\TAY_{Z,\BS{b}}(f)}
\quad
(f\in \CAL{O}_{n,\BS{b}})
,$$
$f$ and $\TAY_{Z,\BS{b}}(f)$ have the same data 
for the interpolation quantities $p_i\in Z_{\BS{b}}\DA$. 
Thus the data of $f$ is interpolated by an element 
$\TAY_{Z,\BS{b}}(f)\in Z_{\BS{b}}$. 
\end{rem}
%%%%%%%%%%%%%%%%%%%%%
\begin{lem}\label{LEMdclosed-ideal}
Let $Z_{\BS{b}}\subset\CAL{O}_{n,\BS{b}}$ be a 
finite-dimensional 
vector subspace. Then the following conditions 
(1) and (2) are equivalent. 
\begin{enumerate}
\item 
The least space $Z_{\BS{b}}\DA$ is $D$-invariant. 
\item
The kernel $\Ker\TAY_{Z,\BS{b}}=(Z_{\BS{b}}\DA)^{\bot_n}$ 
of $\TAY_{Z,\BS{b}}$ is an ideal. 
\end{enumerate}
\end{lem}
%%%%%%%%%%%%%%%%%%%%

The condition (2) is the property of an 
ideal projector of Birkhoff \cite{REFbirkhoff}. 
This property is important because it assures that 
$\TAY_{Z,\BS{b}}$ induces a factor epimorphism of rings. 
Equivalence of (1) and (2) appears in 
M.~G.~Marinari, H.~M.~M\"{o}ller, T.~Mora \cite{REFmmm}, 
% attribution: confirmed by Gasca-Marinari
Proposition 2.4 and de~Boor-Ron \cite[Proposition 6.1]{REFbr2}. 
The interpolation defined by a projector with this property is 
sometimes called Hermite interpolation but there is other 
proposal to use this word to imply the interpolation with 
more good property (de~Boor-Shekhtman \cite{REFbs}). 
%%%%
\begin{proof}
Suppose that $Z_{\BS{b}}\DA$ is $D$-invariant. 
If $p\in Z_{\BS{b}}\DA$ and $f\in (Z_{\BS{b}}\DA)^{\bot_n}$, 
we have
$$
S_{n,\BS{b}}\SLF{p}{t_if }
=
S_{n,\BS{b}}\SLF{\partial p/\partial \tau_i}{f }=0
\quad(i=1,\dots,n)
$$
by {Lemma \ref{LEMadjoint}}. This implies that 
$t_i(Z_{\BS{b}}\DA)^{\bot_n}\subset(Z_{\BS{b}}\DA)^{\bot_n}$. 
Then for any $g(\BS{t})\in\BB{C}[\BS{t}]$, we have 
$g(\BS{t})(Z_{\BS{b}}\DA)^{\bot_n}
\subset(Z_{\BS{b}}\DA)^{\bot_n}$. 
Taking the weak limit, this holds for $g\in\CAL{O}_{n,\BS{b}}$, 
which implies that $(Z_{\BS{b}}\DA)^{\bot_n}$ is an ideal of 
$\CAL{O}_{n,\BS{b}}$ and completes the proof of 
$(1)\Longrightarrow(2)$. 
In a similar way, we see that, if $(Z_{\BS{b}}\DA)^{\bot_n}$ 
is an ideal of $\CAL{O}_{m,\BS{a}}$, 
then $(Z_{\BS{b}}\DA)^{\bot_n\bot_n}$ is $D$-invariant. 
Since every homogeneous part of $Z_{\BS{b}}\DA$ is 
finite-dimensional and belongs to $Z_{\BS{b}}\DA$, 
we see that $Z_{\BS{b}}\DA$ is weakly 
closed. 
Then we have $(Z_{\BS{b}}\DA)^{\bot_n\bot_n}=Z_{\BS{b}}\DA$ 
and $(2)\Longrightarrow(1)$ follows. 
\end{proof}
%%%%%%%%%%%%%%%%%%%%%
\begin{rem}\label{REMartinian-1}
If the condition (2) holds and if $Z_{\BS{b}}\neq\emptyset$, 
the vector subspace $Z_{\BS{b}}$ has 
a structure of a $\BB{C}$-algebra such that the linear mapping 
$\TAY_{Z,\BS{b}}:\,
\CAL{O}_{n,\BS{b}}\longrightarrow Z_{\BS{b}}$ 
is a factor epimorphism of a $\BB{C}$-algebra 
(\cite[Corollary 3.6]{REFBC-2}). 
Since $\dim_{\BB{C}}Z_{\BS{b}}<\infty$, 
it is a local analytic algebra of Krull dimension 0. 
Then it is \emph{Artinian} in the sense that it satisfies 
the descending chain 
condition of ideals (cf. \cite{REFmatsumura}). 
Thus we have associated to each point of $U_Z^{\,\rm{inv}}$ 
an Artinian local algebra. 
Since all the elements of the maximal ideal are nilpotent, 
$Z_{\BS{b}}$ is not a subalgebra of $\CAL{O}_{n,\BS{a}}$ 
in general. 
\end{rem}
%%%%%%%%%%%%%%%%%%%%%%%%%%%%%%%%%%%%%%%%
\begin{exa}\label{EXA1}
We show a common example. Let us take the vector space 
$$
Z:=\Span_{\BB{C}}
(1,\,s,\,t,\,t^2+st^2,\,t^3)\subset\CAL{O}_2(\BB{R}^2)
.$$
%%%%%%%%%%%%%%%%%%%%%%%
\renewcommand{\tablename}{\textsf{diagram}}
\begin{table}[h]
\caption{\scriptsize\textsf{Bases of jet spaces of 
$Z$ at $\BS{t}=(a,b)$}}\label{tab:3}
\vskip-3.7ex
$$
\begin{array}{|c||c|c|c|c|c|c|c|}
\hline
\multirow{4}{8ex}{\text{jet space}}
&R_0      & \multicolumn{6}{|c|}{}\\
\cline{2-8}
&\multicolumn{3}{|c|}{R_1}&\multicolumn{4}{c|}{}\\
\cline{2-8}
&\multicolumn{5}{|c|}{R_2}&\multicolumn{2}{c|}{}\\
\cline{2-8}
&\multicolumn{7}{|c|}{R_3}\\
\hline
\text{fibre coordinates}
& u_{(0,0)} & u_{(1,0)} & u_{(0,1)} & u_{(1,1)} 
& u_{(0,2)} & u_{(1,2)} & u_{(0,3)} \\
\hline\hline
1       &1        & 0      & 0        & 0    & 0      & 0 & 0  \\
s       &a        & 1      & 0        & 0    & 0      & 0 & 0 \\
t       &b        & 0      & 1        & 0    & 0      & 0 & 0 \\
t^2+st^2&b^2+ab^2 & b^2    & 2(1+a)b  & 2b   & 1+a    & 1 & 0  \\
t^3     &b^3      & 0      & 3b^2     & 0    & 3b     & 0 & 1  \\
\hline
\end{array}
$$
\end{table}
\vskip-2ex
%%%%%%%%%%%%%%%%%%%%%%%%%%%
\noindent
Paying attention to the normalizing factor $1/\BS{\nu}!$ 
of fibre coordinates in \S \ref{section-wronskian}, 
we have Diagram 3 for the jets of basis of $Z$. For example, 
the bottom of the diagram implies 
$j^3(t^3)(\BS{b})
=(\BS{b};b^3,\,0,\,3b^2,\,0,\,0,\,3b,\,0,\,0,\,0,\,1)$ 
at $\BS{b}=(a,b)$. Then we have 
$$
\dim_{\BB{C}}R_k(\BS{b})
=\left\{
\begin{array}{ll}
=   1\quad         & \hfill\qquad(k=0)\phantom{.}\\
=   3              & \hfill\qquad(k=1)\phantom{.}\\
=   3\quad         & (a=-1,\ b=0,\ k=2)\\
=   4\quad         & (a\neq -1,\ b=0,\ k=2)\\
=   5\quad         & \hspace{3.5em}(b\neq 0,\ k=2)\\
=   5\quad         & \hspace{6.4em}(k\ge 3).
\end{array}
\right.
$$
Then, denoting $s'=s-a,\ t'=t-b$, we have the following.
\begin{enumerate}
\item
If $b\neq 0$,
$$
Z_{(a,\,b)}\DA=\Span(1,\sigma,\tau,\sigma\tau,\tau^2),
\quad
(Z_{(a,\,b)}\DA)^{\bot_2}
=\left(s'^2,s't'^2,t'^3\right)\BB{C}\{s',t'\}
.$$
The set of these points form $U^{\,\rm{bdl}}$ of $Z$ and, 
since the kernel $(Z_{(a,\,b)}\DA)^{\bot_2}$ are of 
the same form, 
the associated Artinian algebras are all isomorphic 
at these points. 
\item
If $a\neq -1,\ b=0$,
$$
Z_{(a,\,b)}\DA=\Span\left(1,\sigma,\tau,\tau^2,\tau^3\right),
\quad
(Z_{(a,\,b)}\DA)^{\bot_2}
=\left(s'^2,\,s't',\,t'^4\right)\BB{C}\{s',\,t'\}. 
$$
These points form 
$U^{\,\rm{inv}}\setminus U^{\,\rm{bdl}}$ of $Z$. 
The associated Artinian algebras are all isomorphic at these 
points but not isomorphic to those associated to 
bundle points. 
\item
If $a=-1,\ b=0$, 
\begin{gather*}
Z_{(a,\,b)}\DA
=\Span\left(1,\sigma,\tau,\sigma\tau^2,\tau^3\right),
\\
(Z_{(a,\,b)}\DA)^{\bot_2}
=\left(s'^3,\,s't'^3,\,t'^4\right)\BB{C}\{s',\,t'\}
+\BB{C}s't'+\BB{C}t'^2.
\end{gather*}
The least space of $Z$ is not $D$-invariant here. 
\end{enumerate}
We will see this example again in Example 
\ref{EXAdinv-but-notbdle}. 
\end{exa}
%%%%%%%%%%%%%%%%%%%%%%%%%%%
\section{Intrinsic treatment of least spaces}%
\label{section-intrinsic}
%%%%%%%%%%%%%%%%%%%%%%%%%%%%%%%
Here we show an intrinsic treatment of the least space, 
$D$-invariance property and the Artinian algebra 
associated to a $D$-invariant point. 
These are defined for complex manifolds independently 
of local coordinates 
$\BS{t}=(t_1,\dots,t_n)$. To explain them we distinguish 
the coordinate expression 
$f_{\BS{0}}^{\BS{t}}\DA:=(f\COMP\BS{\Phi})_{\BS{0}}\DA$ 
of the least part 
(resp. the least space $Z_{\BS{b}}^{\BS{t}}\DA$) 
and the intrinsic one $f_{\BS{b}}\DA$ (resp. $Z_{\BS{b}}$) 
of $f\in\CAL{O}_{M,\BS{b}}$ in this section.
\\ 

Let $M$ be an $n$-dimensional complex manifold and $\varphi$ 
a local parametrisation of $M$ centred at $\BS{b}$: 
$\varphi(\BS{0})=\BS{b}$. 
Let $\BS{s}$ be the local coordinates for $\varphi$, then 
$f_{\BS{0}}^{\BS{s}}:=(f\COMP\BS{\Phi})_{\BS{0}}\DA$ is expressed as 
a homogeneous polynomial, say of degree $k$, 
in $\BS{\sigma}:=\BS{s}_{\BS{0}}\DA$. 
If $\BS{\Psi}$ is another local parametrisation of $M$ 
centered at $\BS{b}$ 
with local coordinates $\BS{t}$ and dual coordinates 
$\BS{\tau}:=\BS{t}_{\BS{0}}\DA$. 
Then there exists a biholomorphic 
germ $\BS{\Theta}$ with $\BS{\Phi}=\BS{\Psi}\COMP\BS{\Theta}$. 
Let 
$$
\tilde{j}_{\theta}:\,
\BB{C}[\BS{\tau}]\longrightarrow\BB{C}[\BS{\sigma}],
\quad
p(\BS{\tau})\longmapsto p(J_{\theta}\BS{\sigma})
$$
denotes the isomorphism induced by the linear coordinate 
transformation expressed by the Jacobian matrix 
$J_{\theta}:={\partial(\BS{t})}/{\partial(\BS{s})}$ 
evaluated at $\BS{s}=\BS{0}$, where $\BS{\sigma}$ 
and $\BS{\tau}$ are treated as column vectors. 
By the definition of the least part, we have 
\begin{gather*}
\tilde{j}_{\theta} (f_{\BS{0}}^{\BS{t}}\DA)
=
f_{\BS{0}}^{\BS{t}}\DA\hspace{-.5ex}(J_{\theta}\BS{\sigma})
=
(f\COMP\BS{\Psi})_{\BS{0}}\DA\hspace{-.5ex}
(J_{\theta}\BS{\sigma})
=
((f\COMP\BS{\Psi})(J_{\theta}\BS{s}))_{\BS{0}}\DA
%=
%(f\COMP\BS{\Psi})_{\BS{0}}\DA(\Theta\COMP\BS{s})
\\
=
(f\COMP\BS{\Psi}\COMP\Theta\COMP\BS{s})_{\BS{0}}\DA
=
(f\COMP\BS{\Phi})_{\BS{0}}\DA
=
f_{\BS{0}}^{\BS{s}}\DA
.\end{gather*}
Then the collection 
$$
f_{\BS{b}}\DA
:=
\left\{\bigl(\BS{\Phi},f_{\BS{0}}^{\BS{s}}\DA\bigl):\,
\BS{\Phi}=\BS{\Phi}(\BS{s})
\text{ is a local parametrisation of }M
\text{ at }\BS{b}\right\}
,$$
whose elements are related by $\tilde{j}_{\theta}$ as above, 
can be seen as an evaluation 
at $\BS{b}$ of the $k$-fold symmetric tensor product, 
over the sheaf $\CAL{O}_M$ of holomorphic functions on $M$, 
of the usual cotangent sheaf $\CAL{T}^*$ of $\BB{C}^n$. 
The least part $f_{\BS{0}}^{\BS{s}}\DA\in\BB{C}[\BS{\sigma}]$ 
is a coordinate expression of $f_{\BS{b}}\DA$. 
Hence, we have the intrinsic form of the least space: 
$$
\CAL{O}_{M,\BS{b}}\DA
:=\left\{f_{\BS{b}}\DA:\,f\in\CAL{O}_{M,\BS{b}}\right\}
\cong\bigoplus_k\CAL{T}^*({\BS{b}})^{\odot k}
,$$ 
where the symmetric tensor product $\odot$ is taken 
over $\BB{C}$ (cf. Quillen \cite[p.2]{REFquillen}). 
%%%%%%%%%%%%%%%%%%%%%%%
\begin{defin}\label{DEFintrinsic}
Let $M$ be an $n$-dimensional complex manifold with a local 
parametrisation $\BS{\Phi}=\BS{\Phi}(\BS{s})$ at $\BS{b}$ and 
$Z\subset\CAL{O}_{M,\BS{b}}$ a finite-dimensional vector 
subspace. We denote the set of bundle points of all 
the jet spaces of 
$$
Z_{\BS{b}}^{\BS{s}}
:=
\{(f\COMP\BS{\Phi})_{\BS{0}}:\,f\in Z\}
\subset
\BB{C}\{\BS{s}\}
$$ 
and the set of $D$-invariant points of 
$$
Z_{\BS{b}}^{\BS{s}}\DA
:=
\Span_{\BB{C}}\left(f_{\BS{0}}^{\BS{s}}\DA:\,f\in Z\right)
\subset
\BB{C}[\BS{\sigma}]
$$ 
respectively by $M_Z^{\,\rm{bdl}}$ and $M_Z^{\,\rm{inv}}$. 
\end{defin}
%%%%%%%%%%%%%%%%%%%%%%%
\begin{thm}\label{THMintrinsic}
The sets 
$M_Z^{\,\rm{bdl}}$ and $M_Z^{\,\rm{inv}}$ are well-defined 
independently of the local parametrisation $\varphi$ and 
$M_Z^{\,\rm{bdl}}\subset M_Z^{\,\rm{inv}}$. 
\end{thm}
%%%%%%%%%%%%%%%%%%%%%%

\begin{proof}
Take two local parametrisations $\BS{\Phi}$ and $\BS{\Psi}$ and 
isomorphism $\tilde{j}_{\theta}$ as above. 
Since $\tilde{j}_{\theta}$ is a result of 
a linear change of variables, 
it does not alter the rank of $R_k$ in 
\S \ref{section-wronskian}. 
Then the set of bundle points $M_Z^{\,\rm{bdl}}$ 
of all jet spaces of $Z$ is well-defined independently 
of $\varphi$. 
The transformation $\tilde{j}_{\theta}$ changes a partial 
differentiation to a linear combination of 
partial differentiations. Hence the set $M_Z^{\,\rm{inv}}$ 
of $D$-invariant points for $Z$ is well-defined also. 
By Theorem \ref{THMd-closed}, we have 
$M_Z^{\,\rm{bdl}}\subset M_Z^{\,\rm{inv}}$. 
\end{proof}
The collection of $Z_{\BS{b}}^{\BS{s}}\DA$ for all $\varphi$ 
form the intrinsic least space 
$Z_{\BS{b}}\DA\subset\CAL{O}_{M,\BS{b}}\DA$ at $\BS{b}$. 
%%%%%%%%%%%%%%%%%%%%%%%
\begin{thm}\label{THMvector}
Let $M$ be an $n$-dimensional complex manifold and 
$Z\subset\CAL{O}(M)$ 
be a finite-dimensional vector subspace. 
If $\BS{b}\in M_Z^{\,\rm{inv}}$, 
the vector subspace $Z_{\BS{b}}$ 
has a structure 
$\CAL{O}_{n,\BS{b}}/(Z_{\BS{b}}\DA)^{\bot_n}$
of Artinian algebra, which is unique up to the canonical 
isomorphism (induced by ${}^s\tilde{j}_{\theta}$ in the proof) 
as a contravariant tensor. 
\end{thm}
%%%%%%%%%%%%%%%%%%%%%%%%%%

\begin{proof}
Existence of the structure of an Artinian algebra is 
explained in Remark 
\ref{REMartinian-1}. 
Take two local parametrisations around $\BS{b}\in M$ 
which induce isomorphisms 
$\varphi:\,\CAL{O}_{n,\BS{b}}\longrightarrow\BB{C}\{\BS{s}\}$ 
and 
$\psi:\,\CAL{O}_{n,\BS{b}}\longrightarrow\BB{C}\{\BS{t}\}$ 
such that $\varphi=\theta\COMP\psi$ for 
a third isomorphism $\theta$ 
of algebras. We have the homomorphism 
$\tilde{j}_{\theta}:\BB{C}[\BS{\tau}]
\longrightarrow\BB{C}[\BS{\sigma}]$ 
defined above and its adjoint 
${}^s\tilde{j}_{\theta}:\,
\BB{C}\{\BS{s}\}\longrightarrow\BB{C}\{\BS{t}\}$, 
which is nothing but the homomorphism corresponding 
to the coordinate 
transformation defined by the linear part of $\theta$. 
We have the following implications: 
\begin{gather*}
f\in (Z_{\BS{b}}^{\BS{t}}\DA)^{\bot_n}
\Longleftrightarrow
f\in(\tilde{j}_{\theta}(Z_{\BS{b}}^{\BS{s}}\DA))^{\bot_n}
\Longleftrightarrow
\forall p\in \tilde{j}_{\theta}(Z_{\BS{b}}^{\BS{s}}\DA):\,
S_{n,\BS{0}}\SLF{p}{f}=0
\\
\Longleftrightarrow
\forall q\in Z_{\BS{b}}^{\BS{s}}\DA:\,
S_{n,\BS{0}}\SLFF{\tilde{j}_{\theta}(q)}{f}=0
\Longleftrightarrow
\forall q\in Z_{\BS{b}}^{\BS{s}}\DA:\,
S_{n,\BS{0}}\SLFF{q}{{}^s\tilde{j}_{\theta}(f)}=0
\\
\Longleftrightarrow
{}^s\tilde{j}_{\theta}(f)\in(Z_{\BS{b}}^{\BS{s}}\DA)^{\bot_n}
\Longleftrightarrow
f\in({}^s\tilde{j}_{\theta})^{-1}
((Z_{\BS{b}}^{\BS{s}}\DA)^{\bot_n})
.\end{gather*}
This proves 
$$
(Z_{\BS{b}}^{\BS{t}}\DA)^{\bot_n}
=({}^s\tilde{j}_{\theta})^{-1}
((Z_{\BS{b}}^{\BS{s}}\DA)^{\bot_n})
$$ 
and $(Z_{\BS{b}}^{\BS{s}}\DA)^{\bot_n}$ is the image of 
$(Z_{\BS{b}}^{\BS{t}}\DA)^{\bot_n}$ by the isomorphism 
${}^s\tilde{j}_{\theta}$. 
This isomorphism is obtained by substituting $\BS{s}$ by 
${}^sJ_{\theta}\BS{t}$ and hence it is even 
an algebra isomorphism. 
Then $(Z_{\BS{b}}^{\BS{s}}\DA)^{\bot_n}$ is an ideal 
if and only if 
$(Z_{\BS{b}}^{\BS{t}}\DA)^{\bot_n}$ is so and hence 
$Z_{\BS{b}}^{\BS{s}}\DA$ 
is $D$-invariant if and only if $Z_{\BS{b}}^{\BS{t}}\DA$ is 
so by Lemma \ref{LEMdclosed-ideal}. 
When this is the case, we have an isomorphism 
$$
\BB{C}\{\BS{s}\}/(Z_{\BS{b}}^{\BS{s}}\DA)^{\bot_n}
\longrightarrow 
\BB{C}\{\BS{t}\}/(Z_{\BS{b}}^{\BS{t}}\DA)^{\bot_n}
$$ 
of algebras induced by ${}^s\tilde{j}_{\theta}$. 
\end{proof}
%%%%%%%%%%%%%%%%%%%%%
\begin{thm}\label{THMdmorphism}
Let $\varphi:\,\CAL{O}_{m,\BS{a}}:=\BB{C}\{\BS{x}\}
\longrightarrow\CAL{O}_{n,\BS{b}}:=\BB{C}\{\BS{t}\}$ be 
a homomorphism of analytic algebras and let 
${}^s\varphi:\,\BB{C}[\BS{\tau}]\longrightarrow
\BB{C}[\BS{\xi}]$ 
denote its adjoint mapping with respect to the sesquilinear 
forms $S_{m,\BS{a}}$ and $S_{n,\BS{b}}$. 
If $Q\subset\BB{C}[\BS{\tau}]$ 
is a finite-dimensional vector subspace, 
$\varphi$ induces a monomorphism $\psi$ of 
the factor vector spaces: 
$$
\psi:\,\CAL{O}_{m,\BS{a}}/({}^s\varphi(Q))^{\bot_m}
\longrightarrow
\CAL{O}_{n,\BS{b}}/Q^{\bot_n}
.$$
Furthermore we have the following.
\begin{enumerate}
\item
If $Q$ is $D$-invariant, so is ${}^s\varphi(Q)$ and $\psi$ is 
a monomorphism of Artinian algebras. 
\item
If $\varphi$ is an epimorphism, $\psi$ is an isomorphism 
and ${}^s\varphi(Q)$ is $D$-invariant if and only if $Q$ is so. 
If this is the case, then 
$\psi$ is an isomorphism of Artinian algebras. 
\end{enumerate}
\end{thm}
%%%%%%%%%%%%%%%%%%%%

\begin{proof}
To prove that $\psi$ exists and it is a monomorphism, 
we have only to prove the equality 
$({}^s\varphi(Q))^{\bot_m}=\varphi^{-1}(Q^{\bot_n})$. 
This is obvious by 
\begin{gather*}
f\in({}^s\varphi(Q))^{\bot_m}
\Longleftrightarrow
\forall q\in Q:\ S_{m,\BS{a}}\SLF{{}^s\varphi(q)}{f}=0
\\
\Longleftrightarrow
\forall q\in Q:\ S_{n,\BS{b}}\SLF{q}{\varphi(f)}=0
\Longleftrightarrow
f\in\varphi^{-1}(Q^{\bot_n})
.\end{gather*}
\begin{enumerate}
\item
If $Q$ is $D$-invariant, $Q^{\bot_n}$ is a proper ideal 
by Theorem \ref{LEMdclosed-ideal}. 
The algebra $\CAL{O}_{n,\BS{b}}/Q^{\bot_n}$ is 
Artinian because 
$\dim_{\BB{C}}\CAL{O}_{n,\BS{b}}/Q^{\bot_n}
=\dim_{\BB{C}}Q<\infty$ by 
\cite[4, \S 1, n${}^{\circ}5$, Proposition 5]{REFEVT}. Then 
$({}^s\varphi(Q))^{\bot_m}=\varphi^{-1}(Q^{\bot_n})$ is also 
an ideal and $\psi$ is a monomorphism of algebras. Then 
$\CAL{O}_{m,\BS{a}}/({}^s\varphi(Q))^{\bot_m}$ is 
also Artinian. 
\item
If $\varphi$ is an epimorphism, it is obvious that $\psi$ is an 
isomorphism and that the equality 
$({}^s\varphi(Q))^{\bot_m}=\varphi^{-1}(Q^{\bot_n})$ is 
equivalent to \\
$\varphi\big(({}^s\varphi(Q))^{\bot_m}\bigr)=Q^{\bot_n}$. 
Then, if $({}^s\varphi(Q))^{\bot_m}$ is an ideal, 
$Q^{\bot_n}$ is also so. 
This proves equivalence of $D$-invariant properties of 
${}^s\varphi(Q)$ and $Q$. The last assertion is trivial by (1).
\end{enumerate}
\end{proof}

%%%%%%%%%%%%%%%%%%%%%%%%%%%%%%%%%%%%%%%%%%
\section{Higher order tangents of Bos and Calvi}%
\label{section-bos-calvi-tg}
%%%%%%%%%%%%%%%%%%%%%%%%%%%%%%%
Following the method of Bos-Calvi, we introduce the space 
$D_{\BS{a}}^{\varphi,d}$ of higher order 
tangents of complex submanifold $X$ of an open subset 
$\Omega\subset\BB{C}^m$ at $\BS{a}\in X$. 
It is not an intrinsic 
object associated to $X_{\BS{a}}$ as a germ of 
a complex space but 
it reflects the properties of the embedding germ 
$X_{\BS{a}}\subset\BB{C}_{\BS{a}}^m$. It is also dependent upon 
the choice of the local parametrisation $\varphi$ of 
$X_{\BS{a}}$ in general. 
It is a dual space of the space $P^d(X_{\BS{a}})$ of the 
polynomial functions of degrees at most $d$. We will often skip 
the modifier ``higher order" for tangents.\\

Let $X_{\BS{a}}$ be the germ of a regular complex submanifold 
$X$ of an open subset $\Omega\subset\BB{C}^m$ at $\BS{a}$. 
The algebra $\CAL{O}_{X,\BS{a}}$ of germs of 
holomorphic functions 
on respective neighbourhoods (in $X$) of $\BS{a}$ at $\BS{a}$ is 
isomorphic to 
the factor algebra of $\CAL{O}_{m,\BS{a}}=\BB{C}\{\BS{x}\}$ 
$(\BS{x}:=(x_1,\dots,x_m))$ by the ideal $I_{\BS{a}}$ of 
convergent power series vanishing on the germ 
$X_{\BS{a}}$: 
$\CAL{O}_{X,\BS{a}}\cong\CAL{O}_{m,\BS{a}}/I_{\BS{a}}$. 
Hence $\CAL{O}_{X,\BS{a}}$ is an analytic local algebra. 
Let 
$$
\pi:\CAL{O}_{m,\BS{a}}\longrightarrow \CAL{O}_{X,\BS{a}}
$$
denotes the factor epimorphism. 
By the assumption that $X$ is a submanifold, 
there is an isomorphism 
$$
\psi:\,\CAL{O}_{X,\BS{a}}\longrightarrow\CAL{O}_{n,\BS{b}}
=\BB{C}\{\BS{t}-\BS{b}\}
\quad 
(\BS{t}:=(t_1,\dots,t_n),\ \dim X=n)
.$$
Then we have the epimorphism 
$$
\varphi:=\psi\COMP\pi:\,
\CAL{O}_{m,\BS{a}}\longrightarrow\CAL{O}_{n,\BS{b}}
.$$ 
This is just the epimorphism defined by the pullback 
by the germ 
of embedding 
$$
\BS{\Phi}=(\Phi_1,\dots,\Phi_m):\,\BB{C}^n_{\BS{b}}
\longrightarrow \BB{C}^m_{\BS{a}}
,$$ 
namely $\varphi(f)=f\COMP\BS{\Phi}$. 
We call this $\varphi$ or $\BS{\Phi}$ a 
\textit{local parametrisation} 
of $X$ at $\BS{a}$. 
Let $\BB{C}[\BS{\Phi}]\subset\CAL{O}_{n,\BS{b}}$ 
denote the algebras 
of pullbacks of $\BB{C}[\BS{x}]$ by $\varphi$. Let 
$$
P(X_{\BS{a}})=\BB{C}[\BS{x}]|_{X_{\BS{a}}}
=\pi(\BB{C}[\BS{x}])\subset\CAL{O}_{X,\BS{a}}
$$ 
denote the ring of germs of polynomial function on 
$X_{\BS{a}}$. It is easy to see the following. 
%%%%%%%%%%%%%%%%%%%%%%%%
\begin{lem}\label{LEMpolynomial}
We have the algebra isomorphism
$$
\BB{C}[\BS{\Phi}]:=\varphi(\BB{C}[\BS{x}])
=\psi(P(X_{\BS{a}}))\cong P(X_{\BS{a}})
.$$
\end{lem}
%%%%%%%%%%%%%%%%%%%%%

By the general property of the transposed mapping of a 
surjective 
homomorphism, ${}^t\varphi$ is injective and hence 
${}^s\varphi:\BB{C}[\BS{\tau}]\longrightarrow\BB{C}[\BS{\xi}]$ 
is also so. 
The image $D_{\BS{a}}^{\varphi}
:={}^s\varphi(\BB{C}[\BS{\tau}])$ 
is geometrically the space of higher order tangents of $X$ 
at $\BS{a}$. 
The property $D_{\BS{a}}^{\varphi}=I_{\BS{a}}^{\bot_m}$ 
shown below means 
that higher order tangents of $X_{\BS{a}}$ are just the higher 
order tangents of $\BB{C}^m_{\BS{a}}$ 
which annihilate all the functions vanishing on $X_{\BS{a}}$. 
%%%%%%%%%%%%%%%%%%%%%%%. 
%We call its element ${}^s\varphi(p)$ the 
%\textit{Bos-Calvi (higher order) tangent}. 
The sesquilinear form $S_{n,\BS{b}}$ induces a non-degenerate 
sesquilinear form
$$
{S}_{X,\BS{a}}^{\varphi}:D_{\BS{a}}^{\varphi}
\times \CAL{O}_{X,\BS{a}}\longrightarrow \BB{C}
$$
through $\psi$ and ${}^s\psi$ (see Diagram \ref{tab:4} in 
\S \ref{section-taylorian}).  
Let $I_{\BS{a}}^{\bot_m}$ and $I_{\BS{a}}^{\bot_X}$ denote 
the space of annihilators of $I_{\BS{a}}$ 
with respect to $S_{m,\BS{a}}^{\varphi}$ and 
$S_{X,\BS{a}}^{\varphi}$ respectively. 

%%%%%%%%%%%%%%%%%%%%%%%%%%%%%%%%%
\begin{prop}\label{PROindependence}
We have $(D_{\BS{a}}^{\varphi})^{\bot_m}=I_{\BS{a}}$ and 
$I_{\BS{a}}^{\bot_m}=D_{\BS{a}}^{\varphi}$. 
Hence $D_{\BS{a}}^{\varphi}$ is independent of the local 
parametrisation $\varphi$. 
\end{prop}
%%%%%%%%%%%%%%%%%%%%%%%%%

Thus we may omit the superscript $\varphi$ of 
$D_{\BS{a}}^{\varphi}$. 

\begin{proof}
The first equality follows from the following implications.
\begin{gather*}
f\in I_{\BS{a}}
\Longleftrightarrow
\varphi(f)=0
\Longleftrightarrow
\forall p\in\BB{C}[\BS{\tau}]:\,
S_{n,\BS{b}}\SLF{p}{\varphi(f)}=0
\\
\Longleftrightarrow
\forall p\in\BB{C}[\BS{\tau}]:\,
S_{m,\BS{a}}\SLF{{}^s\varphi(p)}{f }=0
\Longleftrightarrow
f\in({}^s\varphi(\BB{C}[\BS{\tau}]))^{\bot_m}
\Longleftrightarrow
f\in(D_{\BS{a}}^{\varphi})^{\bot_m}
\end{gather*}

Then we have 
$I_{\BS{a}}^{\bot_m}=\left(D_{\BS{a}}^{\varphi}\right)^{\bot_m\bot_m}$. 
To see the equality $I_{\BS{a}}^{\bot_m}=D_{\BS{a}}^{\varphi}$, 
we have only to prove that 
$(D_{\BS{a}}^{\varphi})^{\bot_m\bot_m}=D_{\BS{a}}^{\varphi}$ 
or that 
$D_{\BS{a}}^{\varphi}={}^s\varphi(\BB{C}[\BS{\tau}])$ is 
weakly closed 
in $\BB{C}[\BS{\xi}]$. 
Since $\varphi$ is an open continuous epimorphism by 
{Lemma \ref{LEMstrict}}, 
${}^t\varphi(\BB{C}[\BS{\tau}])$ is weakly closed 
by Bourbaki \cite[4, \S 4, n${}^{\circ}1$, Proposition 4]{REFEVT}.  
Since the complex conjugations 
$u:\,\BB{C}[\BS{\tau}]\longrightarrow\BB{C}[\BS{\tau}]$ 
and 
$u:\,\BB{C}[\BS{\xi}]\longrightarrow\BB{C}[\BS{\xi}]$ are 
homeomorphisms, 
${}^s\varphi(\BB{C}[\BS{\tau}])
=u\COMP{}^t\varphi\COMP u(\BB{C}[\BS{\tau}])$ 
is also weakly closed. 
\end{proof}

Let 
$$
P^d(X_{\BS{a}}):=\{p\mod I_{\BS{a}}:\,p
\in\BB{C}[\BS{x}],\ \deg p\le d\}
\subset P(X_{\BS{a}})\subset\CAL{O}_{X,\BS{a}}
$$
denote the vector space of \textit{polynomial functions} on 
$X$ of degree at most $d$ at $\BS{a}$.

%%%%%%%%%%%%%%%%%
\begin{rem}\label{RMKjustify}
If $X_{\BS{a}}$ is defined by an ideal 
$I_{\BS{a}}\subset\CAL{O}_{m,\BS{a}}$, the algebra 
$$
\BB{C}[\BS{x}]/
\bigl((I_{\BS{a}}+\FRAK{m}_{\BS{a}}^{d+1})\cap
\BB{C}[\BS{x}]\bigr)
\cong
\BB{C}\{\BS{x}\}/
(I_{\BS{a}}+\FRAK{m}_{\BS{a}}^{d+1})
$$
is not isomorphic the vector space $P^d(X_{\BS{a}})$ 
in general. The canonical mapping 
$$
\pi_d:\,P^d(X_{\BS{a}})\longrightarrow
\BB{C}[\BS{x}]/
\bigl((I_{\BS{a}}+\FRAK{m}_{\BS{a}}^{d+1})
\cap\BB{C}[\BS{x}]\bigr)
$$
is surjective but not always injective. 
\end{rem}
%%%%%%%%%%%%%%%%%%%%%

Let
$$ 
\BB{C}[\BS{x}]^d:=\{f(\BS{x}):f\in \BB{C}[\BS{x}],\ 
\deg f\le d\}
\subset\BB{C}[\BS{x}]
$$
denote the vector space of polynomials of degrees at most $d$. 
If we put 
$$
\BB{C}[\BS{\Phi}]^d:=\varphi(\BB{C}[\BS{x}]^d)
=\psi(P^d(X_{\BS{a}}))
\subset\CAL{O}_{n,\BS{b}}
$$
using a local parametrisation $\BS{\Phi}$, 
we have an increasing sequence of 
finite-dimensional vector subspaces 
$$
\BB{C}=\BB{C}[\BS{\Phi}]^0\subset \BB{C}[\BS{\Phi}]^1
\subset\cdots
$$
of the $\BB{C}$-algebra $\BB{C}[\BS{\Phi}]
\subset\CAL{O}_{n,\BS{b}}$. 
Then we have a sequence 
$$
\BB{C}=\BB{C}[\BS{\Phi}]^0_{\BS{b}}\DA
\subset \BB{C}[\BS{\Phi}]^1_{\BS{b}}\DA\subset\cdots
$$
of finite-dimensional vector subspaces of 
$\BB{C}[\BS{\tau}]=\CAL{O}_{n,\BS{b}}\DA$. 
%%%%%%%%%%%%%%%%%%%
Let us fix the degree $d$ hereafter. Recall that 
$S_{n,\BS{b}}$ induces a 
non-degenerate sesquilinear form 
$$
S^{\varphi,d}_{n,\BS{b}}:
\BB{C}[\BS{\Phi}]^d_{\BS{b}}\DA\times\BB{C}[\BS{\Phi}]^d
\longrightarrow\BB{C}
$$
by Theorem \ref{THMminimal}. 
%%%%%%%%%%%%%%%%%%%%%%
\begin{defin}\label{DEFbctg}
Let us put 
$$
D_{\BS{a}}^{\varphi,d}
:={}^s\varphi(\BB{C}[\BS{\Phi}]^d_{\BS{b}}\DA)
\subset D_{\BS{a}}
.$$ 
Since there is a natural isomorphism 
$\psi|_{P^d(X_{\BS{a}})}:\,P^d(X_{\BS{a}})
\longrightarrow\BB{C}[\BS{\Phi}]^d$ 
by Lemma \ref{LEMpolynomial}, we see that 
$S^{\varphi,d}_{n,\BS{b}}$ induces 
a non-degenerate sesquilinear form
$$
{S}_{X,\BS{a}}^{\varphi,d}:
D_{\BS{a}}^{\varphi,d}\times P^d(X_{\BS{a}})
\longrightarrow \BB{C}
.$$
We call the elements of $D_{\BS{a}}^{\varphi,d}$ 
\textit{Bos-Calvi tangents of $X_{\BS{a}}$ of dual degree} $d$ 
(see Diagram \ref{tab:4} in \S \ref{section-taylorian}). 
\end{defin}
%%%%%%%%%%%%%%%%%%%%%%

Usually, some element of $D_{\BS{a}}^{\varphi,d}$ has 
a degree higher 
than $d$ as we will see in the examples below. 
%%%%%%%%%%%%%%%%%%%%%%%%%%%
\begin{lem}\label{LEMcoordtr}
Let $X_{\BS{a}}$ be the germ of a regular complex 
submanifold $X$ 
of an open subset $\Omega\subset\BB{C}^m$ at $\BS{a}$. 
Take two local parametrisations 
$$
\BS{\Phi}:\,\BB{C}_{\BS{b}}^n\longrightarrow\BB{C}_{\BS{a}}^m,
\qquad
\BS{\Psi}:\,\BB{C}_{\BS{b}'}^n\longrightarrow\BB{C}_{\BS{a}}^m
$$
of $X_{\BS{a}}$. Let 
$\BS{\Theta}:\,
\BB{C}_{\BS{b}}^n\longrightarrow\BB{C}_{\BS{b}'}^n$ 
denote the biholomorphic germ (coordinate change) such that 
$\BS{\Phi}=\BS{\Psi}\COMP\BS{\Theta}$ and 
$\varphi=\theta\COMP\psi$. 
Then we have the following. 
\begin{enumerate}
\item
The isomorphism $\theta$ induces isomorphsm 
$\theta|_{\BB{C}[\BS{\Psi}]^d}:\,
\BB{C}[\BS{\Psi}]^d\longrightarrow\BB{C}[\BS{\Phi}]^d$.  
\item
If $\BS{b}$ is a bundle points of all the jet spaces of 
$\BB{C}[\BS{\Phi}]^d$, it is the same for 
$\BB{C}[\BS{\Psi}]^d$. 
If $\BS{b}$ is an $D$-invariant point of 
$\BB{C}[\BS{\Phi}]_{\BS{b}}^d\DA$, 
it is the same for $\BB{C}[\BS{\Psi}]_{\BS{b}}^d\DA$. 
\end{enumerate}
\end{lem}
%%%%%%%%%%%%%%%%%%%%

\begin{proof}
There are isomorphisms 
$$
\BB{C}[\BS{\Phi}]^d
\overset{\cong}{\longrightarrow}P^d(X_{\BS{a}})
\overset{\cong}{\longleftarrow}\BB{C}[\BS{\Psi}]^d
$$
obtained as a restriction of the isomorphism in 
Lemma \ref{LEMpolynomial}. 
Since 
$$
\theta|_{\BB{C}[\BS{\Psi}]^d}:\,
\BB{C}[\BS{\Psi}]^d\longrightarrow\BB{C}[\BS{\Phi}]^d
$$ 
is compatible to these mappings, it is also an isomorphism. 
The assertion (2) follows from Theorem \ref{THMintrinsic}. 
\end{proof}

%%%%%%%%%%%%%%%%%%%%%%%%%%%%%%%%%%%%%%%%%
\begin{exa}\label{EXAdifference}
%%%%%%%%%%%%%%%%%%%%%%%%%%%%%
If $\dim X\ge 2$, $D_{\BS{a}}^{\varphi,d}$ is very sensitive to 
a change of the local parametrisation even in the case $d=1$. 
Let us consider the surface $X\subset\BB{C}^3$ defined by 
$x_3=x_2^2$. Take two global parametrisations: 
$$
\varphi:\,x_1=s_1,\ x_2=s_2,\ x_3=s_2^2;
$$$$
\psi:\,x_1=t_1+t_2,\ x_2=t_2,\ x_3=t_2^2
.$$
These are related by a simple linear transformation of local 
coordinates: $t_1=s_1-s_2$, $t_2=s_2$. 
The dimensions of the spaces $R_k(\BS{b})$ 
$(\BS{a}=\varphi(\BS{b}))$ for $\BB{C}[\BS{\Phi}]^1$ 
defined in \S \ref{section-wronskian} are independent of 
$\BS{b}\in\BB{R}^2$:
$$
\dim_{\BB{C}}R_0(\BS{b})=1,\ \dim_{\BB{C}}R_1(\BS{b})=3,\ 
\dim_{\BB{C}}R_k(\BS{b})=4\ (k\ge 2)
.$$
This means that all points are 
bundle points of all the jet spaces of $\BB{C}[\BS{\Phi}]^1$. 
Then $\BB{C}[\BS{\Psi}]^1$ has also the bundle property 
everywhere by Lemma \ref{LEMcoordtr}. Let us put 
$$
\sigma_i:=s_{i,\BS{b}}\DA,\quad \tau_i:=t_{i,\BS{b}}\DA,\quad 
\xi_i:=x_{i,\BS{a}}\DA
.$$ 
The least spaces of the pullbacks of polynomials of degree at 
most 1 with respect to them are 
\begin{gather*}
\BB{C}[\BS{\Phi}]^1_{\BS{b}}\DA=
\Span_{\BB{C}}(1,\ \sigma_1,\ \sigma_2,\ \sigma_2^2)
,\quad
\BB{C}[\BS{\Psi}]^1_{\BS{b}}\DA=
\Span_{\BB{C}}(1,\ \tau_1,\ \tau_2,\ \tau_2^2),
\end{gather*}
and they are $D$-invariant everywhere (which is also a 
consequence of Theorem \ref{THMd-closed}). 
The push-forwards of these bases are computed as follows. 
%We got the following immediately using a computer 
%algebra system: 
%%%%%%%%%%%%
\begin{gather*}
{}^s\varphi(1)
=
1,\quad
%%%%%%%%%
{}^s\varphi(\sigma_1)
=
\xi_1,\quad
%%%%%%%%%
{}^s\varphi(\sigma_2)
=
\xi_2
+2 s_2 \xi_3
\\
%%%%%%%%%%
{}^s\varphi(\sigma_2^2)
=
2 \xi_3+\xi_2^2
+4 s_2 \xi_2 \xi_3+4 s_2^2 \xi_3^2;
\end{gather*}
%%%%%%%%%%%%%%%%%%%%%
\begin{gather*}
{}^s\psi(1)
=
1,\quad
%%%%%%%%%
{}^s\psi(\tau_1)=
\xi_1,
\quad
{}^s\psi(\tau_2)
=
\xi_1+\xi_2+2t_2\xi_3,
\\
{}^s\psi(\tau_2^2)=
2\xi_3
+\xi_1^2
+2\xi_1\xi_2
+4t_2\xi_1\xi_3
+\xi_2^2
+4t_2\xi_2\xi_3
+4t_2^2\xi_3^2.
\end{gather*}
%%%%%%%%%%%%%%
The monomial $\xi_1^2$ appears in the linear span of the latter 
but not in that of the former. Hence two spaces of 
Bos-Calvi tangents 
$D_{\BS{a}}^{\varphi,1}$ and $D_{\BS{a}}^{\psi,1}$ at 
$\BS{a}=\BS{\Phi}(\BS{b})=\BS{\Psi}(\BS{b})\in X$ 
are different. 
\end{exa}
%%%%%%%%%%%%%%%%%%%
\begin{rem}\label{covariant}
In this paper let us use the word ``contravariant" 
to mean that the objects are mapped in the opposit direction to 
the geometric mapping of underlying complex spaces. 
This is an intrinsic usage and it is equivalent to 
``covariant in 
the classical sense" which pays attension to change of 
components (coefficients). 
The inconsistency in Example \ref{EXAdifference} 
originates from 
the treatment of the elements of the least space. 
An element of $\BB{C}[\BS{t}]_{\BS{b}}^1\DA$ 
is defined as a tensor of cotangents, a contravariant object. 
Then we identify it as a higher order tangent, 
a covariant one, using a positive 
sesquilinear form (see Remark \ref{REMhermitian}) 
and send it to $\BS{C}[\BS{\xi}]$ as a higher order tangent. 
Thus our higher order tangents are not geometric 
objects. 
\end{rem}
%%%%%%%%%%%%%%%%%%%%%

%%%%%%%%%%%%%%%%%%%%%%%%%%%%%%%%%%%%
\section{Taylor projector}\label{section-taylorian}
%%%%%%%%%%%%%%%%%%%%%%%%%%%%%%%%%%%%%
Now we can introduce the Taylor projector of degree $d$ 
using Bos-Calvi tangents defined in the previous section. 
In general, this projector depends upon the local 
parametrisation but, in the case of a curve, it is 
independent at a general point.  
\\

Assume the same as in the previous section for 
$\BS{a}\in X\subset\BB{C}^m$ 
and its local parametrisation 
$\varphi:\,\CAL{O}_{m,\BS{a}}
\longrightarrow\CAL{O}_{n,\BS{b}}$. 
%%%%%%%%%%%%%%%%%%%%%%%%%%
\begin{defin}
Let $\iota:\,D_{\BS{a}}^{\varphi,d}\longrightarrow D_{\BS{a}}$
denote the inclusion mapping. 
We call its adjoint linear mapping 
$$
\TAY_{\BS{a}}^{\varphi,d}:={}^s\iota:\,\CAL{O}_{X,\BS{a}}
\longrightarrow P^d(X_{\BS{a}})
$$
\textit{$\varphi$-Taylor projector of degree} $d$ at $\BS{a}$. 
This was introduced by 
Bos and Calvi \cite{REFBC-1}, \cite{REFBC-2}. 
It is a little different from ours. 
Their projector is the composition 
of our $\TAY_{\BS{a}}^{\varphi,d}$ and the factor epimorphism 
$\pi:\CAL{O}_{m,\BS{a}}\longrightarrow\CAL{O}_{X,\BS{a}}$. 
The image $\TAY_{\BS{a}}^{\varphi,d}(f)$ is called the 
$\varphi$-\textit{Taylor polynomial} of $f$ of degree $d$. 
\end{defin}
%%%%%%%%%%%%%%%%%%%%
We know the following by Proposition \ref{PROprojector}. 
\begin{enumerate}
\item
The $\varphi$-Taylor projector $\TAY_{\BS{a}}^{\varphi,d}$ is 
a weakly continuous linear mapping.
\item 
We have the equalities 
$$
\Ker \TAY_{\BS{a}}^{\varphi,d}
=(D_{\BS{a}}^{\varphi,d})^{\bot_X}
,\quad
(\Ker\TAY_{\BS{a}}^{\varphi,d})^{\bot_X}=D_{\BS{a}}^{\varphi,d}
,$$$$
\Ker \TAY_{\BS{a}}^{\varphi,d}\COMP\pi
=(D_{\BS{a}}^{\varphi,d})^{\bot_m}
,\quad
(\Ker\TAY_{\BS{a}}^{\varphi,d}\COMP\pi)^{\bot_m}=D_{\BS{a}}^{\varphi,d}
,$$
where 
$\bot_X$ (resp. $\bot_m$) denotes the space of 
annihilators with respect to 
$S_{X,\BS{a}}^{\BS{\varphi}}$ 
(resp. $S_{m,\BS{a}}^{\BS{\varphi}}$).
\item 
The $\varphi$-Taylor projector 
$\TAY_{\BS{a}}^{\varphi,d}$ is a retraction of 
a vector space, i.e. 
$\TAY_{\BS{a}}^{\varphi,d}\COMP\kappa:\, P^d(X_{\BS{a}})
\longrightarrow P^d(X_{\BS{a}})$ is 
the identity, where $\kappa:\, P^d(X_{\BS{a}})
\longrightarrow \CAL{O}_{X,\BS{a}}$ 
denotes the inclusion. 
\end{enumerate}
%%%%%%%%%%%%%%%%%%%%%%%%%%%
\renewcommand{\tablename}{\textsf{diagram}}
\begin{table}[h]
\caption{\textsf{\scriptsize%
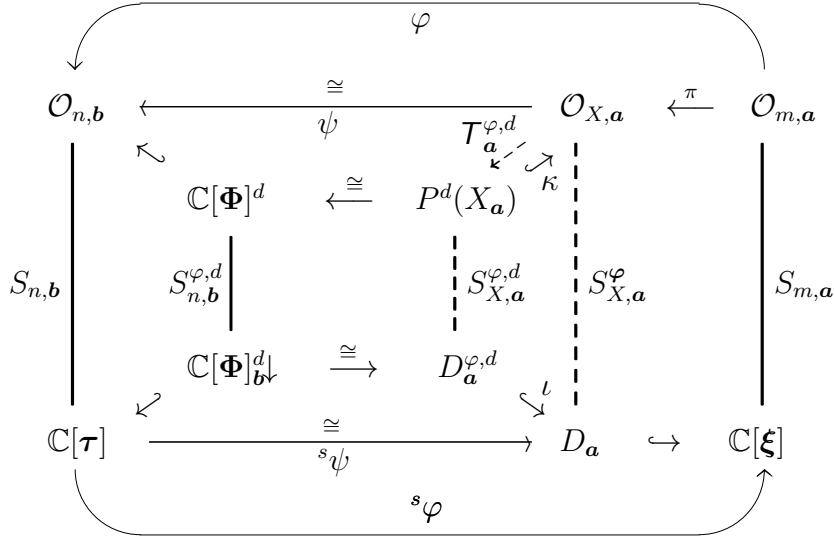
Dualities by the sesquilinear forms} $\DA$}
\label{tab:4}    % Give a unique label
\unitlength=.3em
% \unitlength=1mm
\begin{center}%\fbox{
\begin{picture}(92,62)
%%%%%%%%%%
\put(6,48){$\CAL{O}_{n,\BS{b}}$}
%%%%%%%
\put(15,43){\rotatebox[origin=c]{-35}{$\hookleftarrow$}}
%%%%%%%%%%
\put(21,38){$\BB{C}[\BS{\Phi}]^d
\hspace{1.6em} \stackrel{\ \cong}{\longleftarrow}
\hspace{1em} P^d(X_{\BS{a}})$}
%%%%%%%
\put(57,42){\rotatebox[origin=c]{35}{$\hookrightarrow$}}
\put(53,43){\rotatebox[origin=c]{35}{$\dashleftarrow$}}
%%%%%%%
\put(59,40){$\kappa$}
\put(50,45){$\TAY_{\BS{a}}^{\varphi,d}$}
%%%%%%%%
\put(61,48){$\CAL{O}_{X,\BS{a}}\hspace{1em}
\overset{\pi}{\longleftarrow}\hspace{.8em}\CAL{O}_{m,\BS{a}}$}
%%%%%%%
\put(61,12){$D_{\BS{a}}\hspace{1.2em}
\hookrightarrow\hspace{1.2em}\BB{C}[\BS{\xi}]$}
%%%%%%%%% 
\put(6,12){$\BB{C}[\BS{\tau}]$}
%%%%%%%%%
\put(15,16){\rotatebox[origin=c]{35}{$\hookleftarrow$}}
%%%%%%%%%
\put(56,16){\rotatebox[origin=c]{-35}{$\hookrightarrow$}}
%%%%%%%%%%
\put(59,18){$\iota$}
%%%%%%%%%
\put(21,20){$\BB{C}[\BS{\Phi}]^d_{\BS{b}}\DA
\hspace{1.6em} \stackrel{\cong\ }{\longrightarrow}
\hspace{1.6em} D_{\BS{a}}^{\varphi,d}$}
%%%%%%%%
\put(15.8,48.4){$\scriptscriptstyle\BS{\langle}$}
\path(16,49)(58,49)
\put(36,50.4){$\scriptstyle\cong\ $}
\path(16,48.96)(58,48.96)
\put(35,46.2){$\psi$}
%%%%%%%%
\path(17,13)(58,13)
\put(34.5,14){$\scriptstyle\ \cong$}
\path(17,12.96)(58,12.96)
\put(57.2,12.35){$\scriptscriptstyle\BS{\rangle}$}
\put(35,10){${}^s\psi$}
%%%%%%%%%%%%
\path(16,3)(76,3)
\put(16,10){\arc{14}{1.57}{3.14}}
\put(76,10){\arc{14}{0}{1.57}}
\put(45,5){${}^s\varphi$}
\path(82.2,9.3)(83,10)(83.6,9.5)
\path(82.2,9.3)(83,9.9)(83.8,9.3)
%%%
\path(16,60.1)(76,60.1)
\put(16,53){\arc{14}{3.14}{4.71}}
\put(76,53){\arc{14}{4.71}{6.28}}
\put(45,5){${}^s\varphi$}
\path(8.4,53.6)(9,53)(9.6,53.6)
\path(8.2,53.8)(9,53.1)(9.8,53.8)
\put(45,57.5){$\varphi$}
%%%%%%%%%%%
\thicklines
\path(8.7,17)(8.7,45)
 \path(8.8,17)(8.8,45)
 \put(1.9,29){$S_{n,\BS{b}}$}
\path(25.7,25)(25.7,35)
 \path(25.8,25)(25.8,35)
 \put(19,29){$S_{n,\BS{b}}^{\varphi,d}$}
 \dashline[50]{1.5}(49.7,25)(49.7,35)
 \dashline[50]{1.5}(49.8,25)(49.8,35)
 \put(51,29){$S_{X,\BS{a}}^{\varphi,d}$}
\dashline[50]{1.5}(62.7,17)(62.7,45)
 \dashline[50]{1.5}(62.8,17)(62.8,45)
 \put(64,29){$S_{X,\BS{a}}^{\BS{\varphi}}$}
\path(82.7,17)(82.7,45)
 \path(82.8,17)(82.8,45)
 \put(84,29){$S_{m,\BS{a}}$}
%%%%%%%%%%%%
\end{picture}
%}
\end{center}
\end{table}
%%%%%%%%%%%%%%%%%%%%%

Summing up we have Diagram \ref{tab:4}, where the bold lines 
imply the dual pairings $S_{m,\BS{a}}$ and $S_{n,\BS{b}}$ 
with respect 
to the affine coordinates $\BS{x}$ and $\BS{t}$ 
and the dotted ones imply those induced 
by $S_{n,\BS{b}}$ through isomorphisms $\psi$ and ${}^s\psi$. 
The upper half and the lower half correspond mutually 
by taking adjoint, 
excepting the inclusions in the upper half. 
%%%%%%%%%%%%%%%%%%%%%%%%%
\begin{exa}\label{taylor-deg-one}
Take an analytic set 
$$
X:=\{(x,y,z)\in\BB{C}^3:\ z=x+y+x^2+y^2\}
$$
with a parametrisation $\Phi$ defined by: 
$$
\varphi:\ x=s,\ y=t,\ z=s+t+s^2+t^2
.$$
Obviously we have 
$$
\BB{C}[\BS{\Phi}]_{\BS{0}}^1\DA
=\Span_{\BB{C}}(1,\sigma,\tau,\sigma^2+\tau^2)
.$$
The pushforwards of 
$\sigma,\ \sigma^2\in\BB{C}[\BS{\Phi}]_{\BS{0}}^1\DA$ 
are calculated as follows:
$$
\dfrac{\partial f(s,t,s+t+s^2+t^2)}{\partial s}=f_x+(1+2s)f_z
,$$$$
{}^s\varphi(\sigma)=\xi+\zeta
;$$
\begin{gather*}
\dfrac{\partial^2 f(s,t,s+t+s^2+t^2)}{\partial s^2}
=\dfrac{\partial}{\partial s}(f_x+(1+2s)f_z)
\\
=f_{xx}+(1+2s)f_{xz}
+2f_z
+(1+2s)f_{zx}+(1+2s)^2f_{zz},
\end{gather*}
$$
{}^s\varphi(\sigma^2)=2\zeta+\xi^2+2\xi\zeta+\zeta^2
.$$
By the symmetry of $\sigma$ and $\tau$, we have the following: 
$$
{}^s\varphi(\tau)=\eta+\zeta,\quad 
{}^s\varphi(\tau^2)=2\zeta+\eta^2+2\eta\zeta+\zeta^2
.$$
Thus we have 
\begin{gather*}
D_{\BS{0}}^{\varphi,1}
=\Span_{\BB{C}}({}^s\varphi(1),{}^s\varphi(\sigma),
{}^s\varphi(\tau),{}^s\varphi(\sigma^2+\tau^2))
\\
=\Span_{\BB{C}}(1,\xi+\zeta,\eta+\zeta,
4\zeta+\xi^2+\eta^2+2\xi\zeta+2\eta\zeta+2\zeta^2)
.\end{gather*}
Then the element 
$f:=\sum_{i,j,k=0}^{\infty}a_{ijk}x^iy^jz^k
\in\CAL{O}_{3,\BS{0}}$ 
belongs to 
$(\BB{C}[\BS{\Phi}]_{\BS{0}}^1\DA)^{\bot_3}$ 
if and only if the following hold:
$$
a_{000} = a_{100}+a_{001} = a_{010}+a_{001} 
= 2a_{001}+a_{200}+a_{020}+a_{101}+a_{011}+2a_{002} = 0
.$$
All the functions $f(x,y,z)\in\CAL{O}_{3,\BS{0}}$ 
of degree greater than 2 
at $\BS{0}$ is contained in 
$(\BB{C}[\BS{\Phi}]_{\BS{0}}^1\DA)^{\bot_3}$ 
and hence the Taylor expansion of their restrictions to 
$X$ are identically 0. 
Taylor expansions of the restrictions to $X$ of 
linear functions are 
identity by the property (3) above. 
Let us calculate the Taylor expansion of $x^2|_X$. We have  
\begin{gather*}
S_{3,\BS{0}}\SLFF{1}{x^2}=0
,\quad
S_{3,\BS{0}}\SLFF{\xi+\zeta}{x^2}=0
,\quad
S_{3,\BS{0}}\SLFF{\eta+\zeta}{x^2}=0
,\\
S_{3,\BS{0}}\SLFF{4\zeta+\xi^2+\eta^2%
+2\xi\zeta+2\eta\zeta+2\zeta^2}{x^2}=2
,\end{gather*}
on the other hand, we have 
\begin{gather*}
S_{3,\BS{0}}\SLF{1}{a+bx+cy+dz}=a,
\\
S_{3,\BS{0}}\SLF{\xi+\zeta}{a+bx+cy+dz}=b+d,
\\
S_{3,\BS{0}}\SLF{\eta+\zeta}{a+bx+cy+dz}=c+d,
\\
S_{3,\BS{0}}\SLFF{4\zeta+\xi^2+\eta^2+2\xi\zeta
+2\eta\zeta+2\zeta^2}{a+bx+cy+dz}=4d
.\end{gather*}
Solving the equations 
$a=0,\ b+d=0,\ c+d=0,\ 4d=2$, 
we have $a=0,\ b=c=-1/2,\ d=1/2$. This implies that 
$\TAY_{X,\BS{0}}^{\varphi,1}\left(x^2|_X\right)
=\dfrac{-x-y+z}{2}\biggm|_X$. 
\end{exa}

%%%%%%%%%%%%%%%%%%%%%%%%%
Let us recall that, even in the 1-dimensional case, 
there can exist 
a point with two different Taylor projectors as follows. 
%%%%%%%%%%%%%%%%%%%%%%%%%%
\begin{exa}\label{EXAgap}%
{\rm (Bos-Calvi \cite[Example 4.2]{REFBC-2}})  
Let $X\subset\BB{C}^2$ be the plane curve defined by 
$y-x^2-x^6=0$. 
Take two local parametrisations of $X$ at $\BS{0}$:
\begin{gather*}
\varphi:\BB{C}_{\BS{0}}\longrightarrow\BB{C}_{\BS{0}}^2,\quad
s\longmapsto\BS{\Phi}({s}):=(s,\, s^2+s^6),
\\
\psi:\BB{C}_{\BS{0}}\longrightarrow\BB{C}_{\BS{0}}^2,\quad
t\longmapsto\BS{\Psi}({t})
:=\left(t+t^2,\, (t+t^2)^2+(t+t^2)^6\right)
.\end{gather*} 
Then, putting $\sigma:=s_{\BS{0}}\DA,\ 
\tau:=t_{\BS{0}}\DA$, we have 
$$
\BB{C}[\BS{\Phi}]^2_{\BS{0}}\DA
=\Span_{\BB{C}}(1,\sigma,\sigma^2,\sigma^3,\sigma^4,\sigma^6)
,\qquad
\BB{C}[\BS{\Psi}]^2_{\BS{0}}\DA
=\Span_{\BB{C}}(1,\tau,\tau^2,\tau^3,\tau^4,\tau^6)
.$$
These lack the degree 5 term and are not $D$-invariant. 
We can confirm that 
$$
\TAY_{\BS{0}}^{\varphi,2}(x^5|_X)=0,\qquad
\TAY_{\BS{0}}^{\psi,2}(x^5|_X)=5(y-x^2)|_X
.$$ 
This means that two different local parametrisations 
sometimes define 
different Taylor projectors and that our Taylor projector does 
not necessarily coincide with the ordinary Taylor projector. 
%\vspace*{1ex}
\end{exa}
%%%%%%%%%%%%%%%%%%%%%%%%%%%%%%%%%%%
\begin{defin}
Let us put 
\begin{gather*}
\lambda_{\BS{\Phi}}(d):=\max\left\{k:\,
\bigoplus_{i=0}^{k}\FRAC{\FRAK{m}_{n,\BS{b}}^i}%
{\FRAK{m}_{n,\BS{b}}^{i+1}}
\subset\BB{C}[\BS{\Phi}]^d_{\BS{b}}\DA\right\}
=
\max\left\{k:\,\BB{C}[\BS{\tau}]^k
\subset\BB{C}[\BS{\Phi}]^d_{\BS{b}}\DA\right\}
\\
=
\max\left\{k:\,
\left(\BB{C}[\BS{\Phi}]^d_{\BS{b}}\DA\right)^{\bot_n}
\subset\FRAK{m}_{n,\BS{b}}^{k+1}\right\}
=
\max\left\{k :\, \CAL{O}_{X,\BS{a}}\subset
P^d(X_{\BS{a}})+\FRAK{m}_{X,\BS{a}}^{k+1}\right\}
.\end{gather*}
This is independent of the parametrisation $\BS{\Phi}$. 
\end{defin}
%%%%%%%%%%%%%%%%%
\begin{prop}\label{PROinclusions}
We have $\lambda_{\BS{\Phi}}(d)\ge d$. 
%Hence $\BB{C}[\BS{\Phi}]^d_{\BS{b}}\DA\cap\BB{C}[\BS{\tau}]^{d+1}$ 
%is a $D$-invariant subspace. 
\end{prop}
% \vspace*{2ex}
%%%%%%%%%%%%%%%%%%%%

\begin{proof}
Since $\BS{\Phi}$ generates $\FRAK{m}_{n,\BS{b}}$, all 
$\tau_1,\dots,\tau_n$ appear in the linear terms of 
$\BB{C}[\BS{\Phi}]^1$. 
Hence all terms of $\BB{C}[\BS{\tau}]^d$ appear 
in $\BB{C}[\BS{\Phi}]^d_{\BS{b}}\DA$. 
%This proves the first assertion. 
%Since the partial derivatives of 
%an element of $\BB{C}[\BS{\tau}]^{k+1}$ $(k\le d)$ 
%belongs to $\BB{C}[\BS{\tau}]^k$, 
%it belongs to $\BB{C}[\BS{\Phi}]^d_{\BS{b}}\DA$ also. 
\end{proof}
%%%%%%%%%%%%%%%%%
%\begin{prop}\label{PROPcomparison}

By the inequality $\lambda_{\BS{\Phi}}(d)\ge d$, 
we have the following formal error bound for our 
$\varphi$-Taylor 
polynomial $\TAY_{\BS{a}}^{\varphi,d}(f)$: 
$$
f-\kappa\COMP\TAY_{\BS{a}}^{\varphi,d}(f)
\in
(D_{\BS{a}}^{\varphi,d})^{\bot_{X}}
\subset
\FRAK{m}_{X,\BS{a}}^{\lambda_{\BS{\Phi}}(d)+1}
\quad(f\in\CAL{O}_{X,\BS{a}})
,$$
where $\kappa:\,P^d(X_{\BS{a}})
\longrightarrow\CAL{O}_{X,\BS{a}}$ 
denotes the inclusion. 
This formal error bound is equal to or smaller than that of 
the ordinary Taylor polynomial $\TAY_{\BS{a}}^{*d}(f)$: 
$$
f-\kappa\left(\TAY_{\BS{a}}^{*d}(f)\mod I_{\BS{a}}\right)
\in
\FRAK{m}_{X,\BS{a}}^{d+1}
\quad(f\in\CAL{O}_{X,\BS{a}})
.$$
The author does not know whether 
$$
\forall f\in\CAL{O}_{X,\BS{a}}:\ 
\ord_{X,\BS{a}}
\left(f-\kappa\COMP\TAY_{\BS{a}}^{\varphi,d}(f)\right)
\ge\ord_{X,\BS{a}}\left(f-\kappa
\left(\TAY_{\BS{a}}^{*d}(f)\mod I_{\BS{a}}\right)\right)
$$
holds or not, where $\ord_{X,\BS{a}}$ is the order on 
$X_{\BS{a}}$ defined by 
$$
\ord_{X,\BS{a}}(f)=\max\{k:f\in\FRAK{m}_{X,\BS{a}}^k\}
\quad 
(f\in\CAL{O}_{X,\BS{a}})
.$$
%%%%%%%%%%%%%%%%%%%%%%%%%%
\begin{defin}\label{DEFinvtaylor}
If $\BB{C}[\BS{\Phi}]^d_{\BS{b}}\DA$ is $D$-invariant for all 
(or some by Lemma \ref{LEMcoordtr}) local parametrisation 
$\varphi$, 
we call $\BS{a}=\BS{\Phi}(\BS{b})\in X$ \textit{$D$-invariant 
point of degree $d$.} 
If $\BS{a}$ is $D$-invariant of degree $d$ for all 
$d\in\BB{N}$, 
we call it \textit{$D$-invariant point of degree $\infty$.} 
If the $\varphi$-Taylor projector of degree $d$ at 
$\BS{a}\in X$ is 
independent of the local parametrisation, we call $\BS{a}$ 
\textit{Taylorian of degree $d$} (following Bos and Calvi). 
If $\BS{a}$ is Taylorian of degree $d$ for all $d\in\BB{N}$, 
we call it \textit{Taylorian of degree $\infty$.} 
\end{defin}
%%%%%%%%%%%%%%%%%%%%%%%%%
\begin{prop}\label{PROideal-closed}
Let $X$ be a regular complex submanifold of an 
open subset of $\BB{C}^m$ 
with $\dim X\ge 1$. 
For a point $\BS{a}\in X$, the following conditions (1),\dots,
(4) are equivalent for a fixed $d\in\BB{N}$. 
\begin{enumerate}
\item
$\BS{a}$ is a $D$-invariant point of degree $d$.  
\item
$(\BB{C}[\BS{\Phi}]^d_{\BS{b}}\DA)^{\bot_n}$ is an ideal of 
$\CAL{O}_{n,\BS{b}}$ for all (or some) 
local parametrisation $\varphi$ 
of $X_{\BS{a}}$. 
\item
$\Ker\TAY_{\BS{a}}^{\varphi,d}
=(D_{\BS{a}}^{\varphi,d})^{\bot_X}$ 
is an ideal of $\CAL{O}_{X,\BS{a}}$ for all (or some) local 
parametrisation $\varphi$.  
\item
$D_{\BS{a}}^{\varphi,d}$ is $D$-invariant in $\BB{C}[\BS{\xi}]$ 
for all (or some) local parametrisation $\varphi$. 
\end{enumerate}
If $\BS{b}$ is a bundle point of all the jet spaces of 
$\BB{C}[\BS{\Phi}]^d$, all of these conditions hold. 
\end{prop}
%%%%%%%%%%%%%%%%%%%%%%%%%%%

\begin{proof}
Note that the condition of bundle point is independent 
of the local 
parametrisation (Theorem \ref{THMintrinsic} 
or Lemma \ref{LEMcoordtr}). 
The equivalence 
$(1)\Longleftrightarrow(2)$ is proved in 
Theorem \ref{LEMdclosed-ideal}. 
Since the image or the inverse image of an ideal by 
a ring epimorphism is an ideal, the equivalences 
$(2)\Longleftrightarrow(3)$ holds. 
The equivalence (1)$\Longleftrightarrow$(4) follows from 
Theorem \ref{THMdmorphism}, (2). 
The last assertion on a bundle point follows from 
Theorem \ref{THMintrinsic}. 
\end{proof}
%%%%%%%%%%%%%%%%%%%%%%%%%
\begin{exa}\label{EXAdinv-but-notbdle}
Let us take the global parametrisation 
$\BS{\Phi}:=(s,\,t,\,t^2+st^2,\,t^3)$ of the surface 
$X:=\{(x,\,y,\,z,\,w):\,z=y^2+xy^2,\,y^2=w\}\subset\BB{C}^4$. 
Then $\BB{C}[\BS{\Phi}]^1$ coincides with $Z$ in 
Example \ref{EXA1}. 
All points of the form $(a,\,0)$ $(a\neq -1)$ are 
$D$-invariant ones of 
degree 1 but they are not bundle points of 1-jets of 
$\BB{C}[\BS{\Phi}]^1$. 
\end{exa}
%%%%%%%%%%%%%%%%%%%%%%%%%%

For plane algebraic curves, Bos and Calvi proves that 
$D$-invariance 
(gap-free property) is equivalent to the condition that 
the $\varphi$-Taylor projector is independent of 
the local parametrisation 
(\cite[Theorem 3.4, 4.10]{REFBC-2}). 
Unfortunately this can not be generalised. 
Taylorian property fails in the most simple 2-dimensional 
example as follows. 
%%%%%%%%%%%%%%%%%%%%%%%%
\begin{exa}\label{EXAlintr}
Let us recall the surface $X\subset\BB{C}^3$ defined by 
$x_3=x_2^2$ in Example \ref{EXAdifference}. We have seen that 
two local parametrisation $\varphi$ and $\psi$. 
Although they are related by a linear transformation, 
their sets of Bos-Calvi tangents of order 1 are different:
$$
 (D_{\BS{a}}^{\varphi,1})^{\bot_X\bot_X} 
=
D_{\BS{a}}^{\varphi,1}
\neq
D_{\BS{a}}^{\psi,1}
=
 (D_{\BS{a}}^{\psi,1})^{\bot_X\bot_X} 
\quad(\BS{a}\in X)
.$$
Then the kernels $(D_{\BS{a}}^{\varphi,1})^{\bot_X}$ and 
$(D_{\BS{a}}^{\varphi,1})^{\bot_X}$ of the 
Taylor projectors are different. 
Hence, no point of $X$ are Taylorian although they are 
all bundle points. 
\end{exa}
%%%%%%%%%%%%%%%%%%%%%
\begin{rem}\label{REMgen}
For a general $D$-invariant point, we can only say 
that our Taylor 
projector of order $d$ defines a structure of 
an Artinian algebra on 
$P^d(X_{\BS{a}})$.  This structure is isomorphic to $\mathbb{C}\{\boldsymbol{t}\}/%
(\mathbb{C}[\boldsymbol{\Phi}]_{\boldsymbol{b}}^d\DA)^{\bot_n}$, 
which is transformed by the isomorphism described in Theorem 
\ref{THMvector} as a contravariant tensor by a change of the local parametrisation. 
\end{rem}
%%%%%%%%%%%%%%%%%%%%%%%%%%%%%%%%%%%%%%%%
\section{Taylorian property of points on 
embedded curves}\label{section-curves}
%%%%%%%%%%%%%%%%%%%%%%%%%%%%%%%
If we restrict ourselves to the case of embedded 
analytic curves, the Taylor projector is well-defined 
at a general 
point, which generalise Bos and Calvi \cite{REFBC-2}, 
Theorem 3.4 on plane algebraic curves. 

%%%%%%%%%%%%%%%%%%%%%%%
\begin{thm}\label{THMequivalence}
Let $X$ be a $1$-dimensional regular complex submanifold of a 
neighbourhood of $\BS{a}\in\BB{C}^m$ and let 
$\BS{\Phi}=(\Phi_1,\dots,\Phi_m):\,\BB{C}_{\BS{b}}
\longrightarrow \BB{C}_{\BS{a}}^m$ be its local 
parametrization. 
Then, for any fixed $d\in\BB{N}$, the following 
three properties 
on $\BS{a}\in X$ are equivalent.
\begin{enumerate}
\item
$\BS{a}$ is a bundle point of all the jet spaces of 
$\BB{C}[\BS{\Phi}]_{\BS{b}}^d\DA$. 
\item
The powers of monomials appearing in 
$\BB{C}[\BS{\Phi}]^d_{\BS{b}}\DA\subset\BB{C}[{\tau}]$ 
form a \textit{gap-free} sequence i.e. $\BS{a}$ is 
$D$-invariant 
point of degree $d$ 
\item
$\BS{a}$ is Taylorian of degree $d$ 
(see Definition \ref{DEFinvtaylor}). 
\end{enumerate}
\end{thm}
%%%%%%%%%%%%%%%%%%%%%%%

\begin{proof} 
The degree $k$ part of 
$\BB{C}[\BS{\Phi}]_{\BS{t}}^d\DA$ is 0-dimensional 
or 1-dimensional for any $k\in\BB{N}_0$. 
Suppose that $\BS{a}$ satisfies the condition (1), 
then the degree $k$ part of 
$\BB{C}[\BS{\Phi}]_{\BS{t}}^d\DA$ is constant 
dimensional in a neighbourhood of $\BS{a}$ for 
each $k$. If degree $k$ part is 
0-dimensional on a neighbourhood $V$, 
the parts with higher degrees are all 0 on $V$. Then (2) holds. 

Since $\dim_{\BB{C}}\BB{C}[\BS{\Phi}]_{\BS{b}}^d\DA=
\dim_{\BB{C}}\BB{C}[\BS{\Phi}]_{\BS{a}}^d=d+1$, 
the condition (2) implies that, for each $k$ with 
$0\le k\le d+1$, 
there is at least one $f_d\in\BB{C}[\BS{\Phi}]_{\BS{b}}^d$ 
with $f_d^{(k)}(\BS{b})\neq 0$. 
This situation does not change in a neighbourhood, 
which implies (1). 
%(If there was a gap, it would be filled by 
%some element of $\BB{C}[\BS{\Phi}]_{\BS{b}}^d\DA$ at a point 
%arbitrarily near to $\BS{b}$.) 

To prove $(2)\Longrightarrow(3)$, 
recall that every ideal of $\CAL{O}_{X,\BS{a}}\cong\BB{C}\{x\}$ 
is of the form $\FRAK{m}_{X,\BS{a}}^k$, a power of 
the maximal ideal. 
If $\BS{a}$ is a $D$-invariant point of degree $d$, 
$\Ker\TAY_{\BS{a}}^{\varphi,d}=(D_{\BS{a}}^{\varphi,d})^{\bot_X}$ 
is an ideal and it is determined by 
$\dim_{\BB{C}}\CAL{O}_{X,\BS{a}}
/(D_{\BS{a}}^{\varphi,d})^{\bot_X}
=\dim_{\BB{C}}P_{X,\BS{a}}^d$. 
Hence it is independent of $\varphi$ 
and we have 
$\Ker\TAY_{\BS{a}}^{\varphi,d}=\Ker\TAY_{\BS{a}}^{\psi,d}$ 
for any other local parametrisation $\psi$. 
Suppose that $\BS{a}$ is a $D$-invariant point of 
degree $d$ and take 
any $f\in\CAL{O}_{X,\BS{a}}$. Let us put 
$p:=\TAY_{\BS{a}}^{\varphi,d}(f)\in P^d(X_{\BS{a}})$ 
and $q:=\TAY_{\BS{a}}^{\psi,d}(f)\in P^d(X_{\BS{a}})$. 
In view of the retraction property of the projectors, we have 
\begin{gather*}
p-q
=
\TAY_{\BS{a}}^{\varphi,d}\big(\kappa(p)-\kappa(q)\big)
=
\TAY_{\BS{a}}^{\varphi,d}\big((\kappa(p)-f)-(\kappa(q)-f)\big)
\\
\in
\TAY_{\BS{a}}^{\varphi,d}
\left(\Ker\TAY_{\BS{a}}^{\varphi,d}
-\Ker\TAY_{\BS{a}}^{\psi,d}\right)
=
\TAY_{\BS{a}}^{\varphi,d}
\left(\Ker\TAY_{\BS{a}}^{\varphi,d}
-\Ker\TAY_{\BS{a}}^{\varphi,d}\right)
=
\{\BS{0}\}
,\end{gather*}
where 
$\kappa:\,P^d(X_{\BS{a}})\longrightarrow\CAL{O}_{X,\BS{a}}$ 
denotes the inclusion. This proves that 
$\TAY_{\BS{a}}^{\varphi,d}(f)=\TAY_{\BS{a}}^{\psi,d}(f)$. 
Thus $D$-invariance implies Taylorian property. 

To prove the converse $(3)\Longrightarrow(2)$, 
we follow faithfully the idea of Bos-Calvi. 
The least space $\BB{C}[\BS{\Phi}]^d_{\BS{b}}\DA$ is 
a subspace of $\BB{C}[\tau]$. 
Suppose that $\BS{a}$ is not a $D$-invariant point of 
degree $d$. 
Then there exists $s\in\BB{N}$ such that  
$\tau^s\not\in\BB{C}[\BS{\Phi}]^d_{\BS{b}}\DA$ and 
$\tau^{s+1}\in\BB{C}[\BS{\Phi}]^d_{\BS{b}}\DA$. 
Let $s$ be the maximum of such numbers. There exists 
a coordinate $x_j\in\BB{C}[\BS{x}]$ such that 
$\varphi(x_j)_{\BS{b}}\DA=\alpha\tau$ 
with $\alpha\neq 0$, 
otherwise the image of $\varphi$ does not includes 
the elements of order $1$, 
contradicting the retraction property.  
Let $l$ denote the maximal number such that 
$\tau^{s+l}\in\BB{C}[\BS{\Phi}]^d_{\BS{b}}\DA$. 
Then $l\ge 1$ and 
$$
\tau^{s+1},\,\tau^{s+2},\dots,\tau^{s+l}
\in\BB{C}[\BS{\Phi}]^d_{\BS{b}}\DA,
\quad
\tau^{s+l+1},\,\tau^{s+l+2},\dots
\not\in\BB{C}[\BS{\Phi}]^d_{\BS{b}}\DA
$$
by the maximality of $s$. 
The least non-zero monomial appearing in $\varphi(x_j^s)$ 
is $\alpha^st^s$.  
Taking $g_{s+i}\in\BB{C}[\BS{\Phi}]^d$ such that 
$\varphi(g_{s+i})_{\BS{b}}\DA=\tau^{s+i}$ $(i=1,\dots,l)$, 
we can eliminate the monomials of degree 
$s+1,\dots s+l$ appearing 
in $\varphi(x_j^s)$ by subtracting a linear combination 
$c_1\varphi(g_{s+1})+\dots+c_{s+l}\varphi(g_{s+l})$, 
beginning from 
$g_{s+i}$ with smaller $i$. Then if we put 
$$
h:=x_j^s-(c_1g_{s+1}+\dots+c_{s+l}g_{s+l})
,$$
we have 
$$
\varphi(h)=\alpha^st^s+k(t)\cdot t^{s+l+1}
\in(\BB{C}[\BS{\Phi}]^d_{\BS{b}}\DA)^{\bot_n}
$$ 
for some $k(t)\in\CAL{O}_n$. This proves that 
$\TAY_{\BS{a}}^{\varphi,d}(h)=0$. 

Now take another local parametrisation 
$$
\BS{\Psi}(t):=\BS{\Phi}(t'+t'^2)
=\Bigl(\Phi_1(t'+t'^2),\dots,\Phi_n(t'+t'^2)\Bigr)
.$$ 
The function $\psi(h)$ is expressed as 
$$
\psi(h)=\alpha^s(t'+t'^2)^s+k(t'+t'^2)\cdot(t'+t'^2)^{s+l+1}
.$$
Here the coefficient of $t'^{s+1}$ is not $0$ and it follows 
that $S_n\SLFF{\tau'^{s+1}}{\psi(h)}\neq 0$. Since 
$\tau^{s+1}\in\BB{C}[\BS{\Phi}]^d_{\BS{b}}\DA$ implies 
$\tau'^{s+1}\in\BB{C}[\BS{\Psi}]^d_{\BS{b}}\DA$ 
(see the proof of Theorem \ref{THMvector}), we see that 
$\TAY_{\BS{a}}^{\psi,d}(h)\neq 0$. This is inconsistent with 
$\TAY_{\BS{a}}^{\varphi,d}(h)=0$ and proves that $\BS{a}$ 
is not a Taylorian point. 
\end{proof}
%%%%%%%%%%%%%%%%%%%%
\section{Zero estimate and transcendency index}%
\label{section-transcendency}
%%%%%%%%%%%%%%%%%%%%%%
In this final section we recall that the growth of the space of 
Bos-Calvi tangents of dual degree $d$ measures the 
transcendence of the 
embedding of manifold germ $X_{\BS{a}}$. 
In particular, we explain that the $D$-invariance property of 
$X_{\BS{a}}$ 
implies that the embedding of $X_{\BS{a}}$ has not a 
high index of transcendency.\\

First let us recall some known facts 
on zero-estimate on local algebras. 
The following invariant $\theta_{\BS{\Phi}}(d)$ is called 
``$d$-order" by Bos and Calvi in \cite{REFBC-1}. 
It is more important 
than $\lambda_{\BS{\Phi}}(d)\,
\left(\le\theta_{\BS{\Phi}}(d)\right)$ 
defined in \S \ref{section-taylorian}. 
This invariant coincides with one treated by the 
present author in \cite[\S 1]{REFpitman}.  
%%%%%%%%%%%%%%%%%%%%
\begin{defin}[Izumi \cite{REFjacad}, \cite{REFpitman}] 
Let $X_{\BS{a}}$ be a germ of a complex submanifold of 
a neighbourhood 
of $\BS{a}\in\BB{C}^m$ defined by an ideal 
$I\subset\CAL{O}_{X,\BS{a}}$ and 
$\BS{\Phi}:=(\Phi_1,\dots,\Phi_m):\,\BB{C}^n_{\BS{b}}
\longrightarrow \BB{C}_{\BS{a}}^m$ a local parametrisation of 
the germ $X_{\BS{a}}$. 
Let us use the abbreviation $f|_X:=f\mod I$ 
(the restriction of $f\in\CAL{O}_{n,\BS{a}}$ to $X_{\BS{a}}$), 
$\BS{x}|_X:=\{x_1|_X,\dots,x_m|_X\}$ and 
$A:=\CAL{O}_{X,\BS{a}}$. 
The \textit{zero-estimate function} is defined by 
\begin{gather*}
\theta_{A,\,\BS{x}|_X}(d)
=\theta_{\CAL{O}_{n,\BS{b}},\BS{\Phi}}(d)
:=
\max\left\{\deg~p:~p\in\BB{C}%
[\BS{\Phi}]^d_{\BS{b}}\DA\setminus\{0\}\right\}
\\
=\sup\{\ord_{\BS{b}} F\COMP\BS{\Phi}:
F\in \mathbb C[\BS{x}]^d,\ F\COMP\BS{\Phi}\neq 0\}
\\
= \sup\{\ord_{X,\BS{a}} F|_X:
F\in \mathbb C[\BS{x}]^d,\ F|_X\neq 0\}
\end{gather*}
and \textit{the transcendency index} of $\BS{\Phi}$ by 
$$
\alpha(\BS{\Phi})
:=\limsup_{d\to\infty}\log_d\theta_{A,\,\BS{x}|_X}(d)
=\limsup_{d\to\infty}\log_d%
\theta_{\CAL{O}_{n,\BS{b}},\BS{\Phi}}(d)
,$$
where $\ord_{X,\BS{a}}$ is defined before  
Definition \ref{DEFinvtaylor}. 
\end{defin}
%%%%%%%%%%%%%%%%%%%%%%%%%

Note that the values of 
$\theta_{A,\,\BS{x}|_X}(d)
=\theta_{\CAL{O}_{n,\BS{b}},\BS{\Phi}}(d)$ 
are finite. The zero-estimate function $\theta_{\BS{\Phi}}(d)$ 
and the transcendency index 
$\alpha(\BS{\Phi})$ are dependent on the embedding 
$X_{\BS{a}}\subset\BB{C}_{\BS{a}}^m$ 
but they are independent of 
the local parametrisation for a fixed 
$X_{\BS{a}}\subset\BB{C}_{\BS{a}}^m$. 
We know the following. 
%%%%%%%%%%%%%%%%%%%%%
\begin{thm}[Izumi \cite{REFjacad}, 
\cite{REFpitman}]\label{algebraicity} 
Let $X$ be an $n$-dimensional regular complex 
submanifold $(n\ge 1)$ of a 
neighbourhood of $\BS{a}\in\BB{C}^m$ and 
$\BS{\Phi}:\,\BB{C}_{\BS{b}}^n\longrightarrow\BB{C}_{\BS{a}}^m$ 
a local parametrisation of $X$ at $\BS{a}$. 
Then we have 
$$
\theta_{\CAL{O}_{X,\BS{a}},\,\BS{x}|_X}(d)\ge d,
\quad 
\alpha(\BS{\Phi})\ge 1
$$
and the following conditions are equivalent. 
\begin{enumerate}
\item
The germ $X_{\BS{a}}$ is an analytic irreducible component of 
the germ of an algebraic set at $\BS{a}$. 
\item
There exists $a\ge 1$ and $b\ge 0$ such that 
$$
\theta_{\CAL{O}_{X,\BS{a}},\,\BS{x}|_X}(d)
=\theta_{\CAL{O}_{n,\BS{b}},\BS{\Phi}}(d)
\le ad+b\quad (d\in\BB{N})
.$$
\item
There exists $a\ge 1$ and $b\ge 0$ such that 
$$
\BB{C}[\BS{\Phi}]^d\cap\FRAK{m}_{n,\BS{b}}^{ad+b+1}=\{0\}
\quad\mbox{ i.e. }\ 
\BB{C}[\BS{\Phi}]^d_{\BS{b}}\DA
\cap
\FRAC{\FRAK{m}_{n,\BS{b}}^i}{\FRAK{m}_{n,\BS{b}}^{i+1}}
=\{0\}\ (i>ad+b)
.$$
\item
$\alpha(\BS{\Phi})=1.$
\end{enumerate}
\end{thm}
%%%%%%%%%%%%%%%%%%%%
The latter condition of (3) appears for the first time. 
The assertion $(2)\Longleftrightarrow(3)$ is clear from the definition 
of the least part. 
We give $\alpha(\BS{\Phi})$ the name ``transcendency index" 
because of the equivalence $(1)\Longleftrightarrow(4)$. 
It is known that $\alpha(\boldsymbol{\Phi})$ 
is not necessarily an integer in transcendence theory 
(cf. \cite{REFp1}, \cite{REFp2}, \cite{REFpw}, 
see also a beginner's note \cite[Example 3.4]{REFpitman}). 
In \cite{REFjacad} we treat polynomial functions 
on an analytically 
irreducible germ of the analytic subset $X$ of 
an open subset of 
$\BB{C}^m$ and $X$ need not be a smooth manifold. 
The general inequalities 
$\theta_{\CAL{O}_{X,\BS{a}},\,\BS{x}|_X}(d)\ge d$ and 
$\alpha(\BS{\Phi})\ge 1$ follow from 
Proposition \ref{PROinclusions}. 
A complete proof is given in \cite[Theorem 2.3]{REFpitman} 
in a stronger form. 

The situation of this 
theorem may be well illustrated by the following. 

%%%%%%%%%%%%%%%
\begin{exa}[Izumi \cite{REFjacad}]\label{EXAexp}
Let $C$ be the transcendental plane curve defined by $y=e^x-1$. 
If we parametrise this by $\BS{\Phi}=(t,\,e^t-1)$, we have 
$$
\BB{C}[\BS{\Phi}]^d
=\Span_{\BB{C}}\left(\BB{C}[t]^d,\BB{C}[t]^{d-1}e^t,
\BB{C}[t]^{d-2}e^{2t},\dots,
\BB{C}[t]^1e^{(d-1)t},e^{dt}\right)
.$$
This is just the space of solutions of 
the differential equation
$$
D_t^{d+1}
\left(D_t-1\right)^d
\left(D_t-2\right)^{(d-1)}\cdots
\left(D_t-d+1\right)^2
\left(D_t-d\right)^1f=0
\quad(D_t:={d}/{dt})
.$$
By the elementary theory of ordinary differential equations, 
for any $a\in\BB{C}$, there exists a unique solution $f$ with 
$$
f^{(\nu)}(a)=
\begin{cases}
0 &\big(0\le\nu\le(d+1)(d+2)/2-2\big)\\
1 &\big(\nu=(d+1)(d+2)/2-1\big)
\end{cases}
$$ 
and there exists no solution $f\neq 0$ with 
$$
f^{(\nu)}(a)=0\quad\big(0\le\nu\le(d+1)(d+2)/2)-1\big)
.$$
This proves that 
$\theta_{\CAL{O}_{1,a},\,\{t-a,\,\exp(t-a)-1\}}(d)
=(d+1)(d+2)/2-1$ 
and $\alpha=2$ at all points $(a,\,e^a-1)\in C$. 
\end{exa}
%%%%%%%%%%%%%%%
\begin{rem}\label{REMueda}
An example of a plane curve with 
an extremely transcendental point 
($\alpha(\BS{\Phi})=\infty$) is given by Tetsuo Ueda 
(see \cite[Example 2]{REFjacad}), using a gap power series, a 
functional analogue of Liouville constant. 
\end{rem}
%%%%%%%%%%%%%%%%%%%%%

For an embedding germ $X_{\BS{a}}\subset\BB{C}^m$, 
let $\overline{X}_{\BS{a}}$ denote the Zariski closure of 
$X_{\BS{a}}$ in $\BB{C}^m$, namely the smallest algebraic 
subset of $\BB{C}^m$ 
that includes some representative of the germ $X_{\BS{a}}$ 
(germs are always taken with respect to 
the Euclidean topology). Then the Hilbert function of 
$\overline{X}_{\BS{a}}$ is defined by 
\begin{gather*}
\chi(\overline{X}_{\BS{a}},\,d)
:=\dim_{\BB{C}} P^d(X_{\BS{a}})
-\dim_{\BB{C}} P^{d-1}(X_{\BS{a}})
=\dim_{\BB{C}}D_{\BS{a}}^{\varphi,d}-\dim_{\BB{C}}D_{\BS{a}}^{\varphi,d-1}
\\
\left(\dim_{\BB{C}} P^{-1}(X_{\BS{a}})
=\dim_{\BB{C}} D_{\BS{a}}^{\varphi,-1}=0\right)
\end{gather*}
(cf. Definition \ref{DEFbctg}). 
This is known to coincide with a polynomial of degree 
$\dim\overline{X}_{\BS{a}}-1$ in $d$ for sufficiently 
large $d$. 

%%%%%%%%%%%%%%%%%%%%%
\begin{thm}\label{THMinequality}
Let $\varphi:\CAL{O}_{m,\BS{a}}
\longrightarrow\CAL{O}_{n,\BS{b}}$ $(n\ge 1)$ 
be a local parametrisation of an embedded manifold germ 
$X_{\BS{a}}\subset\BB{C}^m$ with component function germs
$\BS{\Phi}=(\Phi_1,\dots,\Phi_m)$. 
Suppose that $\BS{a}$ is a $D$-invariant point of degree $d$, 
we have a zero-estimate inequality: 
$$
{{n+d}\choose{n}}+\theta_{\CAL{O}_{n,\BS{b}},\BS{\Phi}}(d)-d\le
\dim_{\BB{C}}\BB{C}[\BS{\Phi}]^d
=
\sum_{i=0}^d \chi(\overline{X}_{\BS{a}},\,i)
\le
{{m+d}\choose{m}}
.$$
Hence, if $\BS{a}$ is a $D$-invariant point of degree $\infty$, 
we have an estimate of the transcendency index: 
$$
1\le\alpha(\BS{\Phi})\le\dim \overline{X}_{\BS{a}}\le m
.$$ 
\end{thm}
%%%%%%%%%%%%%%%%%%%%%

\begin{proof} 
Note that 
$$
\dim_{\BB{C}}P^d(X_{\BS{a}})=\dim_{\BB{C}}\BB{C}[\BS{\Phi}]^d
=\dim_{\BB{C}}\BB{C}[\BS{\Phi}]_{\BS{b}}^d\DA
$$
by the isomorphism stated in Definition \ref{DEFbctg}. 
If we take $p\in\BB{C}[\BS{\Phi}]^d_{\BS{b}}\DA$ with 
$\deg p=\theta_{A,\{\Phi_1,\dots,\Phi_m\}}(d)$ 
(the maximal degree), 
the dimension $\dim_{\BB{C}}\BB{C}[\BS{\Phi}]^d_{\BS{b}}\DA$ 
majorises the sum of the dimensions of the 
following linear subspaces: 
\begin{enumerate}
\item
the space 
$\mathbb{C}[\BS{\tau}]^d$ appearted in the proof of 
Proposition \ref{PROinclusions}; 
%$\bigoplus_{i=0}^d{\FRAK{m}_{n,\BS{b}}^i}/{\FRAK{m}_{n,\BS{b}}^{i+1}}$; 
\item
the linear span of 
$$
\left\{
\FRAC{\partial^{|\BS{\nu}|} p}{\partial \BS{\tau}^{\BS{\nu}}}
:\,
d<\ord_{\!\FRAK{m}}\,
\FRAC{\partial^{|\BS{\nu}|} p}{\partial \BS{\tau}^{\BS{\nu}}}
<\infty,\ \BS{\nu}\in\BB{N}_0^n
\right\}
$$
(by $D$-invariance) 
\end{enumerate} 
(see the proof of Proposition \ref{PROinclusions}). 
Since the intersection of these spaces are $\{\BS{0}\}$, 
we have the left inequality of the first. 
The right inequality follows from 
$$
\dim_{\BB{C}}P^d(X_{\BS{a}})\le \dim_{\BB{C}}\BB{C}[\BS{x}]^d
={{m+d}\choose{m}}
.$$ 
The first inequality in the theorem implies that
$$
\theta_{\CAL{O}_{n,\BS{b}},\{ \Phi_1,\dots,\Phi_m \}}(d)
\le \dim_{\BB{C}}\BB{C}[\BS{\Phi}]^d
=\dim_{\BB{C}}P^d(X_{\BS{a}})
=\sum_{i=0}^d \chi(\overline{X}_{\BS{a}},\,i)
.$$
Since the Hilbert function of $\overline{X}_{\BS{a}}$ 
coincide with a polynomial of degree 
$\dim\overline{X}_{\BS{a}}-1$ 
for sufficiently large $d$, the last term is comparable to 
$d^{\,\dim\overline{X}_{\BS{a}}}$ 
and the inequality 
$\alpha(\BS{\Phi})\le\dim \overline{X}_{\BS{a}}$ 
follows. 
\end{proof}

Let $\BS{\Phi}:\,U\longrightarrow\BB{C}^m$ be an embedding 
of an open subset of $\BB{C}^n$ 
onto a submanifold $X\subset\BB{C}^m$ i.e. $\BS{\Phi}$ induces 
a biholomorphic homeomorphism onto the image. 
We have seen that the complement of the set of 
bundle points of all the 
jet spaces of $\BB{C}[\BS{\Phi}]^d$ for all $d\in\BB{N}$ 
is contained in a countable union of thin closed 
analytic subsets of $X$. Then the zero-estimate inequality in 
Theorem \ref{THMinequality} implies the following 
global result. 

%%%%%%%%%%%%%%%%%%%%%%%
\begin{cor}\label{CORtame}
Let $X$ be an $n$-dimensional regular complex submanifold of 
an open subset 
$\Omega\subset\BB{C}^m$ $(n\ge 1)$. 
Then there exists a countable union $A$ of thin closed 
analytic subsets of $X$ such that, for any local 
parametrisation 
$\BS{\Phi}$ at $\BS{a}\in \Omega\setminus A$, we have 
$\alpha(\BS{\Phi})\le\dim \overline{X}_{\BS{a}}\le{m}$. 
Note that the set $A$ is of first category in Baire's sense and 
with Lebesgue measure 0 in $X$. 
\end{cor}
%%%%%%%%%%%%%%%%%%%%%%%%%%%
\begin{rem}
Gabrielov gives a zero-estimate 
\cite[Theorem 5]{REFgabrielov} of 
Noetherian functions on an integral curve of a 
Noetherian vector field 
(see \cite{REFgk} also). 
It yields a zero-estimate of Noetherian functions on $\BB{C}^n$ 
immediately as follows. Suppose that 
$$
\BS{\Psi}:=
\{\BS{x},\,\BS{\Phi}\}\subset\CAL{O}_{n,\BS{b}}
$$
is a join of an affine coordinate system 
$\BS{x}:=(x_1,\dots,x_n)$ 
and a Noetherian chain
$\BS{\Phi}:=\{\Phi_1,\dots,\Phi_m\}$ of order $m$, 
which means that 
$$
\FRAC{\partial \Phi_i}{\partial x_j}
=P_{ij}(x_1,\dots,x_n,\Phi_1,\dots,\Phi_m)
\quad (i=1,\dots,m;\ j=1,\dots,n)
$$
for some polynomials $P_{ij}$. 
This $\BS{\Psi}$ is the set of the 
mapping components of the embedding onto the graph 
$X\subset\BB{C}^{m+n}$ 
of the Noetherian chain $\{\Phi_1,\dots,\Phi_m\}$. Then we have 
$$
\alpha(\BS{\Psi})\le 2(m+n)
$$ 
(cf. \cite[Corollary 12]{REFzero-estimate}). 

Our Corollary \ref{CORtame} gives a slightly stronger estimate 
$$
\alpha(\BS{\Psi})\le \dim\overline{X}_{\BS{a}}\le m+n
$$ 
for (only for) a complement of a small subset of $X$ without 
the Noetherian condition. 
In view of Remark \ref{REMueda}, 
exclusion of some point set is inevitable 
for our general analytic case. 
\end{rem}
%%%%%%%%%%%%%%%%%%%%%%%%%%%
\begin{prop}\label{EXAaffine-eq}
Let $X_{\BS{a}}\subset\BB{C}^m$ and 
$X'_{\BS{a}'}\subset\BB{C}^m$ be affine 
equivalent germs of embedded manifolds, i.e. 
there exists an affine 
transformation $\Theta:\,\BB{C}^m\longrightarrow\BB{C}^m$ 
which maps 
$X_{\BS{a}}\subset\BB{C}^m$ to $X'_{\BS{a}'}\subset\BB{C}^m$ 
biholomorphically. Then $X_{\BS{a}}\subset\BB{C}^m$ 
and $X'_{\BS{a}'}\subset\BB{C}^m$ coincide to have 
the properties of 
bundle point, $D$-invariance, Taylorian and they have the same 
$\theta(d)$ and $\alpha$. 
\end{prop}
%%%%%%%%%%%%%%%%%%%%%%%%%

\begin{proof}
Let $\BS{\Phi}$ be a local parametrisation of $X_{\BS{a}}$. 
Then $\BS{\Phi}':=\Theta\COMP\BS{\Phi}$ is a 
local parametrisation of $X'_{\BS{a}'}$. 
Since $\Theta$ is affine, 
$\BB{C}[\BS{\Phi}]^d=\BB{C}[\BS{\Phi}']^d$ and 
hence $\BB{C}[\BS{\Phi}]^d_{\BS{b}}\DA
=\BB{C}[\BS{\Phi}']^d_{\BS{b}}\DA$. 
This implies every thing. 
\end{proof}
%%%%%%%%%%%%%%%%%%%%%%%%%
It is easy to see that germs of a quadratic curve 
in $\BB{C}^2$ at any pair of points are affine equivalent. 
Of course $\alpha=1$ in this cases because they are algebraic. 
It is interesting that the germs at points on 
the transcendental curve 
$y=\exp x-1$ are also affine equivalent. 
In this case $\theta(d)$ 
and $\alpha$ are already given in Example \ref{EXAexp}. 
%%%%%%%%%%%%%%%%%%%%

\

{\bf\small Acknowledgements}
The author thanks Tohru Morimoto for invaluable suggestions 
and Takashi Aoki for computation in examples 
using a certain computer algebra system. 
The author is also grateful to Yutaka Matsui and 
Yayoi Nakamura for helpful discussions. 
%This study could not be completed without the physical 
%help of Doctor Kotaro Kitani through his skill. 
%%%%%%%%%%%%%%%%%%%%%%%%%%

%%%%%%%%%%%%%%%%%%%%%%%%%%%%%%
%%%%%%%%%%%%%%%%%%%%%%%
\end{document}